\documentclass[reqno, 10pt]{article}

\usepackage{fullpage}

\usepackage[centertags]{amsmath}
\usepackage{amsmath}
\usepackage{amsfonts}
\usepackage{amssymb}
\usepackage{amsthm}
\usepackage{enumerate}
\usepackage{subfig}
\usepackage{hyperref}

\usepackage{algorithm,algorithmic}
\usepackage{dsfont}

\usepackage{newlfont}
\usepackage{color}

\usepackage{authblk}

\usepackage{stmaryrd}
\SetSymbolFont{stmry}{bold}{U}{stmry}{m}{n}

\usepackage{bbm}
\include{amsthm_sc}

\usepackage{stmaryrd}
\usepackage{mathrsfs}

\usepackage{fancyhdr}
\usepackage{graphicx}
\usepackage{fancybox}
\usepackage{setspace}
\usepackage{cleveref}

\newtheorem{thm}{Theorem}

\newtheorem{lem}[thm]{Lemma}

\newtheorem{defn}[thm]{Definition}
\newtheorem{assump}[thm]{Assumption}
\newtheorem{ex}[thm]{Example}
\newtheorem{rem}[thm]{Remark}

\newcommand{\R} {\mathbb{R}}
\newcommand{\C} {\mathbb{C}}
\newcommand{\N} {\mathbb{N}}

\newcommand{\E} {\mathbb{E}}
\newcommand{\T} {\mathbb{T}}
\newcommand{\p} {\mathbb{P}}

\DeclareMathOperator{\diag}{diag}

\DeclareMathOperator{\Tr}{Tr}

\DeclareMathOperator{\err}{\mathrm{err}}
\DeclareMathOperator{\erfc}{\mathrm{erfc}}

\DeclareMathOperator{\im}{\mathrm{Im}}

\newcommand{\caE}{{\mathcal E}}

\newcommand{\caG}{{\mathcal G}}

\newcommand{\caK}{{\mathcal K}}

\newcommand{\caN}{{\mathcal N}}
\newcommand{\caO}{{\mathcal O}}
\newcommand{\caP}{{\mathcal P}}

\newcommand{\caR}{{\mathcal R}}

\newcommand{\caT}{{\mathcal T}}

\newcommand{\bbB}{{\mathbb B}}
\newcommand{\bbC}{{\mathbb C}}

\newcommand{\bbE}{{\mathbb E}}

\newcommand{\bbP}{{\mathbb P}}

\newcommand{\bsa}{{\boldsymbol a}}
\newcommand{\bsb}{{\boldsymbol b}}

\newcommand{\bse}{{\boldsymbol e}}

\newcommand{\bsu}{{\boldsymbol u}}
\newcommand{\bsv}{{\boldsymbol v}}

\newcommand{\bsx}{{\boldsymbol x}}
\newcommand{\bsy}{{\boldsymbol y}}

\newcommand{\bsH}{{\boldsymbol H}}

\newcommand{\bsU}{{\boldsymbol U}}
\newcommand{\bsV}{{\boldsymbol V}}

\newcommand{\bsth}{{\boldsymbol \theta}}

\newcommand{\ttA}{{\texttt A}}
\newcommand{\ttB}{{\texttt B}}
\newcommand{\ttX}{{\texttt X}}

\newcommand{\abs}[1]{\lvert #1 \rvert}
\newcommand{\pbb}[1]{\biggl({#1}\biggr)}
\newcommand{\pb}[1]{\bigl({#1}\bigr)}
\newcommand{\wt}{\widetilde}

\newcommand{\wh}{\widehat}

\newcommand{\beq}{ \begin{equation} }
\newcommand{\eeq}{ \end{equation} }

\newcommand{\dd}{\mathrm{d}}
\newcommand{\ii}{\mathrm{i}}

\renewcommand{\P}{\bbP}

\newcommand{\tM}{\widetilde{M}}
\newcommand{\tY}{\widetilde{Y}}

\newcommand{\fh}{F_g}
\newcommand{\gh}{G^H}
\newcommand{\ef}{E_q}
\newcommand{\mf}{M_q}
\newcommand{\vf}{V_q}
\newcommand{\SNR}{\omega}
\newcommand{\rat}{d_0}
\newcommand{\fhd}{F_{g,d}}
\newcommand{\ghd}{G_{g,d}}

\makeatletter
\def\blfootnote{\xdef\@thefnmark{}\@footnotetext}
\makeatother

\numberwithin{equation}{section} 
\numberwithin{thm}{section}

\title{Detection problems in the spiked matrix models}

\author{Ji Hyung Jung\footnote{Department of Mathematical Sciences, KAIST, Daejeon, 34141, Korea
		\newline email: \texttt{jhjung66@kaist.ac.kr}}
	, Hye Won Chung\footnote{School of Electrical Engineering, KAIST, Daejeon, 34141, Korea
		\newline email: \texttt{hwchung@kaist.ac.kr}}
	, and Ji Oon Lee\footnote{Department of Mathematical Sciences, KAIST, Daejeon, 34141, Korea
		\newline email: \texttt{jioon.lee@kaist.edu}}}

\date{\today}

\begin{document}
	
	\maketitle
	
\begin{abstract}%
		We study the statistical decision process of detecting the low-rank signal from various signal-plus-noise type data matrices, known as the spiked random matrix models. We first show that the principal component analysis can be improved by entrywise pre-transforming the data matrix if the noise is non-Gaussian, generalizing the known results for the spiked random matrix models with rank-$1$ signals. As an intermediate step, we find out sharp phase transition thresholds for the extreme eigenvalues of spiked random matrices, which generalize the Baik-Ben Arous-P\'{e}ch\'{e} (BBP) transition. We also prove the central limit theorem for the linear spectral statistics for the spiked random matrices and propose a hypothesis test based on it, which does not depend on the distribution of the signal or the noise. When the noise is non-Gaussian noise, the test can be improved with an entrywise transformation to the data matrix with additive noise. We also introduce an algorithm that estimates the rank of the signal when it is not known a priori.
	
\end{abstract}


\section{Introduction}\label{sec:intro}
One of the most natural approach to `signal-plus-noise' type data is to consider spiked random matrices, which are the low-rank deformation of large random matrices. Most notable examples of spiked random matrices include spiked Wigner matrix and spiked Wishart matrix, where the signals are given as a low-rank mean matrix (spiked Wigner matrix) and a low-rank perturbation of the identity in its covariance matrix (spiked Wishart matrix). In this paper, we focus on the following three types of noisy data matrices, known as spiked random matrices, which generalize spiked Wigner/Wishart matrices:
\begin{itemize}
	\item Spiked Wigner matrix: the data matrix is of the form
	\beq \label{eq:spiked_Wigner}
	\bsU\Lambda^{1/2}\bsU^T + W,
	\eeq
	where $\bsU = [\bsu(1), \bsu(2), \dots, \bsu(k)] \in \R^{N\times k}$ with $\bsU^T\bsU=I_k$, and $W$ is an $N \times N$ Wigner matrix. The signal-to-noise ratio (SNR) $\Lambda = \diag(\lambda_1, \lambda_2, \dots, \lambda_k)$ with $\lambda_1 \geq \lambda_2 \geq \dots \lambda_k > 0$ for some positive integer $k$, independent of $N$.
	\item Rectangular matrix \emph{with spiked mean} (additive model): the data matrix is of the form
	\beq \label{eq:spiked_rectangular_additive}
	\bsU \Lambda^{1/2}\bsV^T + X,
	\eeq
	where $\bsU = [\bsu(1), \bsu(2), \dots, \bsu(k)] \in \R^{M\times k}$, $\bsV = [\bsv(1), \bsv(2), \dots, \bsv(k)] \in \R^{N\times k}$ with $\bsU^T\bsU=\bsV^T\bsV = I_k$, and $X$ is an $M \times N$ random i.i.d. matrix whose entries are centered with variance $N^{-1}$. The SNR $\Lambda$ is given as in the spiked Wigner matrix.
	
	\item Rectangular matrix \emph{with spiked covariance} (multiplicative model): the data matrix is of the form
	\beq\label{eq:spiked_rectangular_multiplicative}
	(I+  \bsU \Lambda\bsU^T)^{1/2} X,
	\eeq
	where $\bsU = [\bsu(1), \bsu(2), \dots, \bsu(k)]$ with $\bsU^T\bsU=I_k$ and $X$ is an $M \times N$ random i.i.d. matrix whose entries are centered with variance $N^{-1}$. The SNR $\Lambda$ is given as in the spiked Wigner matrix.
\end{itemize}
Here, $I_k$ is the identity matrix with rank $k$ and we allow the case $k=0$ where no signal is present. Throughout the paper, for the ease of notation, we denote by $W$ an $N \times N$ Wigner matrix, and $X$ an $M \times N$ random i.i.d. matrix.

To describe the detection problems we consider in this paper, we first review the known results for the simplest case of the spiked random matrix models with rank-$1$ spike, i.e, $k=1$ in \eqref{eq:spiked_Wigner}, \eqref{eq:spiked_rectangular_additive}, and \eqref{eq:spiked_rectangular_multiplicative}.

\textbf{Signal detection problem in rank-$1$ spiked random matrices:} 
Many problems concerning the signal detection can be answered in the case with Gaussian noise and rank-$1$ spike. In this case, the spikes $\bsU=\bsu$ and $\bsV=\bsv$ are vectors, and SNR $\lambda_1=\lambda$, hence the spiked random matrices are of the following forms:
\beq \label{eq:rank-1}
\sqrt{\lambda} \bsu \bsu^T + W
\eeq
\beq \label{eq:rank-1-additive}
\sqrt{\lambda} \bsu \bsv^T + X
\eeq
\beq \label{eq:rank-1-multiplicative}
(I+\lambda \bsu \bsu^T)^{1/2}X.
\eeq
For this case, reliable detection of the signal, i.e., detection with probability $1 - o(1)$ as $M, N\to\infty$, is impossible if the SNR $\lambda$ is below a certain threshold \cite{Montanari2017,Onatski2013}. The threshold is $1$ as $N \to \infty$ for spiked Wigner matrices; for spiked rectangular matrices, with additional assumption $M/N \to \rat$ as $N \to \infty$, the threshold is $\sqrt{\rat}$ for a general class of priors \cite{Perry2018}. On the other hand, the signal can be reliably detected by the principal component analysis (PCA) if the SNR is above the threshold in which case the signal can actually be estimated \cite{Barbier2016,LelargeMiolane,miolane2017fundamental}.

In the subcritical case where the signal is not reliably detectable, it is natural to consider a hypothesis test on the presence of the signal between $\bsH_0: \lambda = 0$ and $\bsH_1: \lambda = \SNR$, commonly referred to as the weak detection, which is also known as the sphericity test in the case the spike is drawn from the uniform distribution on the unit sphere, known as the spherical prior.
As asserted by Neyman--Pearson lemma, the likelihood ratio (LR) test is optimal in the sense that it minimizes the sum of the Type-I error and the Type-II error. It was proved for several distributions of the spikes, called priors, that this sum for a spiked Wigner matrix converges to
\beq
\erfc \left( \frac{1}{4} \sqrt{-\log\left(1-\lambda\right)} \right)
\eeq
when $H$ is Gaussian Orthogonal Ensemble (GOE), and for a spiked Wishart matrix
\beq
\erfc \left( \frac{1}{4} \sqrt{-\log\left(1-\frac{\lambda^2}{\rat}\right)} \right)
\eeq
when $XX^T$ is a Wishart Ensemble; see, e.g., \cite{Onatski2013,AlaouiJordan2018,el2018detection}. Here, $\erfc(\cdot)$ is the complementary error function defined as
\beq\label{def:erfc}
\erfc(x)=\int_{x}^{\infty}e^{-t^2}\dd t.
\eeq 

Though optimal, the LR test is not efficient, and it is desirable to construct a test that does not depend on information about the prior, which is typically not known in many practical applications. In \cite{chung2019weak}, an optimal and universal test for spiked Wigner matrices was proposed, which is based on the linear spectral statistics (LSS) of the data matrix, a linear functional defined as
\beq\label{eq:LSS1}
L_N(f)=\sum_{i=1}^{N}f(\mu_i)
\eeq
for a given function $f$, where $\mu_1,\cdots\mu_N$ are the eigenvalues of the data matrix. The test is extended to spiked rectangular matrices in \cite{jung2021}, where the singular values of the data matrix is used instead of the eigenvalues.

If the noise is non-Gaussian, it is possible to improve the PCA by transforming the data matrix entrywise for spiked Wigner matrices \cite{lesieur2015mmse,Perry2018} and for spiked rectangular matrices \cite{jung2021}. In this improved PCA, the threshold is lowered by a certain factor that depends on the Fisher information of the noise distribution. Below this threshold, the LSS-based test proposed in \cite{chung2019weak} for spiked Wigner matrices can also be improved by applying the entrywise transformation for the improved PCA. It is not known whether the reliable detection is impossible below the threshold except for the case of the spiked Wigner matrix with Rademacher prior \cite{chung2022asymptotic}.

\textbf{Spiked random matrices with general rank:} The more relevant structure for application is that the latent signal contains multiple spikes, or a spike with a higher rank. For such models of spiked random matrices, similar to the cases with rank-$1$ spikes, it is natural to ask the following questions:
\begin{itemize}
	\item What is the spectral threshold for a reliable detection lower than the existing one for Gaussian noise if the noise is non-Gaussian?
	\item Can we design an efficient algorithm to weakly detect the presence of signal (i.e., better than a random guess) when a reliable detection is not feasible?
\end{itemize}

Contrary to the rank-$1$ spike case, the questions addressed above have never been answered, even for the simplest case of Gaussian noise. Furthermore, for the spikes with general rank, we need to consider another important problem of finding the rank of the spike in case it is not known a priori. While viable solutions to resolve the issue in the context of the community detection were suggested in \cite{lei2016goodness,bickel2016hypothesis} for any spiked Wigner matrices and \cite{passemier2012determining,ding2020high} for spiked rectangular matrices, these methods are not applicable in the sub-critical case. To the best of our knowledge, there are no spectral algorithms for estimating the rank of signal in the sub-critical regime. We thus aim to the following question as well:
\begin{itemize}
	\item Can we design an efficient algorithm to estimate the rank of signal when a reliable detection is not feasible?
\end{itemize}


\subsection*{Main contributions}\label{subsec:main}

Our main contributions are mainly divided into three parts as follows:
\begin{itemize}
	\item (Strong detection) We prove that the PCA can be improved by an entrywise transformation if the noise is non-Gaussian, under a mild assumption on the distribution (prior) of the spike.
	
	\item (Weak detection I) We propose a universal test to detect the presence of signal with low computational complexity, based on the linear spectral statistics (LSS). The test does not require any prior information on the signal, and if the noise is Gaussian the error of the proposed test is optimal. For the spiked Wigner matrix or the additive model of the spiked rectangular matrix with the non-Gaussian noise, we suggest an improved test via an entrywise transformation. 
	
	\item (Weak detection II) We present an LSS-based test for estimating the rank of a signal when $\Lambda=\lambda I$.	
\end{itemize}

Heuristically, it is possible to increase the SNR via an entrywise transformation. Here, we illustrate the main idea of the entrywise transformation for the spiked Wigner matrix of the form $M= \bsU\Lambda^{1/2}\bsU^T+W$. If $|\bsu_i \bsu_j^T|,\,|\bsu_i \bsv_j^T| \ll W_{ij}$, then by applying a function $q$ entrywise to $\sqrt{N} Y$, we obtain a transformed matrix whose entries are
\[ \begin{split}
q(\sqrt{N} M_{ij} ) =q(\sqrt{N} W_{ij} + \sqrt{N} \bsu_i \Lambda^{1/2}\bsu_j^T) \approx q(\sqrt{N} W_{ij}) + \sqrt{N} q'(\sqrt{N} W_{ij}) \bsu_i \Lambda^{1/2}\bsu_j^T ,
\end{split} \]
where $\bsu_i$ and $\bsv_i$ denote $i-$th row vector of the signal matrix $\bsU$ and $\bsV$, respectively.
With negligible error, it is possible to approximate the coefficient $q'(\sqrt{N} W_{ij})$ in the second term in the right side by its expectation. (See Appendix \ref{sec:proof_trans-mean} for the proof) Then,
\[ \begin{split}
q(\sqrt{N} M_{ij} ) =q(\sqrt{N} W_{ij} + \sqrt{N} \bsu_i \Lambda^{1/2}\bsu_j^T) \approx \sqrt{N} \left( \frac{q(\sqrt{N} W_{ij})}{\sqrt{N}} +  \E[q'(\sqrt{N} W_{ij})] \bsu_i \Lambda^{1/2}\bsu_j^T \right)
\end{split} \]
and the transformed matrix is approximately of the form $\bsU (\Lambda')^{1/2}\bsU^T + Q$ after a proper normalization, which becomes another spiked Wigner matrix with different SNR. By optimizing the transformation $q$, we find that the SNR is effectively increased (or equivalently, the threshold $\sqrt{\rat}$ is lowered) in the PCA for the transformed matrix. The change of the threshold and a BBP-type transition for the largest eigenvalues of the transformed matrix can be rigorously proved; see Theorem \ref{thm:trans-wigner} for a precise statement. We remark that the same idea works even if the SNR $\Lambda$ is not a constant multiple of an identity matrix, and also a similar result holds for the additive model of spiked rectangular matrix (Theorem \ref{thm:trans-mean}).

For the multiplicative model of the form $Y=(I+ \bsU \Lambda\bsU^T)^{1/2} X =: (I+\bsU \Gamma\bsU^T) X$, with $\Lambda = 2\Gamma + \Gamma^2$, the analysis is significantly more involved due to the following reason: Applying a function $q$ entrywise to $\sqrt{N} Y$, we find that
\[ \begin{split}
q(\sqrt{N} Y_{ij} ) &= q \Big(\sqrt{N} X_{ij} + \sqrt{ N} \sum_\ell \bsu_i \Gamma\bsu_\ell^T X_{\ell j} \Big) 
\approx q(\sqrt{N} X_{ij}) +  \sqrt{N} q'(\sqrt{N} X_{ij}) \sum_\ell \bsu_i \Gamma\bsu_\ell^T X_{\ell j} \\
&\approx \sqrt{N} \left( \frac{q(\sqrt{N} X_{ij})}{\sqrt{N}} +  \E[q'(\sqrt{N} X_{ij})] \sum_\ell \bsu_i \Gamma\bsu_\ell^T X_{\ell j} \right),
\end{split} \]
and the transformed matrix is of the form $ \bsU \Gamma'\bsU^T X + Q$, which is not a spiked rectangular matrix anymore. Note that $Q$ depends on $X$ entrywise and thus it cannot be considered as an additive model, either. 

In Theorem \ref{thm:trans-cov}, we prove the effective change of the SNR and the BBP-type transition for the multiplicative model. The proof of Theorem \ref{thm:trans-cov} is based on a generalized version of the BBP transition that works with the matrix of the form $\bsU\Gamma \bsU^T X +Q$. We remark that the strategy for the proof, based on recent development of random matrix theory, can also be applied to prove a BBP-type transition for other models. 

As in the rank-$1$ case in \cite{jung2021}, it is notable that the optimal entrywise transform for the multiplicative model is different from the one for the additive model. For the spiked Wigner matrix, the optimal transforms are given by $-g'/g$ for the off-diagonal entries (and $-g_d'/g_d$ for the diagonal entries), where $g$ (and $g_d$ for the diagonal entries) is the density functions of them; the optimal transform for the additive model is also given by $-g'/g$. However, for the multiplicative model, the optimal transform is a linear combination of the function $-g'/g$ and the identity mapping. Heuristically, it is due to that the effective SNRs depend not only on $\Gamma'$ but also on the correlation between $X$ and $Q$; the former is maximized when the transform is $-g'/g$ while the latter is maximized when the transform is the identity mapping. We also remark that the effective SNRs after the optimal entrywise transform is larger in the additive model, which suggests that the detection problem is fundamentally harder for the multiplicative model.

With the BBP-type transition for the largest eigenvalues of the transformed matrices, it is also possible to improve the performance of several statistical inferences \cite{bao2022statistical,ke2021estimation,onatski2009rank}. One of the consequences is that the corresponding eigenspace is adjacent to its true spike $\bsU$ in the sense of direction of arrival (DoA) \cite{COUILLET2015}. In other words, we can not only reliably estimate the number of spikes by parallel analysis (PA) \cite{dobriban2020permutation}, but also approximately recover the true spikes and the corresponding SNRs.

For the subcritical case where it is impossible to reliably detect the signal by the improved PCA, we propose algorithms for weak detection, based on the central limit theorem (CLT) of the LSS, Theorems \ref{thm:CLT}, \ref{thm:trans_CLT}, \ref{thm:CLT_rec}, and \ref{thm:trans_CLT_rec}, analogous to the ones introduced in \cite{chung2019weak}. More precisely, assuming the SNRs are uniform i.e., $\Lambda=\lambda I$, we propose an algorithm for a hypothesis test between 
\beq \label{eq:hyp}
\bsH_{k_1} : k=k_1, \qquad \bsH_{k_2} : k=k_2
\eeq
for non-negative integers $k_1<k_2$. While it may seem obvious, it has not been even known in the simple case $k_1=0$ whether the detection becomes easier as $k_2$ increases. Our test in Algorithm \ref{alg:ht} verifies the claim since the error of the proposed test is an increasing function of $(k_2 - k_1)$ as in Theorem \ref{thm:test}. As in \cite{chung2019weak}, the proposed tests are universal, and the various quantities in it can be estimated from the observed data. The test can further be improved by applying the same entrywise transformation we used for the PCA (Algorithm \ref{alg:htet}) if the data matrix is of additive type (spiked Wigner matrix or rectangular matrix with spiked mean), and it also can be adapted to the rank detection problem where we need to estimate the rank $k$ of the signal without knowing the candidates $k_1$ and $k_2$ a priori (Algorithm \ref{alg:at}). 

The main mathematical achievement of the second part is the CLT for the LSS of spiked random matrices with general ranks. For a rank-$1$ spiked Wigner matrix, the CLT was first proved for a special spike $\frac{1}{\sqrt{N}}(1, 1, \dots, 1)^T$ in \cite{Baik-Lee2017} and later extended for a general rank-$1$ spike by comparison with the special case \cite{chung2019weak}. However, the proof in \cite{Baik-Lee2017} is not readily extended to the spiked Wigner matrices with higher ranks and the spiked rectangular matrices. In this paper, we overcome the difficulty by introducing a direct interpolation between the spiked random matrix and the corresponding pure noise matrix and tracking the change of the LSS. Furthermore, we will prove that the proposed entrywise transformation for the data matrix of additive type also effectively changes the SNR, and that the LSS of the transformed matrix is also asymptotically Gaussian; this result was proved previously only for rank-$1$ spiked Wigner matrices in \cite{chung2019weak}. Thus, the error from the proposed test decreases after the transformation as for spiked Wigner matrices in \cite{chung2019weak}.

\subsection*{Related works}

Spiked random matrix models were first introduced by Johnstone \cite{Johnstone2001}. The model can be applied to various problems such as community detection \cite{Abbe2017} and submatrix localization \cite{Butucea2013}. The transition of the largest eigenvalue was proved by Baik, Ben Arous, and P\'ech\'e \cite{BBP2005} for spiked complex Wishart matrices and generalized by Benaych-Georges and Nadakuditi \cite{Raj2011,benaych2012singular}. For more results from random matrix theory about the extreme eigenvalues and the corresponding eigenvectors of spiked random matrices, we refer to \cite{BKYY2016} and references therein.

The improved PCA based on the entrywise transformation was considered for rank-$1$ spiked Wigner matrices in \cite{lesieur2015mmse,Perry2018}, where the transformation is chosen to maximize the effective SNR of the transformed matrix. Detection problems for rank-$1$ spiked Wigner matrices were also considered, where the analysis is typically easier due to its symmetry and canonical connection with spin glass models. For more results on the rank-$1$ spiked Wigner matrices, we refer to \cite{Montanari2017,Perry2018,AlaouiJordan2018,chung2019weak} and references therein.

The testing problem for rank-$1$ spiked Wishart matrices with the spherical prior was considered by Onatski, Moreira, and Hallin \cite{Onatski2013,Onatski2014}, where they proved the optimal error of the hypothesis test. It is later extended to the case where the entries of the spikes are i.i.d. with bounded support (i.i.d. prior) by El Alaoui and Jordan \cite{el2018detection}. See also \cite{Onatski2015,Montanari2017,Barbier2016,LelargeMiolane,miolane2017fundamental,BanerjeeMa2017} for more about detection limits in statistical learning theory.

Models with sparse or generative structure of the spike have extensively studied in the past literature. Various statistical and algorithmic methods are applicable to the case where SNR is smaller than the spectral threshold. In particular, it can be seen that the sparsity of the spikes and the dimension of the latent vector constituting the generative spike prior actually serve to lower the threshold for the SNR to which several algorithms are applicable; see \cite{aubin2019spiked,cai2015optimal} and references therein.

\subsection*{Organization of the paper}

The rest of the paper is organized as follows:
In Section \ref{sec:prelim}, we introduce the precise definitions of models and relevant previous consequences. In Section~\ref{sec:main1}, we state our results on the improved PCA. In Section \ref{sec:main2}, we propose algorithms for LSS-based tests and a test for rank estimation, and analyze their performance.  In Section~\ref{sec:LSS}, we state general results on the CLT for the LSS. We conclude the paper in Section~\ref{sec:summary} with the summary of our works and future research directions. In Appendix~\ref{sec:ex}, we consider examples of spiked random matrices and provide results from numerical experiments. In Appendices \ref{sec:proof_trans} and \ref{app:CLT}, we provide technical details of the proofs.

\section{Preliminaries} \label{sec:prelim}

In this section, we introduce the precise definition of the models and previous results for the spiked random matrices.

\subsection{Definitions of models}

The noise matrices are defined as follows:

\begin{defn}[Wigner matrix] \label{def:Wigner}
	An $N \times N$ symmetric matrix $W = (W_{ij})$ is a (real) Wigner matrix if $W_{ij}$ ($i, j = 1, 2, \dots, N$) are independent real random variables such that
	\begin{itemize}
		\item For all $i<j$, $N \E[W_{ij}^2]=1$, $N^{\frac{3}{2}} \E[W_{ij}^3]=w_3$, and $N^2 \E[W_{ij}^4]=w_4$ for some $w_3, w_4\in \R$. 
		\item For all $i$, $N\E[W_{ii}^2]=w_2$ for some constant $w_2\geq 0$. 
		\item For any positive integer $p$, there exists $C_p$, independent of $N$, such that $N^{\frac{p}{2}} \E[W_{ij}^p] \leq C_p$ for all $i \leq j$.
	\end{itemize}
\end{defn}
\begin{defn}[Random rectangular matrix] \label{defn:rect}
	An $M \times N$ matrix $X = (X_{ij})$ is a (real) random rectangular matrix if $X_{ij}$ ($1\leq i \leq M$, $1 \leq j\leq N$) are independent real random variables such that
	\begin{itemize}
		\item For all $i, j$, $\E[X_{ij}] = 0$, $N \E[X_{ij}^2]=1$, $N^{\frac{3}{2}} \E[X_{ij}^3]=w_3$, and $N^2 \E[X_{ij}^4]=w_4$ for some constants $w_3, w_4$. 
		\item For any positive integer $p$, there exists $C_p$, independent of $N$, such that $N^{\frac{p}{2}} \E[X_{ij}^p] \leq C_p$ for all $i, j$. 
	\end{itemize}
\end{defn}

The spiked random matrices are defined as follows:

\begin{defn}[Spiked Wigner matrix] \label{def:spiked_Wigner}
	An $N \times N$ matrix $M=\bsU\Lambda^{1/2}\bsU^T + W$ is a spiked Wigner matrix with the SNR (matrix) $\Lambda$ if $W$ is a Wigner matrix and the spike $\bsU = [\bsu(1), \bsu(2), \dots, \bsu(k)] \in \R^{N\times k}$ with $\bsU^T\bsU=I_k$.
\end{defn}

\begin{defn}[Spiked rectangular matrix - additive model] \label{defn:rect_mean}
	An $M \times N$ random matrix $Y = \bsU \Lambda^{1/2}\bsV^T + X$ is a rectangular matrix with spiked mean $\bsU$, $\bsV$ and the SNR (matrix) $\Lambda$ if $X$ is a random rectangular matrix and the spikes $\bsU = [\bsu(1), \bsu(2), \dots, \bsu(k)] \in \R^{M\times k}$, $\bsV = [\bsv(1), \bsv(2), \dots, \bsv(k)] \in \R^{N\times k}$ with $\bsU^T\bsU=\bsV^T\bsV = I_k$.
\end{defn}

\begin{defn}[Spiked rectangular matrix - multiplicative model] \label{defn:rect_cov}
	An $M \times N$ random matrix $Y = (I+ \bsU\Lambda \bsU^T)^{1/2} X$ is a rectangular matrix with spiked covariance $\bsU$ and the SNR (matrix) $\Lambda$ if $X$ is a rectangular matrix and $\bsU = [\bsu(1), \bsu(2), \dots, \bsu(k)] \in \R^{M\times k}$ with $\bsU^T\bsU= I_k$.
\end{defn}
We assume throughout the paper that the SNR matrix $\Lambda$ is a $k\times k$ diagonal matrices with $\Lambda_{ii}=\lambda_i$ and $\lambda_1\geq\lambda_2\geq\ldots\lambda_k\geq0,$ and $\frac{M}{N} \to \rat \in (0, \infty)$ as $M, N \to \infty$.

\subsection{Principal component analysis} \label{subsec:BBP}
Here are the results for principal components of spiked models in the context of random matrix theory.
\subsubsection*{Spiked Wigner matrix}
Let $M$ be the spiked Wigner matrix. The empirical spectral measure of $M$ converges to the Wigner's semicircle law $\mu_{sc}$, i.e., if we denote by $\mu_1 \geq \mu_2 \geq \dots \geq\mu_N$ the eigenvalues of $M$, then
\beq \label{eq:SC_law}
\frac{1}{N} \sum_{i=1}^N \delta_{\mu_i}(x) \dd x \to \dd \mu_{sc}(x)
\eeq
weakly in probability as $N \to \infty$, where 
\beq
\dd \mu_{sc}(x) = \frac{\sqrt{4-x^2}}{2\pi} \mathbf{1}_{(-2, 2)}(x) \dd x.
\eeq
The $k$ largest eigenvalue has the following (almost sure) limit: for $1\leq i\leq k$
\begin{itemize}
	\item If $\lambda_i > 1$, then $\mu_i \to \sqrt{\lambda_i}+\frac1{\sqrt{\lambda_i}}$.
	\item If $\lambda_i< 1$, then $\mu_i \to  2$.
\end{itemize}

\subsubsection*{Sample covariance matrix}
Let $S = YY^T$ be the sample covariance matrix (Gram matrix) derived from a spiked rectangular matrix $Y$. The empirical spectral measure of $S$ converges to the Marchenko--Pastur law $\mu_{MP}$, i.e., if we denote by $\mu_1 \geq \mu_2 \geq \dots\geq \mu_M$ the eigenvalues of $S$, then
\beq \label{eq:MP_law}
\frac{1}{M} \sum_{i=1}^M \delta_{\mu_i}(x) \dd x \to \dd \mu_{MP}(x)
\eeq
weakly in probability as $M, N \to \infty$, where for $M \leq N$
\beq
\dd \mu_{MP}(x) = \frac{\sqrt{(x- d_-)(d_+ -x)}}{2\pi\rat x} \mathbf{1}_{(d_-, d_+)}(x) \dd x,
\eeq
with $d_{\pm} = (1\pm \sqrt{\rat})^2$.
The $k$ largest eigenvalue has the following (almost sure) limit: for $1\leq i\leq k$
\begin{itemize}
	\item If $\lambda_i > \sqrt{\rat}$, then $\mu_i \to (1+\lambda_i)(1+\frac{\rat}{\lambda_i})$.
	\item If $\lambda_i < \sqrt{\rat}$, then $\mu_i \to d_+ = (1+\sqrt{\rat})^2$.
\end{itemize}
This in particular shows that the detection can be reliably done by PCA if $\lambda > \sqrt{\rat}$. We remark that the results above hold for both the additive model and the multiplicative model.

\subsection{Linear spectral statistics}
We introduce the central limit theorems for null models.
\subsubsection*{Spiked Wigner matrix}
The proof of the Gaussian convergence of the LR in \cite{Baik-Lee2016,Baik-Lee2018} is based on the recent study of linear spectral statistics, defined as
\beq \label{eq:LSS0}
L_Y(f) = \sum_{i=1}^N f(\mu_i)
\eeq
for a function $f$, where $\mu_1 \geq \mu_2 \geq \dots \mu_N$ are the eigenvalues of $M$. As the Wigner's semicircle law in \eqref{eq:SC_law} suggests, it is required to consider the fluctuation of the LSS about
\[
N \int_{-2}^{2} f(x) \, \dd \mu_{sc}(x).
\]
The CLT for the LSS is the statement
\beq \begin{split} \label{eq:CLT_LSS_Wig}
	&\left( L_M(f) - N \int_{-2}^{2} f(x) \, \dd \mu_{sc}(x) \right) \Rightarrow \caN(m_M(f), V_M(f)),
\end{split} \eeq
where the right-hand side is the Gaussian random variable with the mean $m_M(f)$ and the variance $V_M(f)$. The CLT was proved for the null case ($\lambda=0$). We will show that the CLT also holds under the alternative and the mean $m_M(f)$ depends on $\lambda$ while the variance $V_M(f)$ does not.

\subsubsection*{Spiked rectangular matrices}
The LSS for the spiked rectangular matrices defined as
\beq \label{eq:LSS}
L_Y(f) = \sum_{i=1}^M f(\mu_i)
\eeq
for a function $f$, where $\mu_1 \geq \mu_2 \geq \dots \mu_M$ are the eigenvalues of $S=YY^T$. As the Marchenko--Pastur law in \eqref{eq:MP_law} suggests, it is required to consider the fluctuation of the LSS about
\[
M \int_{d_-}^{d_+} f(x) \, \dd \mu_{MP}(x).
\]
The CLT for the LSS is the statement
\beq \begin{split} \label{eq:CLT_LSS}
	&\left( L_Y(f) - M \int_{d_-}^{d_+} f(x) \, \dd \mu_{MP}(x) \right) \Rightarrow \caN(m_Y(f), V_Y(f)),
\end{split} \eeq
where the right-hand side is the Gaussian random variable with the mean $m_Y(f)$ and the variance $V_Y(f)$. The CLT was proved for the null case ($\lambda=0$). We will show that the CLT also holds under the alternative and the mean $m_Y(f)$ depends on $\lambda$ while the variance $V_Y(f)$ does not.

\section{Main result I - Improved PCA} \label{sec:main1}



In this section, we state our first main results on the improvement of PCA by entrywise transformations and provide the results from numerical experiments.

\subsection{Improved PCA} \label{subsec:PCA}

We introduce the following assumptions for the spike and the noise.
\begin{assump} \label{assump:entry1}
	For the spike $\bsU$ (and also $\bsV$ in the additive model), we assume, for $\phi\leq1/2$, 
	\begin{enumerate}
		\item the spikes are $\phi$-localized with high probability, i.e. $\|\bsU\|_\infty,\|\bsV\|_\infty\prec N^{-\phi}$
		\item the spike matrix is $\phi$-orthonormal with high probability, i.e. $\|\bsU^T\bsU-I_k\|_F,\,\|\bsV^T\bsV-I_k\|_F\prec N^{-\phi},$ and so the spikes are sampled from Stiefel manifold of orthonormal k-frames in $\mathbb{R}^M$ or $\mathbb{R}^N$ with high probability.
	\end{enumerate}
	
	For the noise, let $\caP$ be the distribution of the normalized entries $\sqrt{N} W_{ij}(i\neq j)$ in \ref{def:Wigner} and $\sqrt{N} X_{ij}$ in \ref{defn:rect}. Further, for the spiked Wigner matrices, let $\caP_d$ be the distribution of the normalized diagonal entries $\sqrt{N} W_{ii}$ in \ref{def:Wigner}. We assume the following:
	\begin{enumerate}
		\item The density functions $g$ and $g_d$ of $\caP$ and $\caP_d$, respectively, are smooth, positive everywhere, and symmetric (about 0).
		\item For any fixed ($N$-independent) $D$, the $D$-th moments of $\caP$ and $\caP_d$ are finite.
		\item The functions $h = -g'/g$, $h_d = -g_d'/g_d$ and their all derivatives are polynomially bounded in the sense that $|h^{(\ell)}(w)|,|h_d^{(\ell)}(w)| \leq C_{\ell} |w|^{C_{\ell}}$ for some constant $C_{\ell}$ depending only on $\ell$.
	\end{enumerate}  
\end{assump}

The first condition on the prior implies that the spike is not necessarily delocalized, i.e., some entries of the signal can be significantly larger than $N^{-1/2}$. The key examples of the prior are as follows:
\begin{ex}\label{Rmk:spike}
	We can consider the following examples of the spike prior:
	\begin{enumerate}
		\item the spherical prior, where $\bsu(\ell)$ (and $\bsv(\ell)$) are i.i.d. drawn uniformly from the unit sphere, or
		\item the i.i.d. prior, where the entries $u_1(\ell), \dots, u_M(\ell)$ (respectively, $v_1(\ell), \dots, v_N(\ell)$) are i.i.d. random variables from the probability measures $\mu_\ell$ (respectively, $\nu_\ell$) with mean zero and variance $M^{-1}$ (respectively $N^{-1}$) such that for any integer $p>2$
		\[
		\E |u_i(\ell)|^p, \E |v_j(\ell)|^p \leq \frac{C_p}{M^{1+(p-2)\phi}}
		\]
		for some ($N$-independent) constants $C_p>0$ and $\phi \leq \frac{1}{2}$, uniformly on $i$, $j$ and $\ell$.
	\end{enumerate} 	
	We remark that for the spike Wigner matrices, due to normalization, the variance of the i.i.d. prior $\mu_\ell$ for $u_i(\ell)$ is $N^{-1}$.
\end{ex}
\subsubsection*{Spiked Wigner matrix}

Given a spiked Wigner matrix $M$, we consider a family of the entrywise transformations
\beq \label{eq:alpha_transform}
h_\alpha(x)=-\frac{g'(x)}{g(x)}+\alpha x, \quad h_{d}(x)=-g_d'(x)/g_d(x)
\eeq
for $\alpha \in \mathbb{R}$. We also consider the transformed matrix $\wt M$ whose entries are
\begin{align} \label{eq:transformed-wigner}
\wt M_{ij} = \frac{1}{\sqrt{\fh N}} h_0(\sqrt{N} M_{ij}) (i\neq j),&& \wt M_{ii} = \sqrt{\frac{w_2}{\fhd N}} h_{d}\left(\sqrt{\frac{N}{w_2}} M_{ii}\right),
\end{align}
where the Fisher information $\fh$ and $\fhd$ of $g$ and $g_d$ are given by
\[
\fh = \int_{-\infty}^{\infty} \frac{(g'(x))^2}{g(x)} \dd x,\qquad \fhd = \int_{-\infty}^{\infty} \frac{g_d'(x)^2}{g_d(x)} \dd x.
\]
Note that $\fh \geq 1$ where the equality holds only if $g$ is the standard Gaussian.

Then following theorem asserts that the effective SNRs of the transformed matrix for PCA are $\lambda_\ell \fh$, which generalizes Theorem 4.8 in \cite{Perry2018}.
\begin{thm} \label{thm:trans-wigner}
	Let $M$ be a spiked Wigner matrix in Definition \ref{def:spiked_Wigner} satisfying Assumption \ref{assump:entry1} with $\phi>1/4.$ Let $\wt M$ be the transformed matrix obtained as in \eqref{eq:transformed-wigner} and $(\wt \mu_\ell, \wt\bsu(\ell))$ the pair of $\ell$-th largest eigenvalue and the corresponding eigenvector of $\wt M$. Then, almost surely, for $1\leq \ell\leq k$
	\begin{itemize}
		\item If $\lambda_\ell > \frac{1}{\fh}$, then $\wt\mu_\ell \to \sqrt{\lambda_\ell \fh}+\frac{1}{\sqrt{\lambda_\ell \fh}}$ and $ |\wt\bsu(\ell)^T\bsu(\ell)|^2\to 1-\frac{1}{\lambda_\ell\fh}$,
		\item If $\lambda_\ell < \frac{1}{\fh}$, then $\wt\mu_\ell \to 2$ and $ |\wt\bsu(\ell)^T\bsu(\ell)|^2\to 0$.
	\end{itemize}
\end{thm}


For the proof, we adapt the strategy in \cite{Perry2018}, where the key observation is that the transformed matrix is approximately equal to another spiked Winger matrix. See Appendix \ref{sec:proof_trans-wigner} for the detail of the proof.

We remark that $h_0$ is the optimal (up to constant factor) among all entrywise transformations. See Appendix \ref{subsec:variation_mean} for the proof of it.

\subsubsection*{Spiked rectangular matrices}

For a spiked rectangular matrix $Y$, we consider the family of the entrywise transformations $h_{\alpha}(x)$ defined in \eqref{eq:alpha_transform} and transformed matrices $\wt Y^{(\alpha)}$ whose entries are
\beq \label{eq:transformed-cov}
\wt Y^{(\alpha)}_{ij} = \frac{1}{\sqrt{(\alpha^2+2\alpha+\fh) N}} h_{\alpha}(\sqrt{N} Y_{ij}).
\eeq
Note that 

For the additive model, we again show that the effective SNRs of the transformed matrix for PCA are $\{\lambda_\ell \fh\}_\ell$.
\begin{thm} \label{thm:trans-mean}
	Let $Y$ be a spiked rectangular matrix in Definition \ref{defn:rect_mean} satisfying Assumption \ref{assump:entry1} with $\phi>1/4$. Let $\wt Y \equiv \wt Y^{(0)}$ be the transformed matrix obtained as in \eqref{eq:transformed-cov} with $\alpha=0$ and $(\wt \mu_\ell, \wt\bsu(\ell))$ the pair of $\ell$-th largest eigenvalue and the corresponding eigenvector of $\wt Y \wt Y^T$. Then, almost surely, for $1\leq \ell\leq k$
	\begin{itemize}
		\item If $\lambda_\ell > \frac{\sqrt{\rat}}{\fh}$, then $\wt\mu_\ell \to (1+\lambda_\ell \fh)(1+\frac{\rat}{\lambda_\ell \fh})$ and $ |\wt\bsu(\ell)^T\bsu(\ell)|^2\to 1-\frac{\rat(1+\lambda_\ell\fh)}{\lambda_\ell\fh(\lambda_\ell\fh+\rat)}$.
		\item If $\lambda_\ell < \frac{\sqrt{\rat}}{\fh}$, then $\wt\mu_\ell \to d_+ = (1+\sqrt{\rat})^2$ and $|\wt\bsu(\ell)^T\bsu(\ell)|^2\to0$.
	\end{itemize}
\end{thm}

From Theorem \ref{thm:trans-mean}, if $\lambda_\ell > \frac{\sqrt{\rat}}{\fh}$, we immediately see that the signal in the additive model can be reliably detected by the transformed PCA. Thus, the detection threshold in the PCA is lowered when the noise is non-Gaussian. We also remark that $h_0$ is the optimal entrywise transformation (up to constant factor) as in the Wigner case; see Appendix \ref{subsec:variation_cov}.

For the proof, we adapt the strategy in \cite{jung2021}, where the key observation is again that the transformed matrix is approximately equal to another spiked rectangular matrix. See Appendix \ref{sec:proof_trans-mean} for the detail of the proof.




For the multiplicative model, we have the following result.

\begin{thm} \label{thm:trans-cov}
	Let $Y$ be a spiked rectangular matrix in Definition \ref{defn:rect_cov} satisfying Assumption \ref{assump:entry1} with $\phi>1/4$. Let $\wt Y \equiv \wt Y^{(\alpha_{g,\ell})}$ be the transformed matrix obtained as in \eqref{eq:transformed-cov} with 
	\[
	\alpha_{g,\ell} := \frac{-\gamma_\ell F_g+\sqrt{4F_g+4\gamma_\ell F_g+\gamma_\ell^2F_g^2}}{2(1+\gamma_\ell)}
	\]
	and  $(\wt \mu_\ell, \wt\bsu(\ell))$ the pair of $\ell$-th largest eigenvalue and the corresponding eigenvector of $\wt Y \wt Y^T$. Then, almost surely, 
	\begin{itemize}
		\item If $(\lambda_g)_\ell > \sqrt{\rat}$, then $\wt\mu_\ell \to (1+ (\lambda_g)_\ell)(1+\frac{\rat}{(\lambda_g)_\ell})$ and  \[ |\wt\bsu(\ell)^T\bsu(\ell)|^2\to 1-\frac{(\lambda_g)_\ell+\rat}{(\lambda_g)_\ell\cdot((\lambda_g)_\ell+1)},\]
		\item If $(\lambda_g)_\ell < \sqrt{\rat}$, then $\wt\mu_\ell \to d_+ = (1+\sqrt{\rat})^2$ and $|\wt\bsu(\ell)^T\bsu(\ell)|^2\to0$.
	\end{itemize}
	where
	\[
	(\lambda_g)_\ell := \gamma_\ell + \frac{\gamma_\ell^2 \fh}{2} + \frac{\gamma_\ell \sqrt{4\fh + 4\gamma_\ell \fh + \gamma_\ell^2 \fh^2}}{2}.
	\]
\end{thm}

Note that
\[ \begin{split}
(\lambda_g)_\ell &\geq \gamma_\ell + \frac{\gamma_\ell^2 \fh}{2} + \frac{\gamma_\ell \sqrt{4 + 4\gamma_\ell \fh + \gamma_\ell^2 \fh^2}}{2} = 2\gamma_\ell + \gamma_\ell^2 F_g \geq 2\gamma_\ell + \gamma_\ell^2 = \lambda_\ell,
\end{split} \]
and the inequality is strict if $F_g > 1$, i.e., $g$ is not Gaussian. 

Note that unlike the additive model, we cannot determine $\alpha_g$ without prior knowledge on the SNR. Nevertheless, we can apply the transformation $h_{\sqrt{\fh}}$ or $h_0$, which effectively increase all SNRs simultaneously; see Appendix \ref{sec:optimize_transform}.

From Theorem \ref{thm:trans-cov}, if $(\lambda_g)_\ell > \sqrt{\rat}$, the signal can be reliably detected by the transformed PCA and the detection threshold in the PCA is lowered if the noise is non-Gaussian. We also remark that $h_{\alpha_{g,\ell}}$ is the optimal entrywise transformation (up to constant factor) for the $\ell$-th largest eigenvalue; see Appendix \ref{sec:optimize_transform}.

We finish this section with an outline of the proof of Theorem \ref{thm:trans-cov}.
We begin by justifying that the transformed matrix $\wt Y$ is approximately of the form $(Q +  \bsU \wh\Gamma^{\frac{1}{2}}\bsU^T X)$, where $\wh\Gamma=\diag(\wh\gamma_1,\cdots,\wh\gamma_k)$. Then, the largest eigenvalue of $\wt Y \wt Y^T$ can be approximated by the largest eigenvalue of $(Q + \bsU \wh\Gamma^{\frac{1}{2}}\bsU^T X)^T (Q + \bsU \wh\Gamma^{\frac{1}{2}}\bsU^T X)$ for which we consider an identity
\[ \begin{split}
&(Q + \bsU \wh\Gamma^{\frac{1}{2}}\bsU^T X)^T (Q + \bsU\wh\Gamma^{\frac{1}{2}} \bsU^T X) - zI = (Q^T Q -zI) (I + L(z)),
\end{split} \]
where
\begin{align*}
&L(z) = \caG(z) ( X^T \bsU \wh\Gamma^{\frac{1}{2}}\bsU^T Q +  Q^T \bsU\wh\Gamma^{\frac{1}{2}} \bsU^T X +  X^T \bsU \wh\Gamma\bsU^T X), &&\caG(z) = (Q^T Q -zI)^{-1}.
\end{align*}
If $z$ is an eigenvalue of $(Q +  \bsU\wh\Gamma^{\frac{1}{2}} \bsU^T X)^T (Q + \bsU\wh\Gamma^{\frac{1}{2}} \bsU^T X)$ but not of $Q^T Q$, the determinant of $(I + L(z))$ must be $0$ and hence $-1$ is an eigenvalue of $L(z)$. Since the rank of $L(z)$ is at most $2k$, we can find that the eigenvector of $L(z)$ is a linear combination of vectors $\caG(z) Q^T \bsu(\ell)$ and $\caG(z) X^T \bsu(\ell).$ Further, by using the facts in Example \ref{Rmk:spike}, we can observe that a linear combination of vectors $\caG(z) Q^T \bsu(\ell)$ and $\caG(z) X^T \bsu(\ell)$ be a possible candidate for the $\ell$-th eigenvector of $L(z)$, and so of $\wt Y^T\wt Y$  i.e., for some $a_\ell, b_\ell$,
\beq \begin{split} \label{eq:eigenvalue}
	&L(z) (a_\ell \caG(z) Q^T \bsu(\ell) + b_\ell \caG(z) X^T \bsu(\ell)) = -(a_\ell \caG(z) Q^T \bsu(\ell) + b_\ell \caG(z) X^T \bsu(\ell)).
\end{split} \eeq

From the definition of $L(z)$,
\[ \begin{split}
L(z) \cdot \caG(z) X^T \bsU &=  \caG(z) X^T \bsU \wh\Gamma^{\frac{1}{2}}(\bsU^TQ\caG(z) X^T \bsU)  +  \caG(z) Q^T \bsU\wh\Gamma^{\frac{1}{2}}(\bsU^T X \caG(z) X^T \bsU) \\& \quad +  \caG(z) X^T \bsU\wh\Gamma(\bsU^T X \caG(z) X^T \bsU),
\end{split} \]
and a similar equation holds for $L(z) \cdot \caG(z) Q^T \bsU$. It suggests that if $\bsU^TQ \caG(z) X^T \bsU$ and $\bsU^TX \caG(z) X^T \bsU$ are concentrated around diagonal matrices where the entries are deterministic functions of $z$, then the left side of \eqref{eq:eigenvalue} can be well-approximated by a (deterministic) linear combination of $\caG(z) Q^T \bsu(\ell)$ and $\caG(z) X^T \bsu(\ell)$. We can then find the location of the largest eigenvalue in terms of a deterministic function of $z$ and conclude the proof by optimizing the function $q$.

The concentration of random matrices $\bsU^TQ \caG(z) X^T \bsU$ and $\bsU^TX \caG(z) X^T \bsU$ is the biggest technical challenge in the proof, mainly due to the dependence between the matrices $Q$ and $X$. We prove it by applying the technique of linearization in conjunction with resolvent identities and also several recent results from random matrix theory, most notably the local Marchenko--Pastur law. 

Once we find out the coefficients $a_\ell$ and $b_\ell$ in \eqref{eq:eigenvalue}, the eigenvector localization is an easy corollary since the vector $a_\ell \caG(z) Q^T \bsu(\ell) + b_\ell \caG(z) X^T \bsu(\ell)$ must be a right singular vector of $\wt Y$ with the corresponding singular value $\sqrt{(1+ (\lambda_g)_\ell)(1+\frac{\rat}{(\lambda_g)_\ell})}.$ In this paper, we will not go into further detail on this part.

The detailed proof of Theorem \ref{thm:trans-cov} can be found in Appendix \ref{sec:proof_trans-cov}.


\section{Main Result II - Weak Detection}\label{sec:main2}

\subsection{Signal detection in rank-$1$ spiked models}\label{subsec:ht}
We begin by recalling the LSS-based detection algorithms for rank-$1$ spiked rectangular matrices in \cite{jung2021}. Suppose that our goal is to detect the presence of the signal by the hypothesis test between $\bsH_0: \lambda = 0$ and $\bsH_1: \lambda = \SNR$ where the SNR $\SNR$ for the alternative hypothesis $\bsH_1$ is known. The key observation is that the variances of the limiting Gaussian distributions of the LSS in \eqref{eq:LSS} do not depend on the SNR while the means do. If we denote by $V_Y(f)$ the common variance, and $m_Y(f)|_{\bsH_0}$ and $m_Y(f)|_{\bsH_1}$ the means, respectively, our goal is to find a function that maximizes the relative difference between the limiting distributions of the LSS under $\bsH_0$ and under $\bsH_1$, i.e.,
\beq\label{loss}
\left| \frac{m_Y(f)|_{\bsH_1} - m_Y(f)|_{\bsH_0}}{\sqrt{V_Y(f)}} \right|.
\eeq
As we will see in Theorem \ref{thm:CLT_rec}, the optimal function $f$ is of the form $C_1\phi_{\SNR} +C_2$ for some constants $C_1$ and $C_2$, where 
\beq \begin{split}\label{eq:optimal_f1}
	\phi_{\SNR}(x) &=\frac{\SNR}{\rat}\left(\frac{2}{w_4-1}-1\right)x  -\log\left(\left(1+\frac{\rat}{\SNR} \right)(1+\SNR) - x\right).
\end{split}\eeq
The test statistic we use is thus defined as
\beq\label{eq:L_lambda1}
\begin{split}
	L_{\SNR} &= \sum_{i=1}^M \phi_{\SNR}(\mu_i) - M \int_{d_-}^{d_+} \phi_{\SNR}(x) \,\dd \mu_{MP}(x) \\
	&= -\log \det \left( \left(1+\frac{\rat}{\SNR} \right)(1+\SNR)I - YY^T \right)  + \frac{\SNR}{\rat} \left( \frac{2}{w_4-1} - 1 \right) (\Tr YY^T-M) \\
	&\quad + M \left[\frac{\SNR}{\rat} - \log\left(\frac{\SNR}{\rat}\right) -\frac{1-\rat}{\rat}\log(1+\SNR) \right].
\end{split} \eeq
Theorem 8 in \cite{jung2021} asserts that $L_{\SNR}$ converges to a Gaussian,
\beq
L_{\SNR} \Rightarrow \mathcal{N}(m(\lambda),V_0).
\eeq
Here, the mean of the limiting Gaussian distribution is given by
\beq \begin{split} \label{eq:mean_test1}
	m(\lambda) &= -\frac{1}{2} \log\left(1-\frac{\SNR^2}{\rat}\right)  +\frac{\SNR^2}{2\rat}(w_4-3)  -\log\left(1-\frac{\lambda^2}{\rat}\right) +\frac{\lambda^2}{\rat}\left(\frac{2}{w_4-1}-1\right)
\end{split} \eeq
with $\lambda = 0$ under $\bsH_0$ and $\lambda = \SNR$ under $\bsH_1$,
and the variance
\beq \label{eq:var_test1}
V_0=-2\log\left(1-\frac{\SNR^2}{\rat}\right) +\frac{2\SNR^2}{\rat}\left(\frac2{w_4-1}-1\right).
\eeq
%

Based on the asymptotic normality of $L_{\SNR}$, we can construct a test in which we compute the test statistic $L_{\SNR}$ and compare it with the average of $m(0)$ and $m(\SNR)$, i.e.,
\beq \begin{split} \label{eq:m_lambda1}
	m_{\SNR} &:= \frac{m(0)+m(\SNR)}{2} = -\log\left(1-\frac{\SNR^2}{\rat}\right) +\frac{\SNR^2}{2\rat} \left( \frac{2}{w_4-1} +w_4 -4 \right).
\end{split} \eeq
See Algorithm \ref{alg:test1} for the detail.

\begin{algorithm}[!tb]
	\caption{Hypothesis test for a rank-$1$ spiked rectangular matrix}
	\label{alg:test1}
	\begin{algorithmic}
		\STATE {\bfseries Input:} data $Y_{ij}$, parameters $w_4$, $\SNR$
		\STATE $L_{\SNR} \gets$ test statistic in \eqref{eq:L_lambda1}
		\STATE $m_{\SNR} \gets$ critical value in \eqref{eq:m_lambda1}
		\IF{$L_{\SNR} \leq m_{\SNR}$}
		\STATE Accept $\bsH_0$
		\ELSE
		\STATE Reject $\bsH_0$
		\ENDIF
	\end{algorithmic}
\end{algorithm}

The limiting error of the proposed test, Algorithm \ref{alg:test1}, is given by 
\beq \begin{split} \label{eq:test_error1}
	\err(\SNR) &= \p( L_{\SNR} > m_{\SNR} | \bsH_0) + \p( L_{\SNR} \leq m_{\SNR} | \bsH_1) \to \erfc \left( \frac{\sqrt{V_0}}{4\sqrt 2} \right),
\end{split} \eeq
where $V_0$ is the variance in \eqref{eq:var_test1} and $\erfc(\cdot)$ is the complementary error function. If the noise $X$ is Gaussian, $w_4=3$ and the limiting error in \eqref{eq:test_error1} is
\[
\erfc \left( \frac{\sqrt{V_0}}{4\sqrt 2} \right) = \erfc \left( \frac{1}{4} \sqrt{-\log \left( 1- \frac{\SNR^2}{\rat} \right)} \right),
\]
and it coincides with the error of the LR test; see Section 2.2 of \cite{jung2021}. It shows that our test is optimal with the Gaussian noise.

\subsection{Signal detection in rank-$k$ spiked models}

When the rank of the spike is larger than $1$, we first consider a simple case where the data is given as a spiked Wigner matrix and our goal is to construct an LSS-based algorithm for a hypothesis test between $\bsH_0: \Lambda=0$ and $\bsH_k: \Lambda=\SNR I_k$, where the rank $k$ of the spike for the alternative hypothesis is known. Our starting point is the following test statistic, which was considered for the rank-$1$ spiked Wigner matrix in \cite{chung2019weak}:
\beq \begin{split} \label{eq:L_lambda}
	L_{\SNR} &= - \log \det \left( (1+\SNR) I - \sqrt{\SNR} M \right) +  \frac{\SNR N}{2} \\
	&\qquad + \sqrt{\SNR} \left( \frac{2}{w_2} - 1 \right) \Tr M 
	+ \SNR \left( \frac{1}{w_4-1} - \frac{1}{2} \right) (\Tr M^2 - N).
\end{split} 
\eeq
If there is no signal present, 
$L_{\SNR} \Rightarrow \caN(m_0, V_0)$,
where
\beq \begin{split} \label{eq:m_0}
	m_0 = -\frac{1}{2} \log(1-\SNR) + \left(\frac{w_2 -1}{w_4-1} -\frac{1}{2} \right) \SNR + \frac{(w_4 -3) \SNR^2}{4},
\end{split} \eeq
\beq \begin{split} \label{eq:V_H}
	V_0 = -2 \log(1-\SNR) + \left( \frac{4}{w_2}-2 \right) \SNR + \left( \frac{2}{w_4-1} - 1 \right) \SNR^2.
\end{split} \eeq

For a rank-$k$ spiked Wigner matrix, we can consider the same $L_{\SNR}$ as in \eqref{eq:L_lambda} and prove that it also converges to a Gaussian with the same variance $V_0$ but an altered mean $m_k$. The following is the precise statement for the limiting distribution of $L_{\SNR}$.

\begin{thm} \label{thm:main-weak}
	Let $M$ be a rank-$k$ spiked Wigner matrix with a spike $\bsU$ as in Definition \ref{def:spiked_Wigner} with $\Lambda = \omega I_k$ for some nonnegative integer $k$. Then,
	\beq
	L_{\SNR} \Rightarrow \caN(m_k, V_0)\,,
	\eeq
	where the variance $V_0$ is as in \eqref{eq:V_H} and the mean $m_k$ is given by
	\beq \begin{split} \label{eq:m_k}
		m_k = m_0 + k \left[ -\log(1-\SNR) + \left( \frac{2}{w_2} - 1 \right) \SNR + \left( \frac{1}{w_4-1} - \frac{1}{2} \right) \SNR^2 \right]=m_0 + \frac{k V_0}{2}.
	\end{split} 
	\eeq
\end{thm}

\begin{proof}
	Theorem \ref{thm:main-weak} directly follows from Theorem \ref{thm:CLT} in Section \ref{sec:LSS}.
\end{proof}

Since the mean of $L_{\SNR}$ depends on the rank of the spike, we can construct a hypothesis test between $\bsH_{k_1}$ and $\bsH_{k_2}$ in \eqref{eq:hyp} based on Theorems \ref{thm:main-weak} and \ref{thm:main_rec}. In this test, for a given spiked Wigner matrix $M$, we compute $L_{\SNR}$ and compare it with the critical value $m_{(k_1+k_2)/2}$,
\beq \begin{split} \label{eq:m_lambda}
	m_{(k_1+k_2)/2} := \frac{m_{k_1}+m_{k_2}}{2}.
\end{split} \eeq
See Algorithm \ref{alg:ht} for the detail.

In Theorems \ref{thm:CLT} and \ref{thm:CLT_rec}, we prove that the proposed test in Algorithm \ref{alg:ht} is optimal among all CLT-based tests, in the sense that the error is minimized with the test statistic $L_{\SNR}$ also for spiked random matrices.

\begin{algorithm}[!tb]
	\caption{Hypothesis test for a spiked Wigner matrix}
	\label{alg:ht}
	
	\begin{algorithmic}
		\STATE \textbf{Data}: 	$M_{ij}$, parameters $w_2, w_4$, $\lambda$
		\STATE $L_{\SNR} \gets$ test statistic in \eqref{eq:L_lambda}, \hskip5pt $m_{(k_1+k_2)/2} \gets$ critical value in \eqref{eq:m_lambda} with \eqref{eq:m_k}
		
		\IF{$L_{\SNR} \leq m_{(k_1+k_2)/2}$} \STATE{ {\textbf{Accept}} $\bsH_1$ } 
		\ELSE \STATE{ \textbf{Accept} $\bsH_2$ }
		\ENDIF
		
	\end{algorithmic}
\end{algorithm}

\begin{thm} \label{thm:test}
	The error of the test, $\err(\SNR) = \p( L_{\SNR} > m_{\SNR} | \bsH_0) + \p( L_{\SNR} \leq m_{\SNR} | \bsH_1)$, in algorithm \ref{alg:ht} converges to
	\[
	\erfc \left( \frac{k_2 - k_1}{4} \sqrt{\frac{V_0}{2}} \right).
	\]
\end{thm}
\begin{proof}
	Theorem \ref{thm:test} is a direct consequence of Theorems \ref{thm:main-weak} and \ref{thm:main_rec}. (See also Section 3 of \cite{AlaouiJordan2018} and the proof of Theorem 2 of \cite{chung2019weak}.)
\end{proof}

\begin{rem}
	When $w_4 =3$, we find that the error $\err(\SNR)$ converges to
	\beq \label{eq:gaussian_error}
	\erfc \left( \frac{k_2-k_1}{4} \sqrt{-\log (1-\SNR) + \left( \frac{2}{w_2} - 1 \right) \SNR} \right).
	\eeq	
	The optimal error for the weak detection, achieved by the LR test, coincides with the limiting error in \eqref{eq:gaussian_error} when the noise is Gaussian and the SNR $\SNR$ is sufficiently small; see \cite{jung2020weak}. Thus, our proposed test is optimal in this case.
\end{rem}

The test in Algorithm \ref{alg:ht} can be readily extended to the spiked rectangular matrices by replacing the test statistic in \eqref{eq:L_lambda} with the following one, which was introduced in \cite{jung2021} for the rank-$1$ spiked rectangular matrices.
\beq\label{eq:L_lambda_rec}
\begin{split}
	L_{\SNR} &= -\log \det \left( \left(1+\frac{\rat}{\SNR} \right)(1+\SNR)I - YY^T \right)  + \frac{\SNR}{\rat} \left( \frac{2}{w_4-1} - 1 \right) (\Tr YY^T-M) \\
	&\quad + M \left[\frac{\SNR}{\rat} - \log\left(\frac{\SNR}{\rat}\right) -\frac{1-\rat}{\rat}\log(1+\SNR) \right].
\end{split} \eeq
We have the following results for the asymptotic normality of Gaussian fluctuation of $L_{\SNR}$:
\begin{thm} \label{thm:main_rec}
	Let $Y$ be a spiked rectangular matrix in Definition \ref{defn:rect_mean} or \ref{defn:rect_cov} with $\Lambda = \SNR I_k$ for some nonnegative integer $k$ and $\lambda \in(0,\sqrt{\rat})$ and $w_4>1$. Then, for any spikes with $\bsU^T\bsU=\bsV^T\bsV=I_k$,
	\beq
	L_{\SNR} \Rightarrow \mathcal{N}(m_k,V_0),
	\eeq
	where the mean and the variance are given by
	\beq \begin{split} \label{eq:m_k_rec}
		m_k &= m_0  +k\left[-\log\left(1-\frac{\SNR^2}{\rat}\right) +\frac{\SNR^2}{\rat}\left(\frac{2}{w_4-1}-1\right)\right]
	\end{split} \eeq
	and
	\beq \label{eq:var_rec}
	V_0=-2\log\left(1-\frac{\SNR^2}{\rat}\right) +\frac{2\SNR^2}{\rat}\left(\frac2{w_4-1}-1\right)
	\eeq
	where
	\beq\label{eq:m_0_rec}
	m_0=-\frac{1}{2} \log\left(1-\frac{\SNR^2}{\rat}\right)  +\frac{\SNR^2}{2\rat}(w_4-3).
	\eeq
\end{thm}

Theorem \ref{thm:main_rec} directly follows from the general CLT result in Theorems \ref{thm:CLT_rec}. See Appendix \ref{sec:compute} for the detailed computation for the mean and the variance.

With Theorem \ref{thm:main_rec}, we find that Algorithm \ref{alg:ht} is available for the weak detection of the signal in the spiked rectangular matrices with the following change:
\begin{itemize}
	\item Data matrix is $Y_{ij}$ (instead of $M_{ij}$).
	\item Test statistic $L_{\SNR}$ is defined by \eqref{eq:L_lambda_rec} (instead of \eqref{eq:L_lambda}).
	\item Critical value $m_{(k_1+k_2)/2}$ is obtained by \eqref{eq:m_lambda} with \eqref{eq:m_0_rec} (instead of \eqref{eq:m_k}).
\end{itemize}

The limiting error of the test in this case is again $\erfc \left( \frac{k_2 - k_1}{4} \sqrt{\frac{V_0}{2}} \right)$ as in Theorem \ref{thm:test}, where $V_0$ is defined by \eqref{eq:var_rec}.

\subsection{Test with entrywise transformation for spiked matrices of additive type} \label{sec:entrywise}

The entrywise transform we applied with the PCA in Section \ref{subsec:PCA} can also be adapted to be used together with the proposed test in Algorithm \ref{alg:ht}; see also \cite{chung2019weak} where the same idea was applied for the rank-$1$ spiked Wigner matrix. Recall the transformation defined in \eqref{eq:alpha_transform} and the transformed matrix $\wt M$ in \eqref{eq:transformed-wigner}. We consider a test statistic
\beq \begin{split} \label{eq:wt L_lambda}
	\wt L_{\SNR} &:= - \log \det \left( (1+\SNR\fh)I - \sqrt{\SNR\fh} \tM \right) + \frac{\SNR\fh}{2} N \\
	&\qquad + \sqrt{\SNR} \left( \frac{2\sqrt{F_{g,d}}}{w_2} - \sqrt{\fh} \right) \Tr \tM 
	+ \lambda \left( \frac{\gh}{\wt{w_4}-1} - \frac{\fh}{2} \right) (\Tr \tM^2 - N),
\end{split} 
\eeq
where 
\[
\gh = \frac{1}{2\fh} \int_{-\infty}^{\infty} \frac{g'(w)^2 g''(w)}{g(w)^2} \dd w,
\quad \wt{w_4} = \frac{1}{(\fh)^2} \int_{-\infty}^{\infty} \frac{(g'(w))^4}{(g(w))^3} \dd w.
\]
We then have the following CLT result for $\wt L_{\SNR}$ that generalizes the results in \cite{chung2019weak}.

\begin{thm} \label{thm:trans_main}
	Assume the conditions in Theorem \ref{thm:main-weak}, satisfying Assumption \ref{assump:entry1} with $\phi>3/8$. If $\lambda \fh < 1$,
	\beq
	\wt L_{\SNR} \Rightarrow \caN(\wt m_k, \wt V_0),
	\eeq
	where the mean and the variance are given by
	\beq \begin{split} \label{eq:wt m_k}
		\wt m_k &= - \frac{1}{2} \log(1-\SNR\fh) + \left( \frac{(w_2 -1)\gh}{\wt w_4 -1} - \frac{\fh}{2} \right) \SNR + \frac{\wt w_4 -3}{4} (\SNR\fh)^2 \\
		&\qquad + k \left[ -\log(1- \SNR\fh) + \left( \frac{2\fhd}{w_2} - \fh \right) \SNR + \left( \frac{(\gh)^2}{\wt w_4-1} - \frac{(\fh)^2}{2}\right) \SNR^2 \right],
	\end{split} \eeq
	\beq \begin{split}
		\wt V_0 = -2 \log(1- \SNR\fh) + \left( \frac{4\fhd}{w_2} - 2\fh \right) \SNR + \left( \frac{2(\gh)^2}{\wt w_4-1} - (\fh)^2 \right) \SNR^2.
	\end{split} \eeq
\end{thm}

\begin{proof}
	Theorem \ref{thm:trans_main} directly follows from Theorem \ref{thm:trans_CLT} in Section \ref{sec:LSS}.
\end{proof}

Based on Theorem \ref{thm:trans_main}, we can adapt the test in Algorithm \ref{alg:ht} to construct a test that utilizes the entrywise transformation. In this test, we compute $\wt L_{\Lambda}$ and compare it with the critical value
\beq \begin{split} \label{eq:wt m_lambda}
	\wt m_{(k_1+k_2)/2} := (\wt m_{k_1} + \wt m_{k_2})/2.
\end{split} \eeq
See Algorithm \ref{alg:htet} for the detail. The limiting error of the test is given as follows.

\begin{algorithm}[!tb]
	\caption{Hypothesis test for a spiked Wigner matrix with entrywise transformation}
	\label{alg:htet}
	
	\begin{algorithmic}
		\STATE \textbf{Data}: $M_{ij}$, parameters $w_2, w_4$, $\lambda$, densities $g, g_d$
		\STATE $\tM \gets$ transformed matrix in \eqref{eq:transformed-wigner}, \hskip5pt $\wt L_{\SNR} \gets$ test statistic in \eqref{eq:wt L_lambda}, \hskip5pt $\wt m_{(k_1+k_2)/2} \gets$ critical value in \eqref{eq:wt m_lambda} with \eqref{eq:wt m_k}
		
		\IF{$\wt L_{\SNR} \leq \wt m_{(k_1+k_2)/2}$} \STATE{ {\textbf{Accept}} $\bsH_1$ }
		\ELSE \STATE{ {\textbf{Accept}} $\bsH_2$ }
		\ENDIF
		
	\end{algorithmic}
\end{algorithm}

\begin{thm} \label{thm:trans_test}
	The error of the test in Algorithm \ref{alg:htet} converges to
	\[
	\erfc \left( \frac{k_2-k_1}{4} \sqrt{\frac{\wt V_0}{2}} \right).	
	\]
\end{thm}

\begin{proof}
	Theorem \ref{thm:trans_test} is a direct consequence of Theorem \ref{thm:trans_CLT_rec}.
\end{proof}

We also propose an analogous test can for the additive model of the spiked rectangular matrices as follows. Recall the transformed matrix $\wt Y \equiv \wt Y^{(0)}$ in \eqref{eq:transformed-cov}. Define the test statistic $\wt{L}_{\SNR}$ by
\beq \label{eq:wt L_lambda_rec}
\begin{split}	
	\wt{L}_{\SNR} &= - \log \det \left( \left(1+\frac{\rat}{\SNR\fh}\right)(1+\SNR\fh)I - \wt{Y}\wt{Y}^T \right) + \frac{2\SNR}{\rat} \left( \frac{\gh}{\wt w_4-1} - \frac{\fh}{2} \right) (\Tr \wt{Y}\wt{Y}^T-M)\\
	& \qquad +M \left[\frac{\SNR\fh}{\rat} - \log\left(\frac{\SNR\fh}{\rat}\right) -\frac{1-\rat}{\rat}\log(1+\SNR\fh)\right].
\end{split}
\eeq

We then have the following CLT for the test statistic.
\begin{thm}\label{thm:main_transformed_rec}
	Assume the conditions in Theorem \ref{thm:main_rec}, satisfying Assumption \ref{assump:entry1} with $\phi>3/8$. If $\lambda < \sqrt{\rat}/\fh$,
	\beq
	\wt L_{\SNR} \Rightarrow\caN(\wt m_k,\wt V_0),
	\eeq
	where the mean and the variance are given by
	\beq\label{eq:m^t_0}
	\wt m_0=-\frac1{2} \log\left(1-\frac{\SNR^2(\fh)^2}{\rat}\right)+\frac{\SNR^2(\fh)^2}{2\rat}(\wt w_4-3)
	\eeq
	\beq\label{eq:m^t_+}
	\wt m_k=\wt m_0+k\left[- \log\left(1-\frac{\SNR^2(\fh)^2}{\rat}\right)+\frac{2\SNR^2}{\rat}\left(\frac{(\gh)^2}{\wt w_4-1}-\frac{(\fh)^2}{2}\right)\right]
	\eeq
	and 
	\beq\label{eq:V^t_0}
	\wt V_0=\frac{4\SNR^2}{\rat}\left(\frac{(\gh)^2}{\wt w_4-1}-\frac{(\fh)^2}{2}\right)-2\log\left(1-\frac{\SNR^2(\fh)^2}{\rat}\right).
	\eeq
\end{thm}

With Theorem \ref{thm:main_transformed_rec}, we can adjust Algorithm \ref{alg:ht} for the weak detection of the signal in the additive model of spiked rectangular matrices, where we make the following change:
\begin{itemize}
	\item Data matrix is $Y_{ij}$ (instead of $M_{ij}$).
	\item Transformed matrix is $\wt Y$ (instead of $\wt M$), defined by \eqref{eq:transformed-cov} with $\alpha = 0$.
	\item Test statistic $\wt L_{\SNR}$ is defined by \eqref{eq:wt L_lambda_rec} (instead of \eqref{eq:wt L_lambda}).
	\item Critical value $m_{(k_1+k_2)/2}$ is obtained by \eqref{eq:wt m_lambda} with \eqref{eq:m^t_+} (instead of \eqref{eq:wt m_k}).
\end{itemize}

In Appendix \ref{sec:ex}, we consider several examples of spiked Wigner matrices and spiked rectangular matrices, where we compare the errors from numerical simulations and the theoretical errors of the proposed algorithms. We find that the numerical errors of the proposed tests closely match the corresponding theoretical errors and the error from Algorithm \ref{alg:htet} is lower than that of Algorithm \ref{alg:ht}.

\subsection{Rank estimation} \label{subsec:adaptive}

The test in Algorithm \ref{alg:ht} requires prior knowledge about $k_1$ and $k_2$, the possible ranks of the planted spike. In this section, we adapt the idea of the proposed tests in Algorithm \ref{alg:ht} to estimate the rank of the signal when there is no prior information on the rank $k$. Recall that the test statistic $L_{\SNR}$ defined in \eqref{eq:L_lambda} does not depend on the rank of the matrix. As proved in Theorem \ref{thm:main-weak}, the test statistic $L_{\SNR}$ converges to a Gaussian random variable with mean $m_k$ and the variance $V_0$, where $m_k$ is equi-distributed with respect to $k$ and $V_0$ does not depend on $k$. It is then natural to set the best candidate for $k$, which we call $\kappa$, be the minimizer of the distance $|L_{\SNR} - m_k|$. This procedure is equivalent to find \emph{the nearest} nonnegative integer of the value
\beq \label{eq:adaptive_value}
\kappa' := \frac{2(L_{\SNR} - m_0)}{V_0}
\eeq
rounding half down.

We describe the procedure in Algorithm \ref{alg:at}; for example, its probability of error for spiked Wigner matrix converges to
\beq \begin{split} \label{eq:rank_k_error}
	& \p(k=0) \cdot \p \left( Z > \frac{\sqrt{V_0}}{4} \right) + \sum_{i=1}^{\infty} \p(k=i) \cdot \p \left( |Z| > \frac{\sqrt{V_0}}{4} \right) \\
	&= \left( 1 - \frac{\p(k=0)}{2} \right) \cdot\erfc \left( \frac{1}{4} \sqrt{\frac{V_0}{2}} \right),
\end{split} \eeq
where $Z$ is a standard Gaussian random variable. Note that it depends only on $\p(k=0)$.

\begin{algorithm}[!tb]
	\caption{Rank estimation}
	\label{alg:at}
	
	\begin{algorithmic}
		\STATE \textbf{Data}: $M_{ij}$ (or $Y_{ij}$), parameters $w_2, w_4$, $\lambda$
		\STATE $L_{\SNR} \gets$ test statistic in \eqref{eq:L_lambda} or \eqref{eq:L_lambda_rec}, \hskip5pt $m_0 \gets$ mean in \eqref{eq:m_0} or \eqref{eq:m_0_rec}, \hskip5pt $m_1 \gets$ mean in \eqref{eq:m_k} or \eqref{eq:m_k_rec} with $k=1$
		\STATE $\kappa' \gets$ value in \eqref{eq:adaptive_value}
		
		\IF{$L_{\SNR} \leq (m_0 + m_1)/2$} \STATE{ {\textbf{Set}} $\kappa=0$ }
		\ELSE \STATE{ {\textbf{Set}} $\kappa = \lceil \kappa' - 0.5 \rceil$ }
		\ENDIF
		
	\end{algorithmic}
\end{algorithm}

The error can be lowered if the range of $k$ is known a priori. See Appendix \ref{sec:ex}. It is also possible to improve Algorithm \ref{alg:at} by pre-transforming the data matrix entrywise as in Section \ref{sec:entrywise}. We omit the detail.

\section{Central Limit Theorems} \label{sec:LSS}
In this section, we collect our results on general CLTs for the LSS of spiked random matrices. To precisely define the statements, we introduce the Chebyshev polynomials of the first kind.

\begin{defn}[Chebyshev polynomial]
	The $n$-th Chebyshev polynomial (of the first kind) $T_n$ is a degree $n$ polynomial defined by $T_0(x) = 1$, $T_1(x) = x$, and
	\[
	T_{n+1}(x) = 2x T_n(x) - T_{n-1}(x).
	\]
\end{defn}

We first state a CLT for the LSS of spiked Wigner matrices. Recall that we denote by $\mu_1 \geq \mu_2 \geq \dots \geq \mu_N$ the eigenvalues of a spiked Wigner matrix $M$.
\begin{thm} \label{thm:CLT}
	Assume the conditions in Theorem \ref{thm:main-weak}. Suppose that a function $f$ is analytic on an open interval containing $[-2, 2]$. Then,
	\[
	\left( \sum_{i=1}^N f(\mu_i) - N \int_{-2}^2 \frac{\sqrt{4-z^2}}{2\pi} f(z) \, \dd z \right) \Rightarrow \caN\left(m_k(f), V_0(f)\right)\,.
	\]
	The mean and the variance of the limiting Gaussian distribution are given by
	\[ \begin{split} 
	m_k(f) = \frac{1}{4} \left( f(2) + f(-2) \right) -\frac{1}{2} \tau_0(f) + (w_2 -2) \tau_2(f) 
	+ (w_4-3) \tau_4(f) + k \sum_{\ell=1}^{\infty} \sqrt{\SNR^{\ell}} \tau_{\ell}(f),
	\end{split} 
	\]
	\[ \begin{split} 
	V_0(f) = (w_2-2) \tau_1(f)^2 + 2(w_4-3) \tau_2(f)^2 
	+ 2\sum_{\ell=1}^{\infty} \ell \tau_{\ell}(f)^2\,,
	\end{split} 
	\]
	where we let
	\[
	\tau_{\ell}(f) = \frac{1}{\pi} \int_{-2}^2 T_{\ell} \left( \frac{x}{2} \right) \frac{f(x)}{\sqrt{4-x^2}} \dd x.
	\]
	Furthermore, for $m_k$, $m_0$, and $V_0$ defined in Theorem \ref{thm:main-weak},
	\[
	\left|\frac{m_k(f) - m_0(f)}{\sqrt{V_{0}(f)}}\right| \leq \left|\frac{m_k- m_0}{\sqrt{V_0}}\right|
	\]
	The equality holds if and only if $f(x)=C_1\phi_{\SNR}(x) +C_2$ for some constants $C_1$ and $C_2$ where 
	\[
	\phi_{\SNR}(x) := \log \left( \frac{1}{1-\sqrt{\SNR}x + \SNR} \right) + \sqrt{\SNR} \left( \frac{2}{w_2} - 1 \right) x 
	+ \SNR \left( \frac{1}{w_4-1} - \frac{1}{2} \right) x^2.
	\]
\end{thm}
We will give a proof of Theorem \ref{thm:CLT} in Appendix \ref{app:CLT}. With the entrywise transformation in Section \ref{sec:entrywise}, we have the following changes in Theorem \ref{thm:CLT}. Recall that $\wt\mu_1 \geq \wt\mu_2 \geq \dots \geq \wt\mu_N$ are the eigenvalues of the transformed matrix $\wt M$.


\begin{thm} \label{thm:trans_CLT}
	Assume the conditions in Theorem \ref{thm:CLT}, satisfying Assumption \ref{assump:entry1} with $\phi>3/8$. If $\lambda \fh < 1$,
	\[ \begin{split}
	\left( \sum_{i=1}^N f(\wt\mu_i) - N \int_{-2}^2 \frac{\sqrt{4-z^2}}{2\pi} f(z) \, \dd z \right)
	\Rightarrow \caN(\wt m_k(f), \wt V_0(f))\,.
	\end{split} 
	\]
	The mean and the variance of the limiting Gaussian distribution are given by
	\beq \begin{split} \label{eq:mean_tM}
		\wt m_k(f) &= \frac{1}{4} \left( f(2) + f(-2) \right) -\frac{1}{2} \tau_0(f) + k \sqrt{\SNR\fhd} \tau_1 (f) 
		+ (w_2 -2 + k \SNR\gh) \tau_2(f)\\
		&\qquad + (\wt{w_4}-3) \tau_4(f) + k \sum_{\ell=3}^{\infty} \sqrt{(\SNR\fh)^\ell} \tau_{\ell}(f),
	\end{split} 
	\eeq
	\[	
	\wt V_{0}(f)= (w_2-2) \tau_1(f)^2 + 2(\wt{w_4}-3) \tau_2(f)^2 + 2\sum_{\ell=1}^{\infty} \ell \tau_{\ell}(f)^2.
	\]
	Furthermore, for $\wt m_k$, $\wt m_0$, and $\wt V_0$ defined in Theorem \ref{thm:main-weak},
	\[
	\left|\frac{\wt m_{k_2}(f) - \wt m_{k_1}(f)}{\sqrt{\wt V_{0}(f)}}\right| \leq \left|\frac{\wt m_{k_2}- \wt m_{k_1}}{\sqrt{\wt V_0}}\right|
	\]
	The equality holds if and only if $f(x)=C_1\wt\phi_{\SNR}(x) +C_2$ for some constants $C_1$ and $C_2$ with the function 
	\[
	\wt \phi_{\SNR}(x) := \log \left( \frac{1}{1-\sqrt{\SNR\fh}x + \SNR\fh} \right) +  \left( \frac{2\sqrt{\fhd}}{w_2} - \sqrt{\fh} \right) x 
	+ \SNR \left( \frac{\gh}{\wt w_4-1} - \frac{\fh}{2} \right) x^2.
	\]
\end{thm}
We will also prove Theorem \ref{thm:trans_CLT} in Appendix \ref{app:CLT}.
\begin{rem}\label{rmk:differentSNRs}
	For a general case where the spike $\Lambda=\diag(\SNR_1,\cdots,\SNR_k)$ with possibly distinct $\SNR_i$'s, we can prove the CLT and the transformed CLT, analogous to Theorems \ref{thm:CLT} and \ref{thm:trans_CLT}, respectively, where the means of the limiting Gaussians are given by
	\[ \begin{split} 
	m_M(f) &= \frac{1}{4} \left( f(2) + f(-2) \right) -\frac{1}{2} \tau_0(f) + (w_2 -2) \tau_2(f) 
	+ (w_4-3) \tau_4(f) \\
	&\qquad + \sum_{s=1}^k\sum_{\ell=1}^{\infty} \sqrt{\SNR_s^{\ell}} \tau_{\ell}(f),
	\end{split} 
	\]
	\[ 
	\begin{split} 
	\wt m_M(f) &= \frac{1}{4} \left( f(2) + f(-2) \right) -\frac{1}{2} \tau_0(f) +(w_2 -2)\tau_2(f)+ (\wt{w_4}-3) \tau_4(f) \\
	&\qquad + \sum_{s=1}^k\sqrt{\SNR_s\fhd} \tau_1 (f) 
	+  \SNR_s\gh \tau_2(f) +  \sum_{s=1}^k\sum_{\ell=3}^{\infty} \sqrt{(\SNR_s\fh)^\ell} \tau_{\ell}(f),
	\end{split} 
	\]
	and the variances are equal to $V_0(f)$ in Theorem \ref{thm:CLT} and $\wt V_0(f)$ in Theorem \ref{thm:trans_CLT}, respectively. Adapting the proposed tests in Algorithms \ref{alg:ht} and \ref{alg:htet}, it is possible to construct hypothesis tests for the weak detection in this case.
\end{rem}

The next result is the CLT for the LSS of spiked rectangular matrices $Y$, where we denote by $\mu_1 \geq \mu_2 \geq \cdots\geq\mu_M$ the eigenvalues of $YY^T$.

\begin{thm} \label{thm:CLT_rec}
	Assume the conditions in Theorem \ref{thm:main_rec}. Suppose that a function $f$ is analytic on an open set containing an interval $[d_-,d_+]$. Then,
	\beq \begin{split} \label{eq:CLT_statement}
		&\left(\sum_{i=1}^{M}f(\mu_i)-M \int_{d_-}^{d_+} \frac{\sqrt{(x-d_-)(d_+ -x)}}{2\pi \rat x} f(x) \, \dd x \right) \Rightarrow \mathcal{N}(m_k(f),V_0(f)).
	\end{split} \eeq
	The mean and the variance of the limiting Gaussian distribution are given by
	\[ \begin{split}
	m_k(f) &= \frac{\wt{f}(2) + \wt{f}(-2)}{4} -\frac{\tau_0(\wt{f})}{2} +(w_4-3)\tau_2(\wt{f})  +k\sum_{\ell=1}^{\infty}\left(\frac{\SNR}{\sqrt{\rat}}\right)^{\ell}\tau_\ell(\wt{f})
	\end{split} \]
	and
	\[
	V_0(f) =2\sum_{\ell=1}^\infty \ell\tau_\ell(\wt{f})^2+(w_4-3)\tau_1(\wt{f})^2,
	\]
	where we let $\wt f(x)=f(\sqrt{\rat}x+1+\rat)$.
	
	Furthermore, for $m_k$, $m_0$, and $V_0$ defined in Theorem \ref{thm:main_rec},
	\[
	\left|\frac{m_{k_2}(f) - m_{k_1}(f)}{\sqrt{V_{0}(f)}}\right| \leq \left|\frac{m_{k_2}- m_{k_1}}{\sqrt{V_0}}\right|
	\]
	The equality holds if and only if $f(x)=C_1\phi_{\SNR}(x) +C_2$ for some constants $C_1$ and $C_2$ with the function 
	\[
	\phi_{\SNR}(x)=\frac{\SNR}{\rat}\left(\frac{2}{w_4-1}-1\right)x  -\log\left(\left(1+\frac{\rat}{\SNR} \right)(1+\SNR) - x\right).
	\]
\end{thm}

Lastly, we state the pre-transformed CLT for the LSS of the additive model of spiked rectangular matrices. We let $\wt Y$ be the transformed matrix and $\wt\mu_1 \geq \wt\mu_2 \geq \dots \geq \wt\mu_N$ the eigenvalues of $\wt Y \wt Y^T$.

\begin{thm}\label{thm:trans_CLT_rec}
	Assume the conditions in Theorem \ref{thm:CLT_rec}, satisfying Assumption \ref{assump:entry1} with $\phi>3/8$. If $\lambda < \sqrt{\rat}/\fh$,
	\begin{align}
	\left(\sum_{i=1}^{M}f(\wt\mu_i)-M\int_{d_-}^{d_+}f(x)\rho_{MP,\rat}(dx)\right)\Rightarrow\mathcal{N}(\wt m_{k}(f),\,\wt V_{0}(f)).
	\end{align}
	The mean and the variance of the limiting Gaussian distribution are given by
	\beq \begin{split}
		\wt m_{k}(f) &= \frac{\wt{f}(2) + \wt{f}(-2)}{4}  -\frac{1}{2} \tau_0(\wt{f})+\frac{k\SNR}{\sqrt{\rat}}(\gh - \fh) \tau_1 (\wt{f})+(\wt{w_4}-3)\tau_2(\wt{f}) \\
		&\qquad +k\sum_{\ell=1}^{\infty}\left(\frac{\SNR\fh}{\sqrt{\rat}}\right)^{\ell}\tau_\ell(\wt{f})
	\end{split} \eeq
	and
	\beq
	\wt V_{0}(f) =2\sum_{\ell=1}^\infty \ell\tau_\ell(\wt{f})^2+(\wt{w_4}-3)\tau_1(\wt{f})^2.
	\eeq
	where $\wt f(x)=f(\sqrt{\rat}x+1+\rat).$
	
	Furthermore, for $\wt m_k$, $\wt m_0$, and $\wt V_0$ defined in Theorem \ref{thm:main_transformed_rec},
	The equality holds if and only if $f(x)=C_1\wh\phi_\SNR(x) +C_2$ for some constants $C_1$ and $C_2$ with the function 
	\[
	\wh\phi_\SNR(x)=\frac{2\lambda}{\rat}\left(\frac{\gh}{\wt w_4-1}-\frac{\fh}{2}\right)x-\log\left(\left(\frac{\rat}{\SNR\fh}+1\right)(\SNR\fh+1)-x\right).
	\]
\end{thm}
\begin{rem}
	As in Remark \ref{rmk:differentSNRs}, for a general case with $\Lambda=\diag(\SNR_1,\cdots,\SNR_k)$, the CLT and the transformed CLT hold with the adjusted means
	\[ \begin{split} 
	m_Y(f)  =\frac{\wt{f}(2) + \wt{f}(-2)}{4} +\frac{\tau_0(\wt{f})}{2} +(w_4-3)\tau_2(\wt{f})  +\sum_{s=1}^k\sum_{\ell=1}^{\infty}\left(\frac{\SNR_s}{\sqrt{\rat}}\right)^{\ell}\tau_\ell(\wt{f}),
	\end{split} 
	\]
	\[ \begin{split} 
	\wt m_Y(f) &= \frac{\wt{f}(2) + \wt{f}(-2)}{4}  -\frac{1}{2} \tau_0(\wt{f})+(\wt{w_4}-3)\tau_2(\wt{f})+\sum_{s=1}^k\frac{\SNR_s}{\sqrt{\rat}}(\gh - \fh) \tau_1 (\wt{f})\\
	&\qquad +\sum_{s=1}^k\sum_{\ell=1}^{\infty}\left(\frac{\SNR_s\fh}{\sqrt{\rat}}\right)^{\ell}\tau_\ell(\wt{f}),
	\end{split} 
	\]
	where the variances are given $V_0(f)$, $\wt V_0(f)$, respectively. Further, the corresponding optimal functions and test statistic can be calculated by following the same procedure in \cite{jung2021}.
\end{rem}

\section{Conclusion and Future Works} \label{sec:summary}

In this paper, we considered the detection problems of the spiked random model with general ranks. First, we prove the sub-optimality of the PCA for the non-Gaussian noise. Further, we proposed a hypothesis test based on the central limit theorem for the linear spectral statistics of the data matrix and introduced a test for rank estimation that do not require any prior information on the rank of the signal. It was shown that the error of the proposed hypothesis test matches the error of the likelihood ratio test in case the noise is Gaussian and the signal-to-noise ratio is small. With the knowledge on the density of the noise, the test was further improved by applying an entrywise transformation.
%

We believe that the hypothesis test with the entrywise transformed matrix proposed in this paper can be extended to the multiplicative model of spiked rectangular matrix. This will be discussed in our future works.

\subsubsection*{Acknowledgments} 
The work of J. H. Jung and J. O. Lee was partially supported by National Research Foundation of Korea under grant number NRF-2019R1A5A1028324. The work of H. W. Chung was partially supported by National Research Foundation of Korea under grant number 2017R1E1A1A01076340 and by the Ministry of Science and ICT, Korea, under an ITRC Program, IITP-2019-2018-0-01402.

\bibliographystyle{abbrv}
\bibliography{weak_detection}
\appendix

\section{Examples and Simulations} \label{sec:ex}

In Appendix \ref{sec:ex}, we consider specific examples of spiked random matrices under various settings. We first demonstrate with an example the change of the threshold by the improved PCA in Section \ref{sec:main1}. We then provide the details of the proposed tests in Algorithms \ref{alg:ht} and \ref{alg:htet} with different examples, and the test for rank estimation in Algorithm \ref{alg:at} for these and compute the theoretical errors. We also perform the numerical simulation for the proposed tests and compare the numerical errors with the theoretical errors.

\subsection{Spiked Wigner matrix}
\subsubsection{Improved PCA with Entrywise Transformation} \label{subsubsec:Wigner_PCA}

Our first example is a spiked Wigner matrix with non-Gaussian noise to which we apply the entrywise transformation for the improved PCA. We let the density function of the noise be a bimodal distribution with unit variance, defined as
\beq\label{eq:bimodal}
g(x)=g_d(x)=\frac{1}{\sqrt{2\pi}}\left(e^{-2(x-\sqrt{3}/2)^2}+e^{-2(x+\sqrt{3}/2)^2}\right),
\eeq
which is the density function of a random variable
\[
\frac{1}{2} \caN + \frac{\sqrt{3}}{2}\caR,
\]
where $\caN$ is a standard Gaussian random variable and $\caR$ is a Rademacher random variable, independent to each other.

We sample $Z_{ij} = Z_{ji}$ independently from the density $g$ and let $W_{ij}=Z_{ij}/\sqrt{N}$.
We let $\bsu(\ell) = (u_1(\ell), u_2(\ell), \dots, u_N(\ell))^T$, where $\sqrt{N} u_i(\ell)$'s are i.i.d. Rademacher random variables for $i=1, 2, \dots, N$ and $\ell=1,2,3$. 
The data matrix $M = \bsU \Lambda^{1/2} \bsU^T$, where $\bsU = [\bsu(1), \bsu(2), \bsu(3)]$ and $\Lambda = \diag (\lambda, \lambda, \lambda, 0, 0, \dots, 0)$. The size of the data matrix is set to be $N=4000$. The BBP-transition predicts that the largest eigenvalue of $M$ pops up from the bulk of the spectrum if $\lambda > 1$.

With the entrywise transformation defined in~\eqref{eq:transformed-wigner}, we obtain a transformed matrix
\beq
\wt M_{ij} = \frac{1}{\sqrt{\fh N}} h(\sqrt{N} M_{ij})
\eeq
where 
\beq
h(x)=-\frac{g'(x)}{g(x)}=\frac{2\left(\sqrt{3}-e^{4\sqrt{3}x\left(\sqrt{3}-2x\right)}+2x\right)}{1+e^{4\sqrt{3}x}}
\eeq
and $\fh = \int_{-\infty}^{\infty} \frac{(g'(x))^2}{g(x)} \dd x\approx 2.50810$.
From Theorem~\ref{thm:trans-wigner}, it is expected that the largest eigenvalue of $\wt M$ separates from other eigenvalues if $\lambda > \frac{1}{\fh} \approx 0.3987$.

In the numerical experiment, we set
\beq\label{eq:lam}
\lambda_\ell=\frac{\ell+\frac{1}{\fh}}{\ell+1}
\eeq
for $\ell = 1, 2, 3$, and we compare the spectrum of the matrices $M$ and $\wt M$. In Figure~\ref{fig:spectrum_wig}, we find three isolated eigenvalues in the spectrum of $\wt M$ (right), which are absent in that of $M$ (left).

\begin{figure}[h]
	\vskip 0.2in
	\begin{center}
		\centerline{\includegraphics[width=\columnwidth]{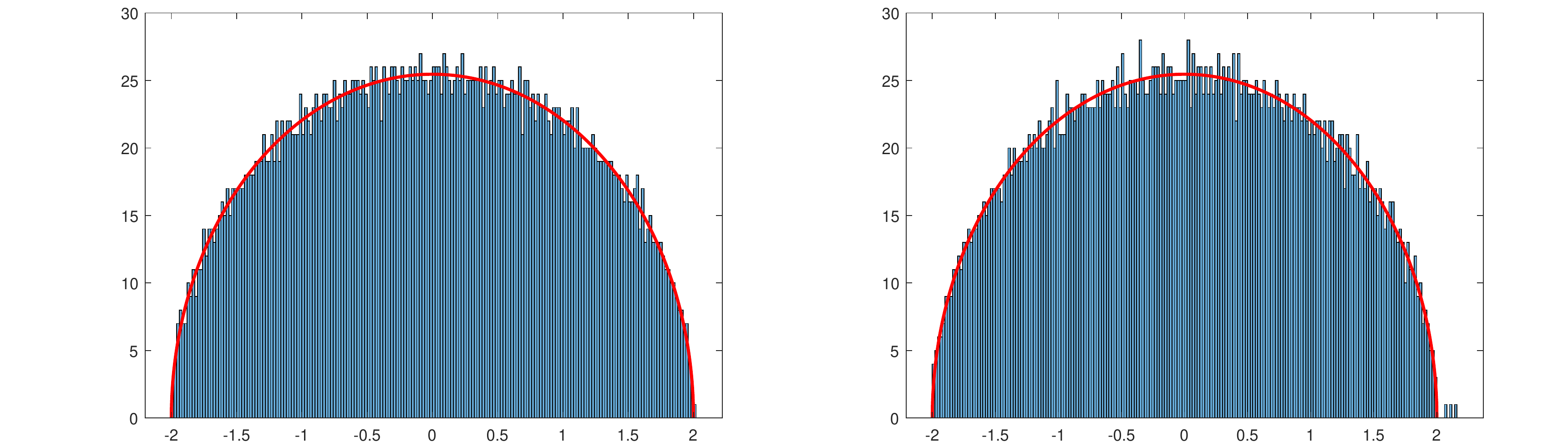}}
		\caption{The spectrum of the data matrix ($N=4000$) with bimodal noise, before (left) and after (right) the entrywise transformation. Three eigenvalues pop up from the bulk of the spectrum after the entrywise transformation.}
		\label{fig:spectrum_wig}
	\end{center}
	\vskip -0.2in
\end{figure}

\subsubsection{Spiked Gaussian Wigner matrix} \label{subsec:GOE}

We consider the weak detection problem with the simplest case of the spiked Gaussian Wigner matrix where $w_2 = 2$ (i.e., $W$ is a GOE matrix) and the signal $\bsu(m)  = (u_1(m), u_2(m), \dots, u_N(m))$ where $\sqrt{N} u_i(m)$'s are i.i.d. Rademacher random variable. Note that the parameters $w_2 = 2$ and $w_4 = 3$. 

In the numerical simulation done in Matlab, we generated 10,000 independent samples of the $256 \times 256$ data matrix $M$, where we fix $k_1=1$ (under $\bsH_1$) and vary $k_2$ from $2$ to $5$ (under $\bsH_{k_2}$), with the SNR $\lambda$ varying from $0$ to $0.7$. To apply Algorithm \ref{alg:ht}, we compute
\beq \begin{split}
	L_{\lambda} = -\log \det \big( (1+\lambda)I - \sqrt{\lambda} M \big) + \frac{\lambda N}{2}.
\end{split} 
\eeq
We accept $\bsH_1$ if 
\[
L_{\lambda} \leq \frac{m_{k_1}+m_{k_2}}{2} = -\frac{k_2 + 2}{2} \log(1-\lambda)
\]
and reject $\bsH_1$ otherwise. The (theoretical) limiting error of the test is
\beq \label{eq:limit_error_1a}
\erfc \left( \frac{k_2 - 1}{4} \sqrt{-\log (1-\lambda)} \right).
\eeq

In Figure \ref{fig:LSS_GOE}, we compare the error from the numerical simulation and the theoretical error of the proposed algorithm, which show that the numerical errors of the test closely match the theoretical errors. 

\begin{figure}[ht]
	\vskip 0.2in
	\begin{center}
		\centerline{\includegraphics[width=250pt]{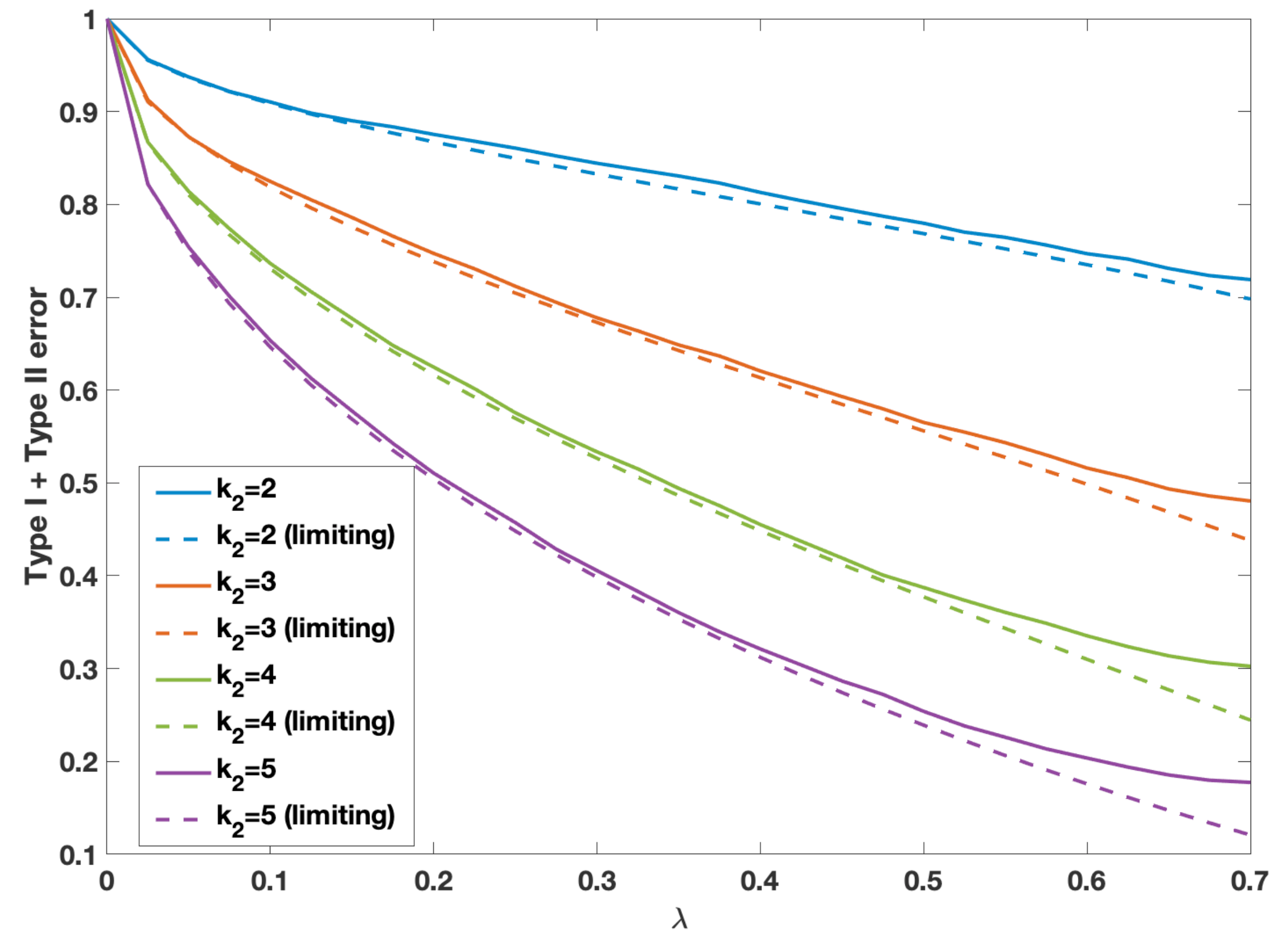}}
		\caption{The errors from the simulation with Algorithm \ref{alg:ht} (solid) versus the limiting errors \eqref{eq:limit_error_1a} (dashed) for the setting in Section \ref{subsec:GOE} with $k_2 = 2, 3, 4, 5$.}
		\label{fig:LSS_GOE}
	\end{center}
	\vskip -0.2in
\end{figure}

\subsubsection{Spiked Wigner matrix} \label{subsec:sech}

We next consider a spiked Wigner matrix with non-Gaussian noise, where the density function of the noise matrix is given by
\beq\label{eq:sech}
g(x) = g_d(x) = \frac{1}{2 \cosh (\pi x/2)} = \frac{1}{e^{\pi x/2} + e^{-\pi x/2}}.
\eeq
We sample $Z_{ij} = Z_{ji}$ from the density $g$ and let $W_{ij} = Z_{ij}/\sqrt{N}$. We again let the signal $\bsu(m)  = (u_1(m), u_2(m), \dots, u_N(m))$ where $\sqrt{N} u_i(m)$'s are i.i.d. Rademacher random variable. Note that the parameters $w_2 = 1$ and $w_4 = 5$. We again perform the numerical simulation 10,000 samples of the $256 \times 256$ data matrix $M$ with the SNR $\lambda$ varying from $0$ to $0.6$, where we fix $k_1=1$ (under $\bsH_1$) and $k_2=3$ (under $\bsH_2$).

In Algorithm \ref{alg:ht}, we compute
\beq \begin{split}
	L_{\lambda} = -\log \det \big( (1+\lambda)I - \sqrt{\lambda} M \big) + \frac{\lambda N}{2} 
	+ \sqrt{\lambda} \Tr M - \frac{\lambda}{4} (\Tr M^2 - N).
\end{split} 
\eeq
We accept $\bsH_1$ if 
\[
L_{\lambda} \leq \frac{m_{k_1}+m_{k_2}}{2} = -\frac{k_2 + 2}{2} \log(1-\lambda) + \frac{k_2 \lambda}{2} - \frac{(k_2 -3) \lambda^2}{8}
\]
and accept $\bsH_2$ otherwise. The (theoretical) limiting error of the test is
\beq \label{eq:limit_error_1}
\erfc \left( \frac{k_2 - 1}{4} \sqrt{-\log (1-\lambda)  + \lambda - \frac{\lambda^2}{4}} \right).
\eeq

We can further improve the test by introducing the entrywise transformation given by
\[
h(x) = -\frac{g'(x)}{g(x)} = \frac{\pi}{2} \tanh \frac{\pi x}{2}.
\]
The Fisher information $\fh = \frac{\pi^2}{8}$, which is larger than $1$. We thus construct a transformed matrix $\wt M$ by
\[
\wt M_{ij} = \frac{2\sqrt{2}}{\pi \sqrt{N}} h(\sqrt{N} M_{ij}) = \sqrt{\frac{2}{N}} \tanh \left( \frac{\pi \sqrt{N}}{2} M_{ij} \right).
\]
If $\lambda > \frac{1}{\fh} = \frac{8}{\pi^2}\approx0.8106$, we can apply PCA for strong detection of the signal. If $\lambda < \frac{8}{\pi^2}$, applying Algorithm \ref{alg:htet}, we compute
\[ \begin{split}
\wt L_{\lambda} = -\log \det \left( \left(1+\frac{\pi^2 \lambda}{8}\right)I - \sqrt{\frac{\pi^2 \lambda}{8}} \tM \right) + \frac{\pi^2 \lambda N}{16} 
+ \frac{\pi \sqrt{\lambda}}{2 \sqrt{2}} \Tr \tM + \frac{\pi^2 \lambda}{16} (\Tr \tM^2 - N).
\end{split} 
\]
(Here, $\fh = \fhd = \frac{\pi^2}{8}$, $\gh = \frac{\pi^2}{16}$, and $\wt w_4 = \frac{3}{2}$.) We accept $\bsH_1$ if 
\[
\wt L_{\lambda} \leq -\frac{k_2 +2}{2} \log \left(1-\frac{\pi^2 \lambda}{8} \right) + \frac{k_2 \pi^2 \lambda}{16} - \frac{3 \pi^4 \lambda^2}{512}
\]
and accept $\bsH_2$ otherwise. The limiting error with entrywise transformation is
\beq \label{eq:limit_error_2}
\erfc \left( \frac{k_2 - 1}{4} \sqrt{ -\log \left(1-\frac{\pi^2 \lambda}{8}\right) + \frac{\pi^2 \lambda}{8}} \right).
\eeq
Since $\erfc(\cdot)$ is a decreasing function and $\frac{\pi^2}{8} > 1$, it is immediate to see that the limiting error in \eqref{eq:limit_error_2} is strictly smaller than the limiting error in \eqref{eq:limit_error_1}.

In Figure \ref{fig:alg12}, we plot the result of the simulation with $k_2=3$, which shows that the numerical error from Algorithm \ref{alg:htet} is smaller than that of Algorithm \ref{alg:ht}; both errors closely match theoretical errors in \eqref{eq:limit_error_2} and \eqref{eq:limit_error_1}.

\begin{figure}[ht!]
	\vskip 0.2in
	\begin{center}
		\centerline{\includegraphics[width=250pt]{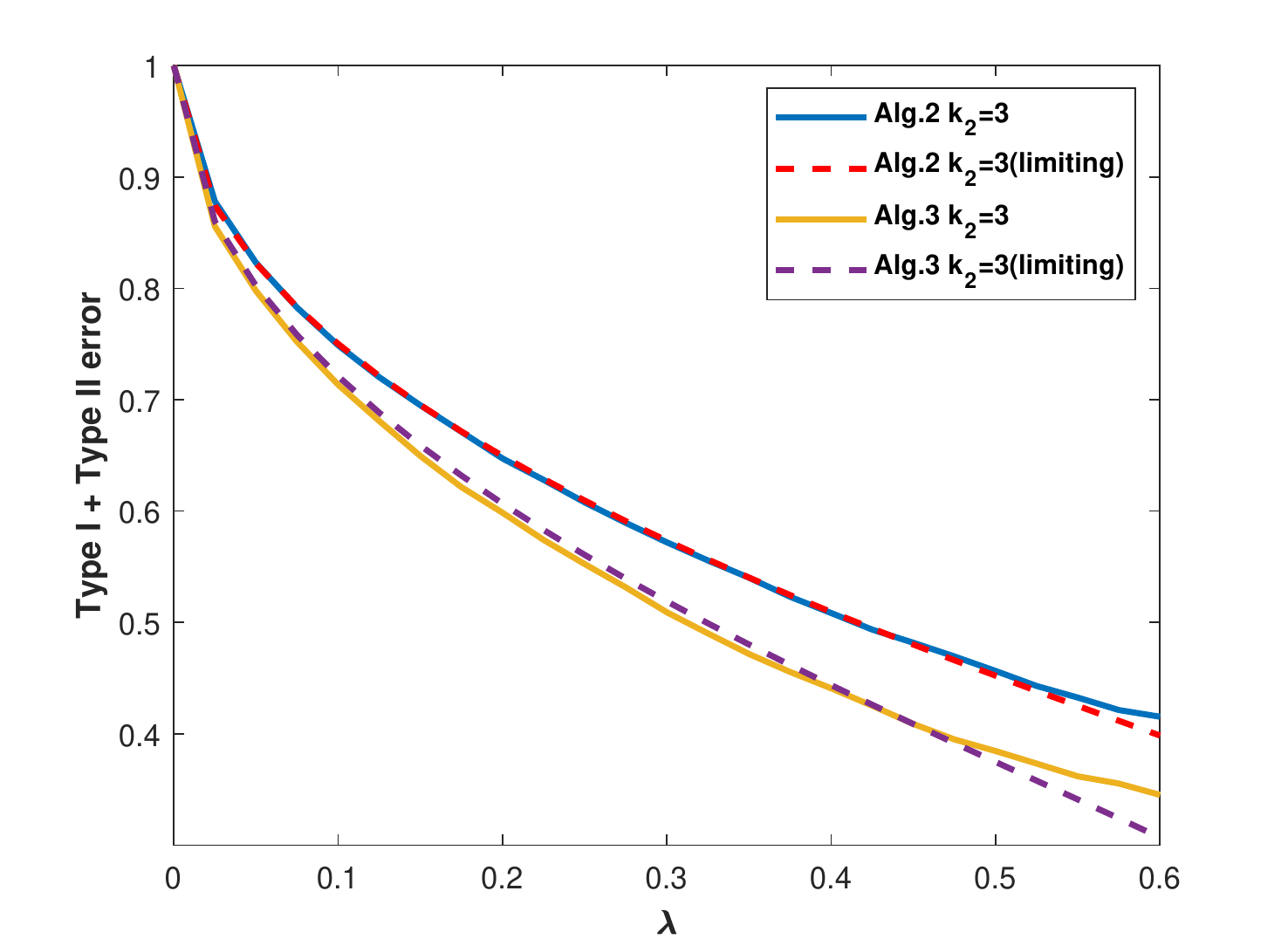}}
		\caption{The errors from the simulation with Algorithm \ref{alg:ht} (blue) and with Algorithm \ref{alg:htet} (yellow), respectively, versus the limiting errors \eqref{eq:limit_error_1} of Algorithm \ref{alg:ht} (red) and \eqref{eq:limit_error_2} of Algorithm \ref{alg:htet} (purple), respectively, for the setting in Section \ref{subsec:sech}.}
		\label{fig:alg12}
	\end{center}
	\vskip -0.2in
\end{figure}

\subsubsection{Rank Estimation} \label{subsec:rank}

We again consider the example in Section \ref{subsec:GOE} and apply Algorithm \ref{alg:at} to estimate the rank of the signal. We again perform the numerical simulation 20,000 samples of the $256 \times 256$ data matrix $M$ with the SNR $\lambda$ varying $0.025$ to $0.6$ and choose the rank of the signal $k$ uniformly from $0$ to $4$. Since we know that the range of the rank $k$ is $[0, 4]$, the (theoretical) limiting error in \eqref{eq:rank_k_error} changes to
\[ \begin{split}
& \p(k=0) \cdot \p \left( Z > \frac{\sqrt{V_0}}{4} \right) + \sum_{i=1}^{3} \p(k=i) \cdot \p \left( |Z| > \frac{\sqrt{V_0}}{4} \right) + \p(k=4) \cdot \p \left( Z > \frac{\sqrt{V_0}}{4} \right) \\
&= \left( 1 - \frac{\p(k=0)+\p(k=4)}{2} \right) \\
&\qquad \times \erfc \left( \frac{1}{4} \sqrt{-\log (1-\lambda) + \left( \frac{2}{w_2} - 1 \right) \lambda + \left( \frac{1}{w_4-1} - \frac{1}{2} \right) \lambda^2} \right).
\end{split} \]

We compute the same test statistic
\beq \begin{split}
	L_{\lambda} = -\log \det \big( (1+\lambda)I - \sqrt{\lambda} M \big) + \frac{\lambda N}{2}
\end{split} \eeq
and find the nearest nonnegative integer of the value
\beq
-\frac{L_{\lambda}}{\log (1-\lambda)} - \frac{1}{2},
\eeq
rounding half down. 
Since $\p(k=0)=\p(k=4) = 0.2$, the limiting error of the estimation is
\beq \label{eq:limit_rank}
\left(1-\frac{\p(k=0)+\p(k=4)}{2}\right) \cdot \erfc \left( \frac{1}{4} \sqrt{-\log (1-\lambda)}\right)=	0.8 \cdot \erfc \left( \frac{1}{4} \sqrt{-\log (1-\lambda)}\right).
\eeq

The result of the simulation can be found in Figure \ref{fig:alg_at}, where we compare the error from the estimation (Algorithm \ref{alg:at}) and the theoretical error in \eqref{eq:limit_rank}. We can see that the error from the numerical simulation matches closely the theoretical error.

\begin{figure}[ht]
	\vskip 0.2in
	\begin{center}
		\centerline{\includegraphics[width=250pt]{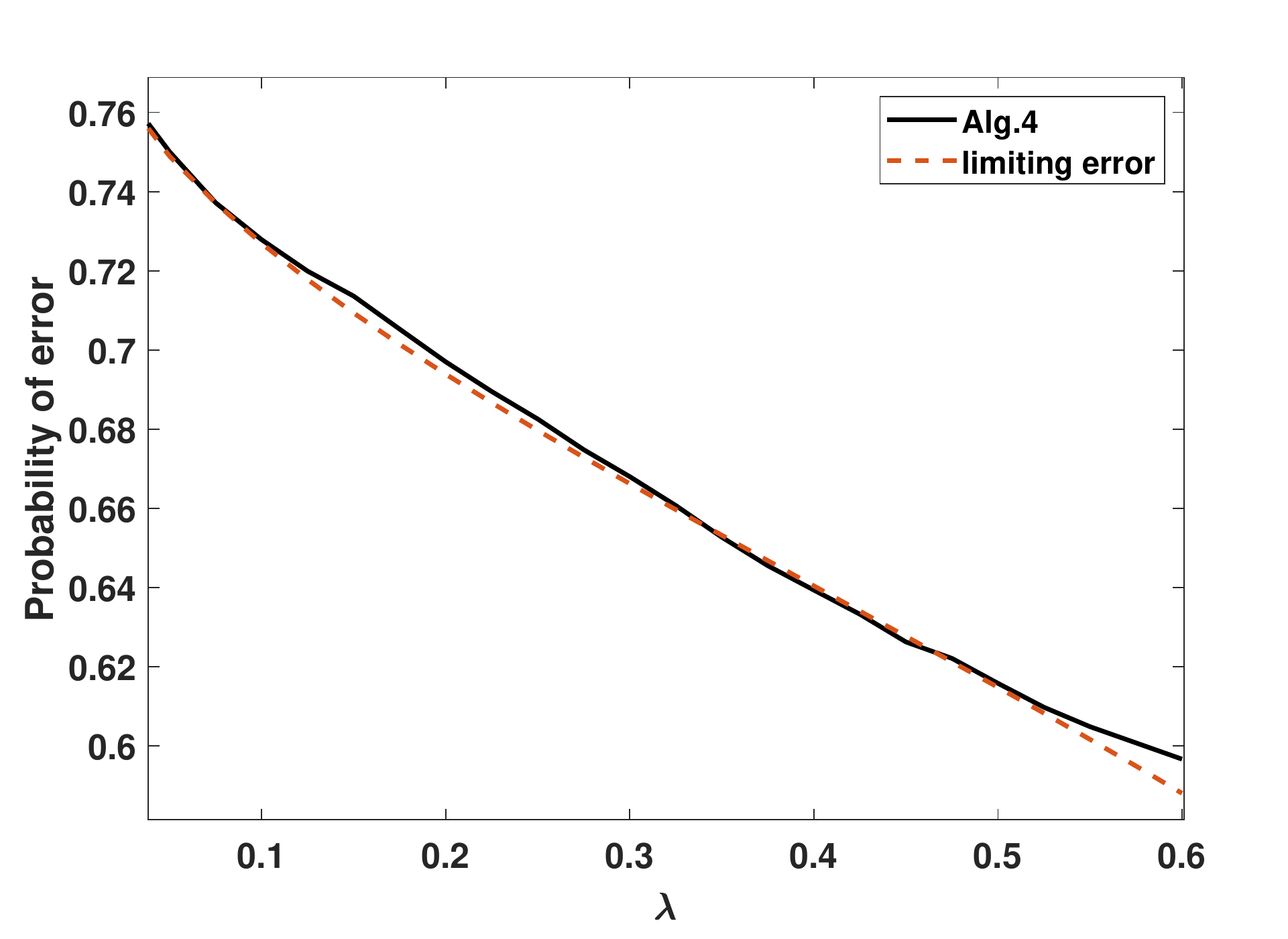}}
		\caption{The errors from the simulation with Algorithm \ref{alg:at} (solid) versus the limiting error \eqref{eq:limit_rank} (dashed) for the setting in Section \ref{subsec:rank}.}
		\label{fig:alg_at}
	\end{center}
	\vskip -0.2in
\end{figure}

\subsection{Spiked rectangular matrices}
In this section, we check the performance of the improved PCA and the pre-transformed LSS-based tests for spiked rectangular matrices.

\subsubsection{Improved PCA with Entrywise Transformation}
\subsubsection*{Additive model}

We consider the data with the non-Gaussian noise whose density function is given by the bimodal distribution in \eqref{eq:bimodal}. We sample $Z_{ij}$ independently from the density $g$ and let $X_{ij}=Z_{ij}/\sqrt{N}$.
We let $\bsu(\ell) = (u_1(\ell), u_2(\ell), \dots, u_M(\ell))^T$ and $\bsv(\ell) = (v_1(\ell), v_2(\ell), \dots, v_N(\ell))^T$, where $\sqrt{M} u_i(\ell)$'s and $\sqrt{N} v_j(\ell)$'s are i.i.d. Rademacher random variables for $i=1, 2, \dots, M,$ $j=1, 2, \dots, N$ and $\ell=1,2,3$. 
When we apply the entrywise transformation, defined in~\eqref{eq:transformed-cov}, with $\alpha=0$ to the rank-3 spiked mean data matrix, we get
\beq
\wt Y_{ij} = \frac{1}{\sqrt{\fh N}} h(\sqrt{N} Y_{ij})
\eeq
where 
\beq
h(x)=-\frac{g'(x)}{g(x)}=\frac{2\left(\sqrt{3}-e^{4\sqrt{3}x\left(\sqrt{3}-2x\right)}+2x\right)}{1+e^{4\sqrt{3}x}}
\eeq
and $\fh = \int_{-\infty}^{\infty} \frac{(g'(x))^2}{g(x)} \dd x\approx 2.50810$.
The size of the data matrix is set to be $M=2000$, $N=4000$, and the ratio $\rat=M/N=0.5$.

Theoretically, the threshold for the BBP-transition of the largest eigenvalue is $\sqrt{\rat}\approx 0.7071$ with the vanilla PCA, whereas the threshold is lowered to $\frac{\sqrt{\rat}}{\fh} \approx 0.2819$ with the improved PCA as predicted by Theorem~\ref{thm:trans-mean}.

For $\ell=1,2,3,$ we set the SNRs 
\beq\label{eq:lam_sim}
\lambda_\ell=\frac{\ell\sqrt{\rat}+\frac{\sqrt{\rat}}{\fh}}{\ell+1}
\eeq
to observe the transitions of the largest eigenvalue after the transformation. In Figure~\ref{fig:spectrum}, we compare the spectrum of the sample covariance matrices, $YY^T$ (left) and $\wt Y \wt Y^T$ (right). As in the spiked Wigner case in Section \ref{subsubsec:Wigner_PCA}, we again find three outlier eigenvalues only in the spectrum of $\wt Y \wt Y^T$ (right), which are absent in that of $YY^T$ (left).
\begin{figure}[h]
	\vskip 0.2in
	\begin{center}
		\centerline{\includegraphics[width=\columnwidth]{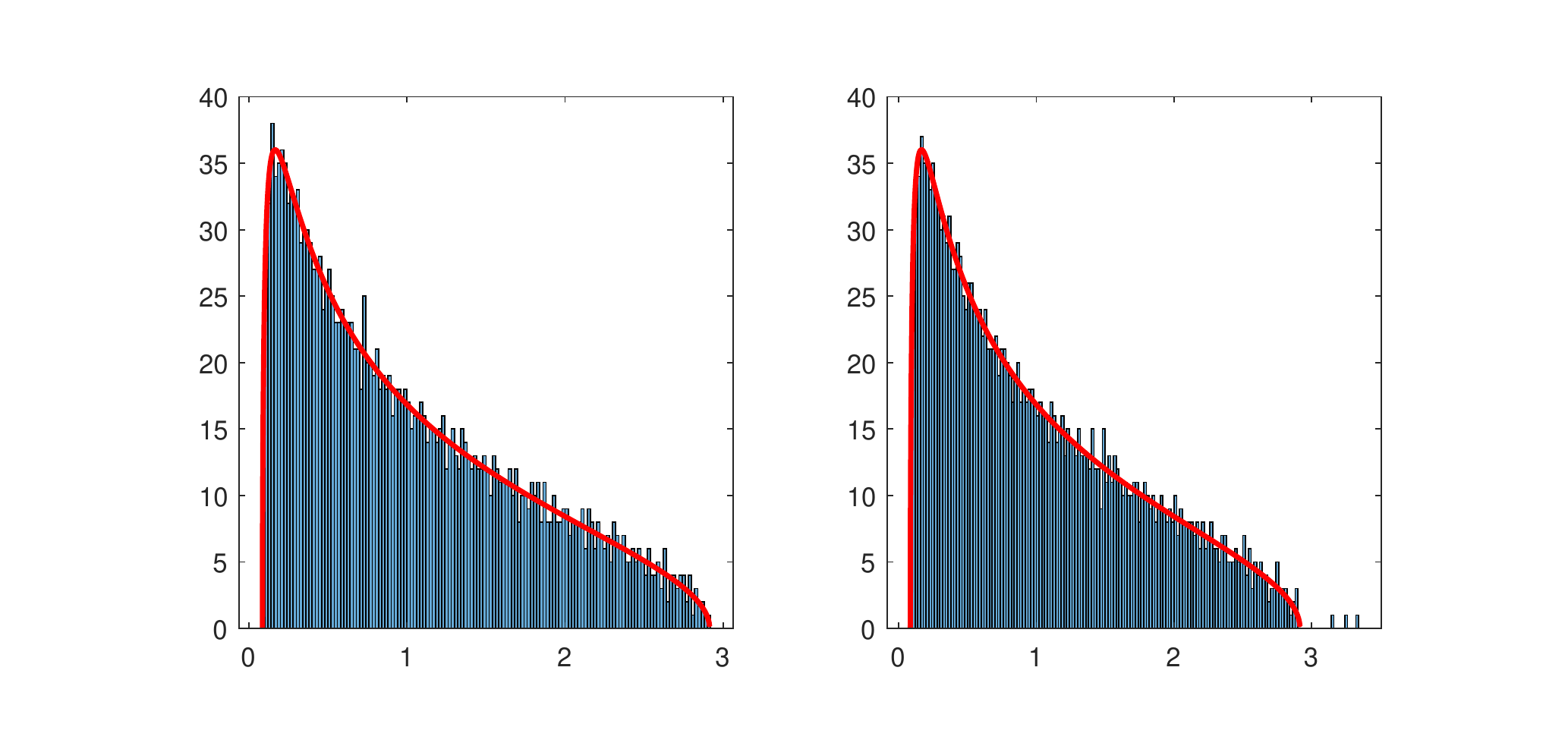}}
		\caption{The spectrum of the sample covariance matrix ($M=2000, N=4000$) with bimodal noise, before (left) and after (right) the entrywise transformation. Three eigenvalues pop up from the bulk of the spectrum after the entrywise transformation.}
		\label{fig:spectrum}
	\end{center}
	\vskip -0.2in
\end{figure}

\subsubsection*{Multiplicative model}
In the spiked covariance model, to clearly observe the outlier in our simulation setting, a distribution with a larger Fisher information value should be used. 
Thus, we let the density function $g_a$ of the noise be the generalized version of the bimodal distribution with unit variance in \eqref{eq:bimodal}, defined as
\[
g_a(x)=\frac{1}{2\sqrt{2(1-a^2)\pi}}\left(e^{-\frac{(x-a)^2}{2(1-a^2)}}+e^{-\frac{(x+a)^2}{2(1-a^2)}}\right),
\]
which $g_a$ is the density function of a random variable
\[
\sqrt{1-a^2} \caN + a\caR.
\]

We sample $Z_{ij}$ independently from the density $g_a$ and let $X_{ij}=Z_{ij}/\sqrt{N}$. We let $\bsu(\ell) = (u_1(\ell), u_2(\ell), \dots, u_M(\ell))^T$ and $\bsv(\ell) = (v_1(\ell), v_2(\ell), \dots, v_N(\ell))^T$, where $\sqrt{M} u_i(\ell)$'s and $\sqrt{N} v_j(\ell)$'s are i.i.d. Rademacher random variables for $i=1, 2, \dots, M,$ $j=1, 2, \dots, N$ and $\ell=1,2,3$. 
When we apply the entrywise transformation, defined in~\eqref{eq:transformed-cov} to the rank-3 spiked covariance data matrix, we get
\beq
\wt Y_{ij} = \frac{1}{\sqrt{(\alpha^2+2\alpha+\fh) N}} h_{a,\alpha}(\sqrt{N} Y_{ij})
\eeq
where 
\beq
h_{a,\alpha}(x)=-\frac{g_a'(x)}{g_a(x)}+\alpha x=\frac{\left((x-a)e^{\frac{2ax}{1-a^2}}+(x+a)\right)}{(1-a^2)(1+e^{\frac{2ax}{1-a^2}})}+\alpha x
\eeq
and $\fh = \int_{-\infty}^{\infty} \frac{(g_a'(x))^2}{g_a(x)} \dd x\approx 5.15583$, when $a=\sqrt{21}/5.$
The size of the data matrix is set to be $M=4000$, $N=8000$. We also use $\alpha=\sqrt{\fh}$, and the ratio $d_0=M/N=0.5$.
The threshold for the BBP-transition of the largest eigenvalue is $\sqrt{d_0}\approx 0.7071$ for the vanilla PCA, whereas the threshold changes to $\lambda_{g,\ell}=\frac{(1+\sqrt{\fh})}{2}\cdot(2\gamma_\ell+\sqrt{\fh}\gamma_\ell^2)=\sqrt{\rat}$ for the transformed PCA. (See Theorem~\ref{thm:trans-cov}.)
For $\ell=1,2,3$, we set the SNRs
\beq\label{eq:lam_sim_mult}
\lambda_\ell=\frac{\ell\sqrt{d_0}+\frac{2\sqrt{\rat}}{1+\sqrt{\fh}}}{\ell+1}
\eeq
to observe the transitions of the largest eigenvalue after the transformation.  

We obtain a result analogous to the additive model. See Figure~\ref{fig:spectrum_mult}.
\begin{figure}[h]
	\vskip 0.2in
	\begin{center}
		\centerline{\includegraphics[width=\columnwidth]{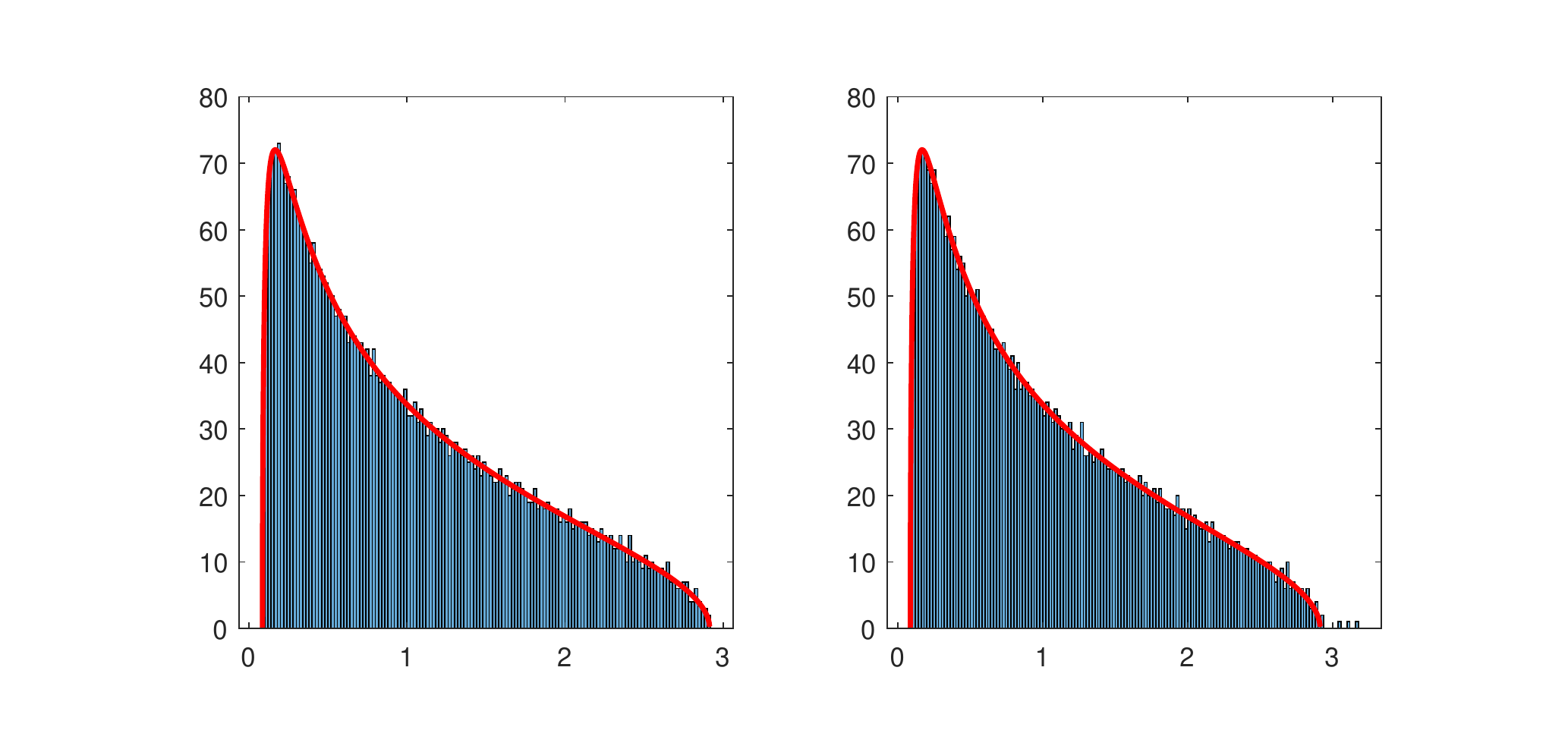}}
		\caption{The spectrum of the sample covariance matrix ($M=4000, N=8000$) with bimodal noise, before (left) and after (right) the entrywise transformation. Three eigenvalues pop up from the bulk of the spectrum after the entrywise transformation.}
		\label{fig:spectrum_mult}
	\end{center}
	\vskip -0.2in
\end{figure}

\subsubsection{Hypothesis Testing with pre-transformed LSS estimator}\label{sec:Improved Test}

We now consider an (additive) spiked rectangular matrix with the non-Gaussian noise whose density function is given by \eqref{eq:sech}. We let the signal $\bsu = (u_1, u_2, \dots, u_M)^T$ and $\bsv = (v_1, v_2, \dots, v_N)^T$, where $\sqrt{M} u_i$'s and $\sqrt{N} v_j$'s are i.i.d. Rademacher random variables for $i=1, 2, \dots, M$ and $j=1, 2, \dots, N$. Let the data matrix $Y = \sqrt{\lambda} \bsu \bsv^T + X$. 

Recall that $w_4=5,$ $\fh = \frac{\pi^2}{8}$, $\gh = \frac{\pi^2}{16}$, and $\wt w_4 = \frac{3}{2}$. 
The LSS estimators are given by 
\beq \begin{split}
	L_{\SNR} &= -\log \det \left( \left(1+\frac{\rat}{\SNR} \right)(1+\SNR)I - YY^T \right) -\frac{\SNR}{2\rat}(\Tr YY^T-M)\\
	&~~~+M \left[\frac{\SNR}{\rat} - \log\left(\frac{\SNR}{\rat}\right) -\frac{1-\rat}{\rat}\log(1+\SNR) \right],
\end{split} \eeq
and 
\beq \begin{split}
	\wt L_\SNR &= -\displaystyle\log \det \left( \left(1+\frac{8\rat}{\SNR\pi^2}\right)(1+\SNR \frac{\pi^2}{8})I - \wt{Y}\wt{Y}^T \right) + \frac{\pi^2\SNR}{8\rat} (\Tr \wt{Y}\wt{Y}^T-M)
	\\&~~~+M\left[\frac{\SNR\pi^2}{8\rat} - \log\left(\frac{\SNR\pi^2}{8\rat}\right) -\frac{1-\rat}{\rat}\log\left(1+\SNR \frac{\pi^2}{8}\right)\right].
\end{split} \eeq
With critical values $m_\SNR=-\log\left(1-\frac{\SNR^2}{\rat}\right)+\frac{3\SNR^2}{4\rat}$ and $\wt m_\SNR=-\log\left(1-\frac{\SNR^2\pi^4}{64\rat}\right)-\frac{3\pi^4\SNR^2}{256\rat,}$ the errors are
\[
\erfc\left(\frac1{4}\sqrt{-\log\left(1-\frac{\SNR^2}{\rat}\right)-\frac{\SNR^2}{2\rat}}\right)
\]
and 
\[
\erfc\left(\frac1{4}\sqrt{-\log\left(1-\frac{\pi^4\SNR^2}{64\rat}\right)}\right).
\]
In Figure \ref{fig:sech_rec}, we plot empirical average (after 1,000 Monte Carlo simulations) of the error of the proposed test and the theoretical (limiting) error, varying the SNR $\SNR$ from $0$ to $0.5$, with $M=256$ and $N=512$. It can be checked that the error of the proposed test closely matches the theoretical error.
\begin{figure}[h]
	\vskip 0.2in
	\begin{center}
		\centerline{\includegraphics[width=200pt]{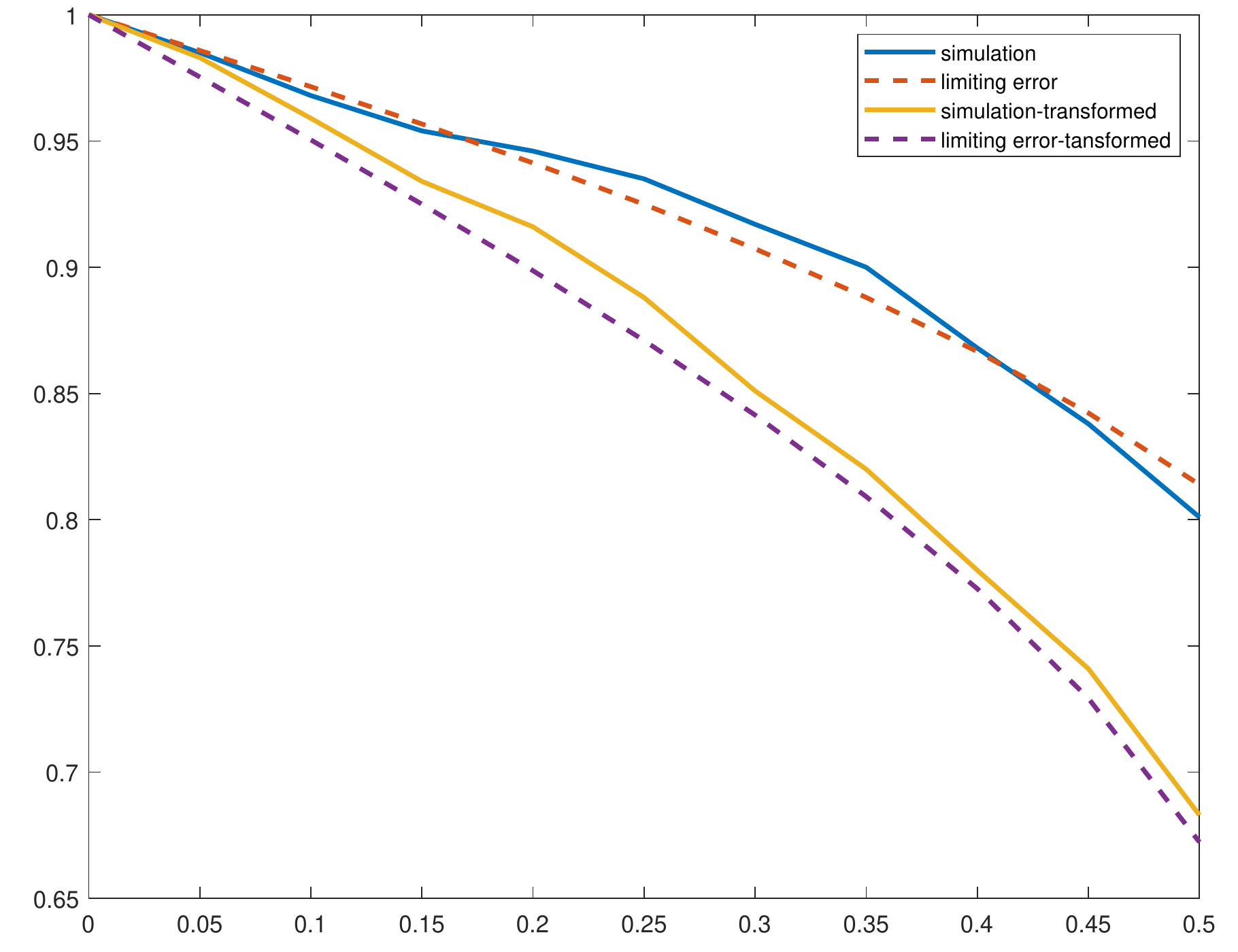}}
		\caption{The error from the simulation (solid) and the theoretical limiting error (dashed), respectively.}
		\label{fig:sech_rec}
	\end{center}
	\vskip -0.2in
\end{figure}

\section{Proof of Theorems for improved PCA} \label{sec:proof_trans}

In this section, we rigorously prove Theorems \ref{thm:trans-mean} and \ref{thm:trans-cov} in Section \ref{sec:main1}, which are about the detection threshold of the improved PCA.

\subsection{Preliminaries}
We first introduce the following notions, which provide a simple way of making precise statements regarding the bound up to small powers of $N$ that holds with probability higher than $1-N^{-D}$ for all $D>0$.

\begin{defn}[Overwhelming probability]\label{def:overwhelming}
	We say that an event (or family of events) $\Omega$ holds with overwhelming probability if for all (large) $D>0$ we have $\P(\Omega) \leq N^{-D}$ for any sufficiently large $N$.
\end{defn}

\begin{defn}[Stochastic domination]\label{def:stocdom}
	Let
	\begin{equation*}
	\xi = \pb{\xi^{(N)}(u) \;:\; N \in \N, u \in U^{(N)}} \,, \qquad
	\zeta = \pb{\zeta^{(N)}(u) \;:\; N \in \N, u \in U^{(N)}}
	\end{equation*}
	be two families of random variables, where $U^{(N)}$ is a possibly $N$-dependent parameter set. 
	We say that $\xi$ is \emph{stochastically dominated by $\zeta$, uniformly in $u$,} if for all (small) $\epsilon > 0$ and (large) $D > 0$
	\begin{equation*}
	\sup_{u \in U^{(N)}} \P \left({|\xi^{(N)}(u)| > N^\epsilon \zeta^{(N)}(u)}\right) \;\leq\; N^{-D}
	\end{equation*}
	for any sufficiently large $N\geq N_0(\varepsilon, D)$.
	Throughout this appendix, the stochastic domination will always be uniform in all parameters, including matrix indices and the spectral parameter $z$.
	
	We write $\xi \prec \zeta$ or $\xi = \caO_\prec(\zeta)$, if $\xi$ is stochastically dominated by $\zeta$, uniformly in $u$.
\end{defn}


For a Wigner matrix $W$, we will use the following result for the resolvents, which is called an isotropic local semicircle law.
\begin{lem}[Isotropic local semicircle law] \label{lem:local_SC}
	Suppose that $z \in \R$ outside an open interval containing $[-2, 2]$. Let $s_{sc}(z)$ be the Stieltjes transform of the Marchenko--Pastur law, which is also given by
	\beq
	s_{sc}(z)=\frac{-z + \sqrt{z^2 - 4}}{2}.
	\eeq
	Then,
	\[
	\langle \bsu(\ell_1), (W-zI)^{-1} \bsu(\ell_2) \rangle = s_{sc}(z) \langle \bsu(\ell_1),\bsu(\ell_2)\rangle + \caO_\prec(N^{-\frac{1}{2}})
	\]
\end{lem}
See Theorem 2.3 of \cite{Knowles-Yin2013} (also Lemma 7.7 of \cite{chung2019weak}) for the proof of Lemma \ref{lem:local_SC}.

Further, for a rectangular matrix $X$, we will use the following analogous result for the resolvents, which is called an isotropic Marchenko--Pastur law.
\begin{lem}[Isotropic local Marchenko--Pastur law] \label{lem:local_MP}
	Suppose that $z \in \R$ outside an open interval containing $[d_-, d_+]$. Let $s(z)$ be the Stieltjes transform of the Marchenko--Pastur law, which is also given by
	\beq
	s(z) = \frac{(1-\rat-z)+ \sqrt{(1-\rat-z)^2 -4\rat z}}{2\rat z}.
	\eeq
	Then,
	\[
	\langle \bsv(\ell_1), (X^T X-zI)^{-1} \bsv(\ell_2) \rangle = -\left( \frac{1}{zs(z)}+1 \right) \langle \bsv(\ell_1),\bsv(\ell_2)\rangle + \caO_\prec(N^{-\frac{1}{2}})
	\]
	and
	\[	
	\langle X^T \bsu(\ell_1), (X^T X-zI)^{-1} X^T\bsu(\ell_2)\rangle = (zs(z)+1) \langle \bsu(\ell_1),\bsu(\ell_2)\rangle +\caO_\prec(N^{-\frac{1}{2}}).
	\]
\end{lem}
See Theorem 2.5 of \cite{Bloemendal-Erdos-Knowles-Yau-Yin2014} (also Lemma 3.7 of \cite{BKYY2016}) for the proof of Lemma \ref{lem:local_MP}.

The following concentration inequality will be frequently used in the proof, which is sometimes called the large deviation estimate in random matrix theory.
\begin{lem}[Large deviation estimate] \label{lemma: LDE}
	Let $\pb{\xi_i^{(N)}}$ and $\pb{\zeta_i^{(N)}}$ be independent families of random variables and 
	$\pb{a_{ij}^{(N)}}$ and $\pb{b_i^{(N)}}$ be deterministic; here $N \in \N$ and $i,j = 1, \dots, N$. Suppose that complex-valued random variables $\xi_i^{(N)}$ and $\zeta_i^{(N)}$ are independent and satisfy for all $p\geq 2$ that
	\beq\label{cond on X}
	\E \xi \;=\; 0\,, \qquad \E \abs{\xi}^p \;\leq\; \frac{C_p}{NB^{p-2}}
	\eeq
	for some $B \leq N^{1/2}$ and some ($N$-independent) constant $C_p$. Then we have the bounds
	\begin{align} \label{LDE}
	\sum_i b_i \xi_i &\;\prec\; \pbb{\frac1{N}\sum_i \abs{b_i}^2}^{1/2}+\frac{\max_i|b_i|}{B}\,,
	\\ \label{two-set LDE}
	\sum_{i,j} a_{ij} \xi_i \zeta_j &\;\prec\; \pbb{\frac1{N^2}\sum_{i\neq j} \abs{a_{ij}}^2}^{1/2}+\frac{\max_{i\neq j}|a_{ij}|}{B}+\frac{\max_{i}|a_{ii}|}{B^2}\,,
	\\ \label{offdiag LDE}
	\sum_{i \neq j} a_{ij} \xi_i \xi_j &\;\prec\; \pbb{\frac1{N^2}\sum_{i \neq j} \abs{a_{ij}}^2}^{1/2}+\frac{\max_{i\neq j}|a_{ij}|}{B}\,.
	\end{align}
	If the coefficients $a_{ij}^{(N)}$ and $b_i^{(N)}$ depend on an additional parameter $u$, then all of these estimates are uniform in $u$, i.e. $N_0= N_0(\varepsilon, D)$ in the definition of $\prec$ depends not on $u$ but only on the constant $C$ from \eqref{cond on X}. 
	
	If $B = N^{1/2}$, the bounds can further be simplified to
	\beq
	\sum_i b_i \xi_i \prec \pbb{\frac1{N}\sum_i \abs{b_i}^2}^{1/2}, \quad \sum_{i,j} a_{ij} \xi_i \zeta_j \prec \pbb{\frac1{N^2}\sum_{i, j} \abs{a_{ij}}^2}^{1/2}, \quad
	\sum_{i \neq j} a_{ij} \xi_i \xi_j \prec\pbb{\frac1{N^2}\sum_{i \neq j} \abs{a_{ij}}^2}^{1/2}.
	\eeq
\end{lem}

\begin{proof}
	These estimates are an immediate consequence of Lemma 3.8 in \cite{Erdos-Knowles-Yau-Yin2013a}.
\end{proof}
Finally, we recall that, for our prior, $|\langle\bsu(\ell_1),\bsu(\ell_2)\rangle-\delta_{\ell_1\ell_2}|,\,|\langle\bsv(\ell_1),\bsv(\ell_2)\rangle-\delta_{\ell_1\ell_2}|\prec N^{-\phi}.$

\subsection{Proof of Theorem \ref{thm:trans-wigner}} \label{sec:proof_trans-wigner}
We first prove the behavior of the $k$ largest eigenvalues described in Section \ref{subsec:BBP}, which we will call the BBP result, in our setting, following the strategy of \cite{Raj2011,benaych2012singular}.
\beq \begin{split}
	M -zI &=W+  \bsU\Lambda^{1/2} \bsU^T - zI \\
	&= (W-zI)(I + (W-zI)^{-1} ( \bsU \Lambda^{1/2}\bsU^T)).
\end{split} \eeq
Thus, if $z$ is an eigenvalue of $M$ but not of $W$, then it satisfies
\[
\det (I + (W-zI)^{-1}  \bsU \Lambda^{1/2}\bsU^T) = 0,
\]
which also implies that $-1$ is an eigenvalue of 
\[
T \equiv T(z) := (W-zI)^{-1}  \bsU \Lambda^{1/2}\bsU^T.
\]
We then see that
\[\begin{split}
T(W-zI)^{-1}\bsu(\ell)&=(W-zI)^{-1}  \bsU\Lambda^{1/2} \bsU^T(W-zI)^{-1}\bsu(\ell)\\
&=\sqrt{\lambda_\ell}\langle\bsu(\ell),(W-zI)^{-1}\bsu(\ell)\rangle(W-zI)^{-1}\bsu(\ell)+O_\prec(N^{-\phi}).
\end{split}\]
i.e., $(W-zI)^{-1}\bsu(\ell)$ must be a eigenvector for $T$ with the corresponding eigenvalue $-1.$ Thus, by Lemma \ref{lem:local_SC},
\[
\sqrt{\lambda_\ell}s_{sc}(z)=-1+O_\prec(N^{-\phi}).
\]
It is elementary to check that the solution of the above equation is $z=\sqrt{\lambda_\ell}+\frac1{\sqrt{\lambda_\ell}}+O_\prec(N^{-\phi})$ if and only if $\lambda_\ell>1.$

We now turn to the proof of Theorem \ref{thm:trans-wigner}. For the spike $\|\bsU\|_\infty\prec N^{-\phi}$, suppose that a function $q$ and its all derivatives are polynomially bounded in the sense of Assumption \ref{assump:entry1}. Following the proof of Theorem 4.8 in \cite{Perry2018}, we have the following local linear estimation of $q(\sqrt{N} M_{ij})$ by
\[
q(\sqrt{N}M_{ij})=q(\sqrt{N}W_{ij})+\sqrt{\lambda N} \bsu_i \bsu_j^T \E[q'(\sqrt{N}W_{ij})]+\mathcal{R}_{ij},
\]
where the error $\caR_{ij}$ is negligible.
Set
\[
\mf := \E [q'(\sqrt{N}W_{ij})], \quad \vf := \E [q(\sqrt{N}W_{ij})^2], \quad \widehat\lambda := \lambda \mf^2/\vf,
\]
and
\[
Q_{ij} := \frac{1}{\sqrt{N\vf}}q(\sqrt{N}W_{ij}).
\]
Then the spectrum of the transformed matrix is determined by the matrix $Q+\bsU\hat\Lambda^{1/2}\bsU^T$. Since $Q$ is also Wigner matrix with $N\E[Q_{ij}^2]=1,$ by repeating the same process above, we get the result.

\subsection{Proof of Theorem \ref{thm:trans-mean}} \label{sec:proof_trans-mean}

We first prove the behavior of the $k$ largest eigenvalues described in Section \ref{subsec:BBP}, which we will again call the BBP result, in our setting, following the strategy of \cite{Raj2011,benaych2012singular}. Note that the $k$ largest eigenvalue of $YY^T$ is equal to the $k$ largest eigenvalues of $Y^T Y$. Consider the identity
\beq \begin{split}
	Y^T Y -zI &= (X+  \bsU \Lambda^{1/2}\bsV^T)^T (X+  \bsU \Lambda^{1/2}\bsV^T) - zI =(X^T X-zI)T(z)
\end{split} \eeq
where
\[
T \equiv T(z) := (X^T X-zI)^{-1} ( X^T \bsU\Lambda^{1/2} \bsV^T +  \bsV \Lambda^{1/2}\bsU^T X + \bsV\Lambda^{1/2}\bsU^T\bsU\Lambda^{1/2} \bsV^T).
\]
Thus, if $z$ is an eigenvalue of $YY^T$ but not of $XX^T$, then it satisfies
\[
\det (T(z)) = 0,
\]
which also implies that $-1$ is an eigenvalue of $T(z).$

Note that since $\| X \|, \| (X^T X-zI)^{-1} \| \prec 1$, from Lemma \ref{lemma: LDE}, 
\[ \begin{split}
&\langle \bsb, (X^T X-zI)^{-1} X^T \bsa \rangle = \sum_{i, j} \left( (X^T X-zI)^{-1} X^T \right)_{ij} b_i a_j \\
&\prec \left( \frac{1}{N^2} \sum_{i \neq j} \left| \left( (X^T X-zI)^{-1} X^T \right)_{ij} \right|^2 \right)^{1/2} + N^{-\phi} \max_{i, j} \left| \left( (X^T X-zI)^{-1} X^T \right)_{ij} \right| \\
&\prec \left( \frac{1}{N} \| (X^T X-zI)^{-1} X^T \|^2 \right)^{1/2} + N^{-\phi} \| (X^T X-zI)^{-1} X^T \| \prec N^{-\phi}.
\end{split} \]

Then the matrix $T$ satisfies
\[ \begin{split}
&T \cdot (X^T X-zI)^{-1} X^T \bsu (\ell)\\
&= (X^T X-zI)^{-1} X^T \bsU\Lambda^{1/2} (\bsV^T (X^T X-zI)^{-1} X^T \bsu(\ell)) \\
& \qquad + (X^T X-zI)^{-1} \bsV \Lambda^{1/2} (\bsU^T X(X^T X-zI)^{-1} X^T \bsu(\ell))  \\
& \qquad +   (X^T X-zI)^{-1} \bsV \Lambda^{1/2}(\bsU^T\bsU)\Lambda^{1/2} (\bsV^T (X^T X-zI)^{-1} X^T \bsu(\ell))
\\&=\sqrt{\lambda_\ell} \langle\bsu(\ell), X(X^T X-zI)^{-1} X^T \bsu(\ell)\rangle(X^T X-zI)^{-1}\bsv(\ell) + \bsth_1(\ell)
\end{split} \]
and
\[ \begin{split}
&T \cdot (X^T X-zI)^{-1} \bsv(\ell) \\
&=  (X^T X-zI)^{-1} X^T \bsU\Lambda^{1/2} (\bsV^T (X^T X-zI)^{-1} \bsv(\ell)) \\
& \qquad + (X^T X-zI)^{-1} \bsV \Lambda^{1/2} (\bsU^T X(X^T X-zI)^{-1} \bsv(\ell))  \\
& \qquad +    (X^T X-zI)^{-1} \bsV\Lambda^{1/2} (\bsU^T\bsU)\Lambda^{1/2}(\bsV^T (X^T X-zI)^{-1}\bsv(\ell)) 
\\&=\sqrt{\lambda_\ell} \langle\bsv(\ell), (X^T X-zI)^{-1} \bsv(\ell)\rangle(X^T X-zI)^{-1}X^T\bsu(\ell)
\\&\qquad+ \lambda_\ell  \langle\bsv(\ell), (X^T X-zI)^{-1} \bsv(\ell)\rangle (X^T X-zI)^{-1} \bsv(\ell) + \bsth_2(\ell)
\end{split} \]
where $\| \bsth_1(\ell) \|,\|\bsth_2(\ell)\| = \caO_{\prec}(N^{-\phi})$ since $\|\bsU^T\bsU- I\|_F\prec N^{-\phi}.$ 

In particular, $k$ extremal eigenvectors of $T$ are a linear combination of $(X^T X-zI)^{-1} X^T \bsu(\ell)$ and $(X^T X-zI)^{-1} \bsv(\ell)$.

Suppose that $a_\ell (X^T X-zI)^{-1} X^T \bsu(\ell) + b_\ell (X^T X-zI)^{-1} \bsv(\ell)$ is an eigenvector of $T$ with the corresponding eigenvalue $-1$. 
Thus, from Lemma \ref{lem:local_MP},
\beq \begin{split} \label{eq:ab_equation}
	&-\left( a_\ell (X^T X-zI)^{-1} X^T \bsu(\ell) + b_\ell (X^T X-zI)^{-1} \bsv(\ell) \right) 
	\\&= T \left( a_\ell (X^T X-zI)^{-1} X^T \bsu(\ell) + b_\ell (X^T X-zI)^{-1} \bsv(\ell) \right) \\
	&= -b_\ell\sqrt{\lambda_\ell} \left( \frac{1}{zs(z)}+1 \right) (X^T X-zI)^{-1} X^T \bsu(\ell) \\
	&\qquad + a_\ell\sqrt{\lambda_\ell} (zs(z)+1) (X^T X-zI)^{-1} \bsv(\ell) - b_\ell\lambda_\ell \left( \frac{1}{zs(z)}+1 \right) (X^T X-zI)^{-1} \bsv(\ell) + \wt\bsth(\ell)
\end{split} \eeq
for some $\wt\bsth(\ell)$, which is a linear combination of $(X^T X-zI)^{-1} X^T \bsu(\ell)$ and $(X^T X-zI)^{-1} \bsv(\ell)$ with $\| \wt\bsth(\ell) \| = \caO_{\prec}(N^{-\phi})$.

Since $\bsU$, $\bsV$, and $X$ are independent, $(X^T X-zI)^{-1} X^T \bsu(\ell)$ and $(X^T X-zI)^{-1} \bsv(\ell)$ are linearly independent with overwhelming probability. Thus, from \eqref{eq:ab_equation}, 
\[ \begin{split}
-a_\ell &= -b_\ell\sqrt{\lambda_\ell} \left( \frac{1}{zs(z)}+1 \right) + \caO_{\prec}(N^{-\phi}), \\
-b_\ell &= a_\ell\sqrt{\lambda_\ell} (zs(z)+1) - b_\ell\lambda_\ell \left( \frac{1}{zs(z)}+1 \right) + \caO_{\prec}(N^{-\phi}).
\end{split} \]
It is then elementary to check that
\[
\lambda_\ell (zs(z)+1) +1 = \caO_{\prec}(N^{-\phi}),
\]
which has the solution
\[
z = (1+\lambda_\ell)\left(1+ \frac{\rat}{\lambda_\ell}\right) + \caO_{\prec}(N^{-\phi})
\]
if and only if $\lambda_\ell > \sqrt{\rat}$. This proves the BBP result in our setting.

We now turn to the proof of Theorem \ref{thm:trans-mean}. To simplify the exposition, we focus on the case that SNRs are the same i.e., $\Lambda=\lambda I$. For the spike prior in Assumption \ref{assump:entry1}, suppose that a function $q$ and its all derivatives are polynomially bounded in the sense of Assumption \ref{assump:entry1}. Following the proof of Theorem 4.8 in \cite{Perry2018}, we define the error term from the local linear estimation of $q(\sqrt{N} Y_{ij})$ by
\[
q(\sqrt{N}Y_{ij})=q(\sqrt{N}X_{ij})+\sqrt{\lambda N} \bsu_i \bsv_j^T q'(\sqrt{N}X_{ij})+\mathcal{R}_{ij}
\]
where 
\[
\mathcal{R}_{ij}=\frac{1}{2} q''(\sqrt{N}X_{ij}+e_{ij}) \lambda N (\bsu_i \bsv_j^T)^2
\] 
for some $|e_{ij}|\leq|\sqrt{\lambda N}\bsu_i \bsv_j^T|$. The Frobenius norm of $\mathcal{R}$ is bounded as
\[ \begin{split}
\|\mathcal{R}\|^2_F &= \Tr \mathcal{R}^T \mathcal{R}=\frac{\lambda^2 N^2}{4}\sum_{i=1}^{M}\sum_{j=1}^{N}(\bsu_i\bsv_j^T)^4 q''(\sqrt{N}X_{ij}+e_{ij})^2 \\
&\leq\frac{\lambda^2 N^{2-4\phi}}{4}\sum_{i=1}^{M}\sum_{j=1}^{N}(\bsu_i\bsv_j^T)^2 q''(\sqrt{N}X_{ij}+e_{ij})^2.
\end{split} \]
Since $q''$ is polynomially bounded, $q''(\sqrt{N}X_{ij}+e_{ij})$ is uniformly bounded by an $N$-independent constant. Thus, with overwhelming probability,
\[
\|\mathcal{R}\|^2\leq\|\mathcal{R}\|^2_F\leq C\lambda^2N^{2-4\phi}.
\]

Next, we approximate $q(\sqrt{N}X_{ij})$ by its mean. Let 
\[ \label{eq:local_linear}
\mathcal{E}_{ij}=q'(\sqrt{N}X_{ij})-\E [q'(\sqrt{N}X_{ij})], \qquad \Delta_{ij}=\sqrt{\lambda N}\bsu_i \bsv_j^T \mathcal{E}_{ij}.
\]
Then, $\|\Delta\|\prec N^{\frac12-2\phi}\|\mathcal{E}\|$ and, since the entries of matrix $\mathcal{E}$ are i.i.d., centered and with finite moments, its norm $\|\mathcal{E}\|=O(\sqrt{N})$ with overwhelming probability. (See, e.g., \cite{BKYY2016}.) Thus, $\| \Delta \| = \caO_{\prec}(N^{1-2\phi})$.

Set
\[
\mf := \E [q'(\sqrt{N}X_{ij})], \quad \vf := \E [q(\sqrt{N}X_{ij})^2], \quad \widehat\lambda := \lambda \mf^2/\vf,
\]
and
\[
Q_{ij} := \frac{1}{\sqrt{N\vf}}q(\sqrt{N}X_{ij}).
\]
We have proved so far that the difference of the largest eigenvalue of $Q + \widehat\lambda^{\frac{1}{2}} \bsU \bsV^T$ and that of the matrix
\[
\left( \frac{1}{\sqrt{N\vf}}q(\sqrt{N}Y_{ij}) \right)
\]
is $\caO_{\prec}(N^{\frac12-2\phi})$, which is $o(1)$ with overwhelming probability for $\phi > \frac{1}{4}$. It is directly applicable to the case that $\Lambda$ in our model with $\widehat\Lambda_{\ell\ell} := \lambda_\ell \mf^2/\vf$ since the above process does not require any information of the SNRs. The BBP result holds the matrix $Q +  \bsU \widehat\Lambda^{\frac{1}{2}}\bsV^T$, which is another (additive) spiked rectangular matrix. This shows that the BBP result also holds for $\wt Y$ with SNR matrix $\widehat\Lambda := \frac{\mf^2}{\vf}\Lambda $. This proves Theorem \ref{thm:trans-mean}. 


\subsection{Proof of Theorem \ref{thm:trans-cov}}\label{sec:proof_trans-cov}

Recall that the spike prior satisfies the technical conditions in Assumption \ref{assump:entry1} with $\phi>1/4$. For the sake of brevity, we assume that $\Lambda=\lambda I.$ As in the additive case, we further assume that a function $q$ and its all derivatives are polynomially bounded and consider the local linear approximation of $q(\sqrt{N} Y_{ij})$,
\beq \label{eq:local_linear2}
q(\sqrt{N} Y_{ij} ) = q(\sqrt{N} X_{ij}) + \gamma \sqrt{N} \E[q'(\sqrt{N} X_{ij})] \sum_\ell \bsu_i \bsu_\ell^T X_{\ell j} + \mathcal{R}_{ij} + \gamma \Delta_{ij},
\eeq
where
\[
\mathcal{R}_{ij}=\frac{1}{2}q''\Big(\sqrt{N} X_{ij} + \theta\gamma\sum_\ell \bsu_i \bsu_\ell^T \sqrt{N} X_{\ell j} \Big)\left(\gamma  \sum_\ell \bsu_i \bsu_\ell^T \sqrt{N}X_{\ell j} \right)^2
\]
for some $\theta \in [-1, 1]$ and 
\[
\Delta_{ij}=\sqrt{N} \caE_{ij} \sum_\ell \bsu_i \bsu_\ell^T X_{\ell j}, \quad \caE_{ij}=q'(\sqrt{N} X_{ij})-\E[q'(\sqrt{N} X_{ij})].
\]

For any unit vectors $\bsa = (a_1, a_2, \dots, a_M)$ and $\bsb = (b_1, b_2, \dots, b_N)$,
\[\begin{split}
\bsa^T\Delta \bsb&=\sum_s\sum_{i,j}a_iu_i(s)\caE_{ij}b_j\left(\sum_\ell u_\ell(s)\sqrt{N}X_{\ell j}\right)\\ 
&=\sum_s\sum_{i,j}a_iu_i(s)^2b_j\caE_{ij}\sqrt{N}X_{ij}+\sum_s\sum_{i,j}a_iu_i(s)\caE_{ij}b_j\left(\sum_{\ell\neq i} u_\ell(s) \sqrt{N}X_{\ell j}\right)
\end{split}\]
From the concentration inequalities such as Lemma \ref{lemma: LDE},
\beq \label{eq:u_concentration}
\sum_\ell \bsu_i \bsu_\ell^T \sqrt{ N}X_{\ell j}=\sum_{s}\sum_\ell u_i(s) u_\ell(s)^T \sqrt{ N}X_{\ell j}\prec \sum_{s}|u_i(s)|\left(\sum_\ell u_\ell(s)^2\right)^{1/2}\prec N^{-\phi}.
\eeq
Recall that $\| \caE \| = O(\sqrt{N})$ with overwhelming probability. Note that, by Assumption \ref{assump:entry1}, the density function $q$ have to be
an odd function. Further, since $q$ is an odd function (hence $xq'(x)$ is an odd function of $x$), the norm of the matrix whose $(i, j)$-entry is $\caE_{ij} \sqrt{N} X_{ij}$ is also $O(\sqrt{N})$. 
Thus,
\[
\bsa^T\Delta \bsb \prec N^{-2\phi} + N^{\frac{1}{2}-\phi},
\]
which shows that $\| \Delta \| \prec N^{\frac{1}{2}-\phi}$. Moreover, since $q''$ is polynomially bounded, following the proof of Theorem \ref{thm:trans-mean} with \eqref{eq:u_concentration},
\[
\|\mathcal{R} \|^2 \leq \| \mathcal{R}\| _F^2 \leq CN^{2-4\phi}.
\]
Thus, as in the additive case, the error terms $\mathcal{R}_{ij}$ and $\Delta_{ij}$ in \eqref{eq:local_linear2} are negligible when finding the limit of the extreme eigenvalues of the transformed matrix.

Set
\[
\mf := \E [q'(\sqrt{N}X_{ij})], \quad \vf := \E [q(\sqrt{N}X_{ij})^2], \quad \ef = \E [\sqrt{N}X_{ij} q(\sqrt{N}X_{ij})], \quad \widehat\gamma :=\gamma \mf/\sqrt{\vf},
\]
and
\[
Q_{ij} := \frac{1}{\sqrt{N\vf}}q(\sqrt{N}X_{ij}).
\]
With the approximation \eqref{eq:local_linear2}, we now focus on the largest eigenvalue of
\[
(Q+\wh\gamma\bsU\bsU^TX)^T(Q+\wh\gamma\bsU\bsU^TX).
\]
Note that the assumption on the polynomial boundedness of $q$ implies that the matrix $Q$ is also a rectangular matrix satisfying the assumptions in Definition \ref{defn:rect}.

Let $G(z)$ and $\caG(z)$ be the resolvents
\[
G \equiv G(z) := (QQ^T -zI)^{-1}, \qquad \caG \equiv \caG(z) := (Q^T Q -zI)^{-1}
\]
for $z \in \R$ outside an open interval containing $[d_-, d_+]$. We note that the following identities hold for $G(z)$ and $\caG(z)$:
\beq\label{eq:relation between g}
G(z)Q = Q \caG(z), \qquad Q^T G(z) Q=I+z\caG(z).
\eeq

As in the proof of Theorem \ref{thm:trans-mean}, we consider
\beq\begin{split} \label{eq:determinant}
	&(Q+\wh\gamma\bsU\bsU^TX)^T(Q+\wh\gamma\bsU\bsU^TX)-zI 
	\\&=(Q^TQ-zI)(I+(Q^TQ-zI)^{-1}(\wh\gamma X^T\bsU\bsU^TQ+\wh\gamma Q^T \bsU\bsU^TX+  \wh\gamma^2X^T\bsU\bsU^T\bsU\bsU^TX)).
\end{split}\eeq
Let 
\[
L \equiv L(z)=\caG(z)(\wh\gamma X^T\bsU\bsU^TQ+\wh\gamma Q^T \bsU\bsU^TX+  \wh\gamma^2X^T\bsU\bsU^T\bsU\bsU^TX),
\]
Then, as in the proof of Theorem \ref{thm:trans-mean}, if $z$ is an eigenvalue of $(Q+\wh\gamma\bsU\bsU^T X)^T (Q+\wh\gamma\bsU\bsU^T X)$ (but not of $Q^T Q$), $-1$ is an eigenvalue of $L(z)$. Again, the rank of $L$ is at most $2k$, with
\beq \begin{split} \label{eq:eigenvector}
	L \cdot \caG Q^T\bsU &=\wh\gamma \caG X^T\bsU\bsU^TQ\caG Q^T\bsU+\wh\gamma \caG Q^T\bsU\bsU^TX\caG Q^T\bsU+\wh\gamma^2\caG X^T\bsU\bsU^T\bsU\bsU^TX\caG Q^T\bsU, \\
	L \cdot \caG X^T\bsU &=\wh\gamma \caG X^T\bsU\bsU^TQ\caG X^T\bsU+\wh\gamma \caG Q^T\bsU\bsU^TX\caG X^T\bsU+\wh\gamma^2\caG X^T\bsU\bsU^T\bsU\bsU^TX\caG X^T\bsU,
\end{split} \eeq
and an eigenvector of $L$ is a linear combination of $\caG Q^T\bsu(\ell)$ and $\caG X^T\bsu(\ell)$ for $1\leq \ell\leq k$.

In the simplest case where $Q$ is the identity mapping, $Q=X$, hence the rank of $L$ is $k$, and the eigenvalue equation \eqref{eq:eigenvector} is simplified to
\beq \label{eq:eigenvector_i}
L \cdot \caG Q^T\bsU =\wh\gamma \caG Q^T\bsU(\bsU^TQ\caG Q^T\bsU)+\wh\gamma \caG Q^T\bsU(\bsU^TQ\caG Q^T\bsU)+\wh\gamma^2\caG Q^T\bsU(\bsU^T\bsU\bsU^TQ\caG Q^T\bsU).
\eeq
In this case, $\caG Q^T\bsu(\ell)$ are eigenvectors of $L$ corresponding to the eigenvalue $-1$, i.e., $L \cdot \caG Q^T\bsu(\ell) = -\caG Q^T\bsu(\ell)$. The right side of \eqref{eq:eigenvector_i} can be approximated as follows, which is a direct consequence of the isotropic local Marchenko--Pastur law (e.g., Theorem 2.5 of \cite{Bloemendal-Erdos-Knowles-Yau-Yin2014}).

With the isotropic local Marchenko--Pastur law, \eqref{eq:eigenvector_i} can be approximated by a deterministic vector equation on $z$ (and $s(z)$), and the location of the $k$ largest eigenvalues can be proved by solving the equation. In a general case where $Q$ is not a multiple of $X$ and the vectors $\caG Q^T\bsu(\ell)$ and $\caG X^T\bsu(\ell)$ are linearly independent, however, the eigenvalue equation \eqref{eq:eigenvector} contains other matrices $\bsU^TQ\caG Q^T\bsU$, $\bsU^TQ\caG X^T\bsU$, and $\bsU^TX\caG Q^T\bsU$, which cannot be estimated by Lemma \ref{lem:local_MP}. For these matrices, we use the following lemma.
\begin{lem} \label{lem:local_MP2}
	Suppose that the assumptions in Lemma \ref{lem:local_MP} hold. Then,
	\[
	\langle\bsu(\ell_1),X\caG Q^T\bsu(\ell_2)\rangle = \langle\bsu(\ell_1),Q\caG X^T\bsu(\ell_2)\rangle = \left[ \frac{\ef}{\sqrt{\vf}} (zs(z)+1) \right]\delta_{\ell_1\ell_2}+\caO_\prec(N^{-\phi})
	\]
	and
	\[
	\langle\bsu(\ell_1),X\caG X^T\bsu(\ell_2)\rangle =\left[  \frac{\ef^{2}}{\vf} zs(z) \left( \rat s(z) + \frac{\rat-1}{z} \right)^2 + \rat s(z) + \frac{\rat-1}{z} \right] \delta_{\ell_1\ell_2} +\caO_\prec(N^{-\phi}).
	\]
\end{lem}
We defer the proof to Appendix \ref{subsec:proof_lem}.

With Lemma \ref{lem:local_MP2}, we are ready to finish the proof. From the definition of $s(z)$ in Lemma \ref{lem:local_MP}, we notice that
\beq
s(z)=\frac{1}{1-\rat - \rat z s(z)-z},
\eeq
or
\beq
z\left( \rat s(z) + \frac{\rat-1}{z} \right) = -\frac{1}{s(z)} -z.
\eeq
Set $\sigma(z) := zs(z) + 1$.
By applying Lemmas \ref{lem:local_MP} and \ref{lem:local_MP2} to \eqref{eq:eigenvector_i}, for $1\leq \ell\leq k$
\beq \begin{split} \label{eq:eigen_F}
	L \cdot \caG Q^T\bsu(\ell) &=\wh\gamma \langle\bsu(\ell),Q\caG Q^T\bsu(\ell)\rangle \cdot \caG X^T\bsu(\ell)+\wh\gamma \langle\bsu(\ell),X\caG Q^T\bsu(\ell)\rangle \cdot \caG Q^T\bsu(\ell)\\
	&~~~+\|\bsu(\ell)\|^2\wh\gamma^2\langle\bsu(\ell),X\caG Q^T\bsu(\ell)\rangle \cdot \caG X^T\bsu(\ell) \\
	&= \wh\gamma \sigma(z) \caG X^T\bsu(\ell)+ \wh\gamma \sigma(z) \frac{\ef}{\sqrt{\vf}} \caG Q^T\bsu(\ell)+ \wh\gamma^2 \frac{\ef}{\sqrt{\vf}} \sigma(z) \caG X^T\bsu(\ell)+ \bsth_1(\ell) \,,
\end{split} \eeq
and
\beq \begin{split} \label{eq:eigen_X}
	L \cdot \caG X^T\bsu(\ell) &=\wh\gamma \langle\bsu(\ell),Q\caG X^T\bsu(\ell)\rangle \cdot \caG X^T\bsu(\ell)+ \wh\gamma \langle\bsu(\ell), X\caG X^T\bsu(\ell)\rangle \cdot \caG Q^T\bsu(\ell)\\
	&~~~+\|\bsu(\ell)\|^2\wh\gamma^2\langle\bsu(\ell),X\caG X^T\bsu(\ell)\rangle \cdot \caG X^T\bsu(\ell) \\
	&= \wh\gamma \sigma(z) \frac{\ef}{\sqrt{\vf}} \caG X^T\bsu(\ell)+ \wh\gamma \left( \left(\sigma(z) + \frac{\sigma(z)}{\sigma(z)-1} \right) \frac{\ef^2}{\vf} - \frac{\sigma(z)}{\sigma(z)-1} \right) \caG Q^T\bsu(\ell) \\
	&\qquad +\wh\gamma^2 \left( \left(\sigma(z) + \frac{\sigma(z)}{\sigma(z)-1} \right) \frac{\ef^2}{\vf} - \frac{\sigma(z)}{\sigma(z)-1} \right) \caG X^T\bsu(\ell) + \bsth_2(\ell) \,,
\end{split} \eeq
for some $\bsth_1(\ell), \bsth_2(\ell)$, which are linear combinations of $\caG Q^T\bsu(\ell)$ and $\caG X^T\bsu(\ell)$, with $\| \bsth_1 (\ell)\|, \| \bsth_2 (\ell)\| = \caO_\prec(N^{-\phi})$.

Suppose that $a_\ell\caG Q^T\bsu(\ell)+b_\ell\caG X^T\bsu(\ell)$ is an eigenvector of $L$ with the corresponding eigenvalue $-1$. From \eqref{eq:eigen_F}, \eqref{eq:eigen_X}, and the linear independence between $\caG Q^T\bsu(\ell)$ and $\caG X^T\bsu(\ell)$, we find the relation
\[ \begin{split}
-a_\ell &= a_\ell\wh\gamma \sigma(z) \frac{\ef}{\sqrt{\vf}} + \frac{b_\ell\wh\gamma \sigma(z)^2}{\sigma(z)-1} \cdot \frac{\ef^2}{\vf} - \frac{b_\ell\wh\gamma \sigma(z)}{\sigma(z)-1} + \caO(N^{-\phi}), \\
-b_\ell &= a_\ell\wh\gamma \sigma(z) + a_\ell \wh\gamma^2 \sigma(z) \frac{\ef}{\sqrt{\vf}} + b_\ell \wh\gamma \sigma(z) \frac{\ef}{\sqrt{\vf}} + \frac{b_\ell\wh\gamma^2 \sigma(z)^2}{\sigma(z)-1} \cdot \frac{\ef^2}{\vf} - \frac{b_\ell\wh\gamma^2 \sigma(z)}{\sigma(z)-1} + \caO(N^{-\phi}).
\end{split} \]
We then find that
\[
\frac{b_\ell}{a_\ell} \left( 1 +\wh\gamma \sigma(z) \frac{\ef}{\sqrt{\vf}} \right) + \wh\gamma \sigma(z) -\wh\gamma = \caO(N^{-\phi})
\]
and 
\[
a_\ell\left(1+\wh\gamma \sigma(z) \frac{\ef}{\sqrt{\vf}}\right)= b_\ell\left(\frac{\wh\gamma \sigma(z)}{\sigma(z)-1}\left(1- \sigma(z)\cdot \frac{\ef^2}{\vf}\right)\right)+\caO(N^{-\phi}),
\]
which implies that
\beq\label{res2}
1+2\wh\gamma \sigma(z) \frac{\ef}{\sqrt{\vf}}+ \wh\gamma^2 \sigma(z) =1+\left(\frac{2\gamma \mf\ef+\gamma^2\mf^2}{\vf}\right) \sigma(z) = \caO(N^{-\phi}).
\eeq 
From the explicit formula for $s$, it is not hard to check that \eqref{res2} holds if and only if
\[
\lambda_q := \frac{2\gamma \mf\ef+\gamma^2\mf^2}{\vf} > \sqrt{\rat}
\]
and
\beq \label{eigen2}
z = (1+\lambda_q) \left(1+\frac{\rat}{\lambda_q}\right) +\caO(N^{-\phi}).
\eeq
We see that it is valid for general $\Lambda$ in our model, since the above process also does not require any information of the SNRs as in the additive case.
Now, the desired theorem follows from the direct computation for the case $q = h_{\alpha_g}$; see also Appendix \ref{subsec:variation_cov}.

\subsection{Optimal entrywise transformation} \label{sec:optimize_transform}

\subsubsection{Additive model} \label{subsec:variation_mean}
Recall that 
\[
\E [q'(\sqrt{N}W_{ij})]=\E [q'(\sqrt{N}X_{ij})]=\mf, \qquad \E [q(\sqrt{N}W_{ij})^2]=\E [q(\sqrt{N}X_{ij})^2]=\vf.
\]
Following the proof of Theorem \ref{thm:trans-mean} in Appendix \ref{sec:proof_trans-mean}, it is not hard to see that the effective SNR is maximized by optimizing $\mf^2/\vf$. Such an optimization problem was already considered in \cite{Perry2018} for the spiked Wigner matrix. For the sake of completeness, we solve this problem by using the calculus of variations. Recall the density of random variables $\sqrt{N}W_{ij}$ and $\sqrt{N}X_{ij}$ is $g$. 

To optimize $q$, we need to maximize
\beq\label{target}
\left(\int_{-\infty}^{\infty}q'(x)g(x)\dd x\right)^2\slash\left(\int_{-\infty}^{\infty}q(x)^2g(x)\dd x\right)=\left(\int_{-\infty}^{\infty}q(x)g'(x)\dd x\right)^2\slash\left(\int_{-\infty}^{\infty}q(x)^2g(x)\dd x\right).
\eeq

Putting $(q+\varepsilon\eta)$ in place of $q$ in \eqref{target} and differentiating with respect to $\varepsilon$, we find that the optimal $q$ satisfies
\beq\label{target2}
\left(\int_{-\infty}^{\infty}\eta(x)g'(x)\dd x\right)\left(\int_{-\infty}^{\infty}q(x)^2g(x)\dd x\right)=\left(\int_{-\infty}^{\infty}q(x)\eta(x)g(x)\dd x\right)\left(\int_{-\infty}^{\infty}q(x)g'(x)\dd x\right)
\eeq
for any $\eta$. It is then easy to check that $q=-Cg'/g$ is the only solution of \eqref{target2}. Since the value in \eqref{target} does not change if we replace $q$ by $Cq$, and the effective SNR is increased with the entrywise transform $-g'/g$ is the optimal entrywise transformation for PCA.

\subsubsection{Multiplicative model} \label{subsec:variation_cov}
As we can see from the proof of Theorem \ref{thm:trans-cov} in Appendix \ref{sec:proof_trans-cov}, we need to maximize
\beq\begin{split}
	&\frac{2\left(\int_{-\infty}^{\infty}xq(x)g(x)\dd x\right)\left(\int_{-\infty}^{\infty}q'(x)g(x)\dd x\right)+\gamma\left(\int_{-\infty}^{\infty}q'(x)g(x)\dd x\right)^2}{\left(\int_{-\infty}^{\infty}q(x)^2g(x)\dd x\right)}
	\\&=\frac{-2\left(\int_{-\infty}^{\infty}xq(x)g(x)\dd x\right)\left(\int_{-\infty}^{\infty}q(x)g'(x)\dd x\right)+\gamma\left(\int_{-\infty}^{\infty}q(x)g'(x)\dd x\right)^2}{\left(\int_{-\infty}^{\infty}q(x)^2g(x)\dd x\right)}.
\end{split}\eeq
Putting $(q+\varepsilon\eta)$ in place of $q$ in \eqref{target} and differentiating with respect to $\varepsilon$, we find that the optimal $q$ satisfies
\beq\label{opt}\begin{split}
	&-2\left(\int x\eta g\right)\left(\int q g'\right)\left(\int q^2 g\right)-2\left(\int xq g\right)\left(\int \eta g'\right)\left(\int q^2 g\right)+2\gamma\left(\int \eta g'\right)\left(\int q g'\right)\left(\int q^2 g\right)
	\\&+4\left(\int q\eta g\right)\left(\int xq g\right)\left(\int q g'\right)-2\gamma\left(\int q g'\right)^2\left(\int q\eta g\right)=0
\end{split}\eeq
which is written with slight abuse of notation such as $\int x\eta g = \int_{-\infty}^{\infty}x \eta(x)g(x)\dd x$. Since the equation contains the terms
\[
\int x\eta g, \quad \int \eta g', \quad \int q\eta g,
\]
it is natural to consider an ansatz
\beq
q(x)=-\frac{g'(x)}{g(x)}+\alpha x
\eeq
for a constant $\alpha$. Collecting the terms involving $\int x\eta g$ and the terms involving $\int \eta g'$, we get
\[
2(\fh+\alpha)(\fh+2\alpha+\alpha^2)-4\alpha(1+\alpha)(\fh+\alpha)-2\alpha\gamma(\fh+\alpha)^2=0
\]
and
\[
-2(1+\alpha)(\fh+2\alpha+\alpha^2)-2\gamma(\fh+\alpha)(\fh+2\alpha+\alpha^2)+4(1+\alpha)(\fh+\alpha)+2\gamma(\fh+\alpha)^2=0.
\]
We can then check that 
\[
\alpha=\alpha_g=\frac{-\gamma \fh+\sqrt{4\fh+4\gamma \fh+\gamma^2\fh^2}}{2(1+\gamma)},
\]
and hence \eqref{opt} is satisfied with 
\[
q(x)=-\frac{g'(x)}{g(x)}+\frac{-\gamma \fh+\sqrt{4\fh+4\gamma \fh+\gamma^2\fh^2}}{2(1+\gamma)} x.
\]
The corresponding effective SNR
\[
\lambda_{h_{\alpha_g}} \equiv \lambda_g= \gamma + \frac{\gamma^2 \fh}{2} + \frac{\gamma \sqrt{4\fh + 4\gamma \fh + \gamma^2 \fh^2}}{2}.
\] 

For a general $\alpha$, when the entrywise transform $h_{\alpha}$ is applied, the effective SNR
\[
\lambda_{h_{\alpha}} = \frac{2\gamma (1+\alpha)(\fh+\alpha) + \gamma^2 (\alpha+\fh)^2}{\alpha^2 + 2\alpha + \fh},
\]
In particular, if $\alpha = \sqrt{\fh}$,
\[
\lambda_{h_{\sqrt{\fh}}} = \gamma(1+\sqrt{\fh}) + \frac{\gamma^2}{2} (\fh+\sqrt{\fh}) \geq 2\gamma + \gamma^2 = \lambda
\]
where the inequality is strict if $\fh > 1$.

\subsection{Proof of Lemma \ref{lem:local_MP2}} \label{subsec:proof_lem}
\subsubsection{Key ingredient: Entrywise local estimates}

Recall the definition of the random matrices $X$ and $Q.$ The couple of random matrices $(X,Q)$ is one example of the following concept for a coupled random matrices:
\begin{defn}[Entrywise correlated random matrices]\label{defn:ecrm}
	Suppose that $\ttA$ and $\ttB$ are $M\times N$ random rectangular matrices in Definition \ref{defn:rect} satisfying the following conditions:
	\begin{itemize}
		\item For all $1\leq a, b\leq M$ and $1\leq \alpha, \beta\leq N$, $\ttA_{a\alpha}$ and $\ttB_{b\beta}$ are dependent only when $a=b$ and $\alpha=\beta.$
		\item For all $a, \alpha$, $\E[\ttA_{a\alpha}] = \E[\ttB_{a\alpha}] = 0$, $N \E[\ttA_{a\alpha}^2]=w_\ttA$, $N \E[\ttB_{a\alpha}^2]=w_\ttB$, and $N\E[\ttA_{a\alpha}\ttB_{a\alpha}]=w_{\ttA\ttB}$. 
		\item For any positive integer $p$, there exists $C_p$, independent of $N$, such that 
		\[
		N^{\frac{p}{2}} \E[\ttA_{a\alpha}^p], N^{\frac{p}{2}} \E[\ttB_{a\alpha}^p] \leq C_p
		\] 
		for all $a, \alpha$. 
	\end{itemize}
	A couple of random matrices $(\ttA,\ttB)$ is called the entrywise correlated.
\end{defn}
The key estimates in the proof of Lemma \ref{lem:local_MP2} are the exact bounds on the entries of $K:=Q(QQ^T-zI)^{-1} X^T$ and $\caK:=X(QQ^T-zI)^{-1} X^T$. We prove the following lemma for the entrywise correlated random matrices $(\ttA,\ttB),$ which exactly contains the desired result.
\begin{lem}\label{lem:entrywise_ecrm}
	Let $(\ttA,\ttB)$ be the entrywise correlated random matrices with $w_B=1$. For $z \in \R$ outside an open interval containing $[d_-, d_+]$,
	\beq
	|(\ttA(\ttB^T\ttB-zI)^{-1}\ttA^T)_{ij}-(w_\ttA\mathfrak{s}(z)+w_{\ttA\ttB}^2zs(z)\mathfrak{s}(z)^2)\delta_{ij}|=\caO_\prec(N^{-1/2}),
	\eeq
	\beq
	|(\ttA(\ttB^T\ttB-zI)^{-1}\ttB^T)_{ij}-(w_{\ttA\ttB}\mathfrak{s}(z)+w_{\ttA\ttB}zs(z)\mathfrak{s}(z)^2)\delta_{ij}|=\caO_\prec(N^{-1/2})
	\eeq
	and
	\beq
	|(\ttB(\ttB^T\ttB-zI)^{-1}\ttA^T)_{ij}-(w_{\ttA\ttB}\mathfrak{s}(z)+w_{\ttA\ttB}zs(z)\mathfrak{s}(z)^2)\delta_{ij}|=\caO_\prec(N^{-1/2}).
	\eeq
\end{lem}
\begin{rem}
	Recall that $\sigma(z)=zs(z)+1$. For $(X,Q)$, since $w_{X}=w_Q=1$ and $w_{XQ}=\ef/\sqrt{\vf},$ we have the following:
	
	For $z \in \R$ outside an open interval containing $[d_-, d_+]$,
	\beq
	|K_{ij}-\widetilde{s}(z)\delta_{ij}|=\caO_\prec(N^{-1/2}), \qquad |\caK_{ij}-\check{s}(z)\delta_{ij}|=\caO_\prec(N^{-1/2}),
	\eeq
	where 
	\beq
	\widetilde{s}(z):=\sigma(z)\frac{\ef}{\sqrt{\vf}}, \qquad \check{s}(z):=zs(z) \left(\rat s(z)+\frac{\rat-1}{z}\right)^2 \frac{\ef^2}{\vf}+ \left(\rat s(z)+\frac{\rat-1}{z}\right).
	\eeq
	
\end{rem}

\subsubsection{Linearization} \label{subsec:linearization}

We consider
\[
G \equiv G_\ttB(z) = (\ttB\ttB^T -zI)^{-1}, \qquad \caG \equiv \caG_\ttB(z) = (\ttB^T \ttB -zI)^{-1}.
\]
In the proof of Lemma \ref{lem:entrywise_ecrm}, we use the formalism known as the linearization to simplify the computation. 
We define an $(M+N) \times (M+N)$ symmetric matrix $H_\ttB$ by
\beq \label{eq:linearization}
H_\ttB \equiv H_\ttB(z) =
\begin{pmatrix}
	-zI_M & \ttB \\
	\ttB^T & -I_N
\end{pmatrix},
\eeq
where $I_M$ and $I_N$ are the identity matrices with size $M$ and $N$, respectively. 

Let $R_\ttB(z) = H_\ttB(z)^{-1}$. (For the invertibility of $H_\ttB(z)$, we refer to Section 5.1 in \cite{Lee-Schnelli_Sample}.) By Schur's complement formula,
\beq \label{eq:resolvent_linearization}
R_Q(z) =
\begin{pmatrix}
	G_\ttB(z) & G_\ttB(z)\ttB \\
	\ttB^TG_\ttB(z) & z\caG_\ttB(z)
\end{pmatrix}.
\eeq
Therefore, 
\beq \label{eq:schur}
R_{ab}(z) = (\ttB\ttB^T -zI)^{-1}_{ab} = G_{ab}(z), \qquad R_{\alpha\beta}(z) = z(\ttB^T \ttB -zI)^{-1}_{\alpha-M, \beta-M} = z\caG_{\alpha-M, \beta-M}(z),
\eeq
and
\beq 
R_{\alpha a}(z)=R_{a\alpha}(z) = (G\ttB)_{a,\alpha-M}(z),
\eeq
where we use lowercase Latin letters $a, b, c, \dots$ for indices from $1$ to $M$ and Greek letters $\alpha, \beta, \gamma, \dots$ for indices from $(M+1)$ to $(M+N)$. We also use uppercase Latin letters $A, B, C, \dots$ for indices from $1$ to $(M+N)$. In the rest of Appendix \ref{sec:proof_trans}, we omit the subscript $Q$ for brevity.

For $\T \subset \{1, 2, \dots, M+N \}$, we define the matrix minor $H^{(\T)}$ by
\beq
(H^{(\T)})_{AB} := \mathbf{1}_{\{A, B \notin \T \}} H_{AB}\,.
\eeq
Moreover, for $A, B \notin \T$ we define
\beq
R^{(\T)}_{AB}(z) := (H^{(\T)})^{-1}_{AB},
\eeq
In the definitions above, we abbreviate $(\{A \})$ by $(A)$; similarly, we write $(AB)$ instead of $(\{A,B\})$.

We have the following resolvent (decoupling) identities for the matrix entries of $R$ and $R^{(\T)}$, which are elementary consequences of Schur's complement formula; see e.g.\ Lemma 5.1 of \cite{Lee-Schnelli_Sample}. 

\begin{lem}[Resolvent identities for $R$] \label{lem:res identity}
	Suppose that $z \in \R$ is outside an open interval containing $[d_-, d_+]$.
	
	- For $a \neq b$,
	\[
	R_{ab} = -R_{aa} \sum_{\alpha} H_{a \alpha} R^{(a)}_{\alpha b} = -R_{bb} \sum_{\beta} R^{(b)}_{a\beta} H_{\beta b}.
	\]
	
	- For $\alpha \neq \beta$,
	\[
	R_{\alpha\beta} = -R_{\alpha\alpha} \sum_{a} H_{\alpha a} R^{(\alpha)}_{a \beta} = -R_{\beta\beta} \sum_{b} R^{(\beta)}_{\alpha b} H_{b \beta}.
	\]
	
	- For any $a$ and $\alpha$,
	\[
	R_{a \alpha} = -R_{aa} \sum_{\beta} H_{a\beta} R^{(a)}_{\beta \alpha} = -R_{\alpha\alpha} \sum_{b} R^{(\alpha)}_{ab} H_{b \alpha}.
	\]
	
	- For $A, B \neq C$,
	\[
	R_{AB} = R_{AB}^{(C)} + \frac{R_{AC} R_{CB}}{R_{CC}}.
	\]
\end{lem}

Throughout this section, we will frequently use the estimate that all entries of $X$ and $Q$ (and hence all off-diagonal entries of $W$) are $\caO_{\prec}(N^{-1/2})$, which holds since all moments of the entries of $\sqrt{N} Q$ and $\sqrt{N} X$ are bounded. For the entries of $R$, we have the following estimates:

\begin{lem}\label{lem:entrywise_MPlaw}
	Let
	\beq
	\mathfrak{s}(z)=\left(\rat s(z)+\frac{\rat-1}{z}\right).
	\eeq
	For $z \in \R$ outside an open interval containing $[d_-, d_+]$,
	\beq
	\left|R_{ij}(z)-s(z)\delta_{ij}\right|, \left|R_{\mu\nu}(z)- z\mathfrak{s}(z)\delta_{\mu\nu}\right|, \left|R_{i \mu}(z)\right| \prec N^{-1/2}.
	\eeq
\end{lem}

\begin{proof}[Proof of Lemma \ref{lem:entrywise_MPlaw}]
	The first two estimates can be checked from Theorem 2.5 (and Remark 2.7) in \cite{Bloemendal-Erdos-Knowles-Yau-Yin2014} with the deterministic unit vectors $\mathbf{v}=\bse_i$ and $\mathbf{w}=\bse_j$ where $\bse_i\in\R^N$ or $\R^M$ is a standard basis vector whose $i$-th coordinate is 1 and all other coordinates are zero. For the last estimate, we apply Lemma \ref{lem:res identity} to find that
	\[
	R_{i \mu}(z) = -R_{ii} \sum_{\alpha} H_{i\alpha} R^{(i)}_{\alpha \mu}.
	\]
	Since $H_{i\alpha}$ and $R^{(i)}_{\alpha \mu}$ are independent, $R^{(i)}_{\alpha \mu} \prec N^{-1/2}$ for $\alpha \neq \mu$, and $R^{(i)}_{\mu\mu} = \Theta(1)$ with overwhelming probability, we find from Lemma \ref{lemma: LDE} that
	\[
	\sum_{\alpha} H_{i\alpha} R^{(i)}_{\alpha \mu} \prec \left( \frac{1}{N} \sum_{\alpha} | R^{(i)}_{\alpha \mu}|^2 \right)^{1/2} \prec N^{-1/2}.
	\]
\end{proof}

\begin{proof}[Proof of Lemma \ref{lem:entrywise_ecrm}] \label{subsec:proof_entrywise}
	
	Throughout this section, for the sake of brevity, we will use the notation
	\[
	\ttB_{a\alpha} := \ttB_{a, (\alpha-M)} = H_{a\alpha}, \qquad \ttA_{a\alpha} := \ttA_{a, (\alpha-M)}.
	\]
	
	We begin by estimating the diagonal entry $(\ttB\caG \ttA^T)_{ii}$.
	From Schur's complement formula, \eqref{eq:schur}, we can decompose it into
	\beq \label{eq:decomp_K}
	(\ttB\caG \ttA^T)_{ii} = \frac{1}{z} \sum_{\alpha} H_{i\alpha} R_{\alpha\alpha} \ttA_{i\alpha} + \frac{1}{z} \sum_{\alpha\neq\beta} H_{i\alpha} R_{\alpha\beta} \ttA_{i\beta}.
	\eeq
	From concentration inequalities it is not hard to see that
	\[
	\sum_{\alpha} \ttB_{i\alpha} \ttA_{i\alpha} = \E[ \ttB_{i\alpha} \ttA_{i\alpha} ] +\caO_{\prec}(N^{-1/2}) = w_{\ttA\ttB}+\caO_{\prec}(N^{-1/2}).
	\]
	Applying Lemma \ref{lem:entrywise_MPlaw}, we find for the first term in the right side of \eqref{eq:decomp_K} that
	\beq \label{eq:decomp_K1}
	\frac{1}{z} \sum_{\alpha} H_{i\alpha} R_{\alpha\alpha} \ttA_{i\alpha} =w_{\ttA\ttB}\mathfrak{s}(z)+\caO_{\prec}(N^{-1/2}).
	\eeq
	
	We next estimate the second term in the right side of \eqref{eq:decomp_K}. We expand it with the resolvent identities in Lemma \ref{lem:res identity} as follows:
	\beq \begin{split} \label{eq:decomp1}
		\sum_{\alpha\neq\beta} H_{i\alpha} R_{\alpha\beta} \ttA_{i\beta} &= \sum_{\alpha\neq\beta} H_{i\alpha} R^{(i)}_{\alpha\beta} \ttA_{i\beta} + \sum_{\alpha\neq\beta} H_{i\alpha} \frac{R_{\alpha i} R_{i\beta}}{R_{ii}} \ttA_{i\beta} \\
		&= \sum_{\alpha\neq\beta} H_{i\alpha} R^{(i)}_{\alpha\beta} \ttA_{i\beta} + \sum_{\alpha\neq\beta} H_{i\alpha} \frac{R_{\alpha i} R_{i\beta}}{s(z)} \ttA_{i\beta} + \caO_{\prec}(N^{-1/2}).
	\end{split} \eeq
	Here, in the estimate for the second term, we simply counted the power (of $N$) as it involves two indices for the sum (hence $O(N^2)$ terms) of $H_{i\alpha}, R_{\alpha i}, R_{i\beta}, \ttA_{i\beta} \prec N^{-1/2}$, hence $\sum_{\alpha\neq\beta} H_{i\alpha} R_{\alpha i} R_{i\beta} \ttA_{i\beta} = \caO_{\prec}(1)$.
	Applying Lemma \ref{lemma: LDE} to the first term in the right side of \eqref{eq:decomp1},
	\[
	\sum_{\alpha\neq\beta} H_{i\alpha} R^{(i)}_{\alpha\beta} \ttA_{i\beta} \prec \left( \frac{1}{N^2} \sum_{\alpha, \beta} |R^{(i)}_{\alpha\beta}|^2 \right)^{1/2} \prec N^{-1/2}.
	\]
	For the second term in the right side of \eqref{eq:decomp1}, we further expand it to find
	\[ \begin{split}
	\sum_{\alpha\neq\beta} H_{i\alpha} R_{\alpha i} R_{i\beta} \ttA_{i\beta} = \sum_{\alpha\neq\beta} H_{i\alpha} R_{\alpha i} R_{i\beta} \ttA_{i\beta} = -\sum_{\alpha\neq\beta} H_{i\alpha} \left( R_{ii} \sum_{\mu} R^{(i)}_{\alpha\mu} H_{\mu i} R_{i\beta} \ttA_{i\beta} \right)
	\end{split} \]
	Note that
	\[
	\sum_{\mu} R^{(i)}_{\alpha\mu} H_{\mu i} \prec N^{-1/2},
	\]
	as in the proof of Lemma \ref{lem:entrywise_MPlaw}. Since 
	\[
	|R_{ij} - s(z)| \prec N^{-1/2}, \quad R_{i\beta} = R^{(\alpha)}_{i\beta} + \frac{R_{i\alpha} R_{\alpha\beta}}{R_{\alpha\alpha}} = R^{(\alpha)}_{i\beta} + N^{-1},
	\]
	we have
	\beq \label{eq:expand_oneside} \begin{split}
		&-\sum_{\alpha\neq\beta} H_{i\alpha} \left( R_{ii} \sum_{\mu} R^{(i)}_{\alpha\mu} H_{\mu i} R_{i\beta} \ttA_{i\beta} \right) = -s(z) \sum_{\alpha\neq\beta} H_{i\alpha} \left( \sum_{\mu} R^{(i)}_{\alpha\mu} H_{\mu i} R^{(\alpha)}_{i\beta} \ttA_{i\beta} \right) + \caO_{\prec} (N^{-1/2}) \\
		&= -s(z) \sum_{\alpha\neq\beta} H_{i\alpha} \left( \sum_{\mu:\mu \neq \alpha} R^{(i)}_{\alpha\mu} H_{\mu i} R^{(\alpha)}_{i\beta} \ttA_{i\beta} \right) -s(z) \sum_{\alpha\neq\beta} (H_{i\alpha})^2 R^{(i)}_{\alpha\alpha} R^{(\alpha)}_{i\beta} \ttA_{i\beta} + \caO_{\prec} (N^{-1/2}).
	\end{split} \eeq
	Applying Lemma \ref{lemma: LDE} again to the first term in the right side of \eqref{eq:expand_oneside},
	\[ \begin{split}
	\sum_{\alpha\neq\beta} H_{i\alpha} \left( \sum_{\mu:\mu \neq \alpha} R^{(i)}_{\alpha\mu} H_{\mu i} R^{(\alpha)}_{i\beta} \ttA_{i\beta} \right) &\prec \left( \frac{1}{N} \sum_{\alpha} \left| \sum_{\beta: \beta\neq \alpha} \left[ \sum_{\mu:\mu \neq \alpha} R^{(i)}_{\alpha\mu} H_{\mu i} \right] R^{(\alpha)}_{i\beta} \ttA_{i\beta} \right|^2 \right)^{1/2} \\
	& \prec \left( \frac{1}{N} \sum_{\alpha} \left[ \sum_{\beta: \beta\neq \alpha} N^{-1/2} \left| R^{(\alpha)}_{i\beta} \ttA_{i\beta} \right| \right]^2 \right)^{1/2} \prec N^{-1/2}.
	\end{split} \]
	Similarly, by expanding $R^{(\alpha)}_{i\beta}$, we find for the second term in the right side of \eqref{eq:expand_oneside} that
	\[ \begin{split}
	-&s(z) \sum_{\alpha\neq\beta} (H_{i\alpha})^2 R^{(i)}_{\alpha\alpha} R^{(\alpha)}_{i\beta} \ttA_{i\beta} = zs(z) \mathfrak{s}(z) \sum_{\alpha\neq\beta} (H_{i\alpha})^2 R^{(\alpha)}_{ii} \sum_{\nu:\nu \neq \alpha}H^{(\alpha)}_{i\nu} R^{(i\alpha)}_{\nu\beta} \ttA_{i\beta} + \caO_{\prec} (N^{-1/2}) \\
	&= zs(z)^2 \mathfrak{s}(z) \sum_{\alpha\neq\beta} (H_{i\alpha})^2 \sum_{\nu:\nu \neq \alpha, \beta} H_{i\nu} R^{(i\alpha)}_{\nu\beta} \ttA_{i\beta} + zs(z)^2 \mathfrak{s}(z) \sum_{\alpha\neq\beta} (H_{i\alpha})^2 H_{i\beta} R^{(i\alpha)}_{\beta\beta} \ttA_{i\beta} + \caO_{\prec} (N^{-1/2}) \\
	&= z^2 s(z)^2 \mathfrak{s}(z)^2 \sum_{\alpha\neq\beta} (H_{i\alpha})^2 H_{i\beta} \ttA_{i\beta} + \caO_{\prec} (N^{-1/2}),
	\end{split} \]
	where we used Lemma \ref{lemma: LDE} to find
	\[
	\sum_{\nu \neq \beta:\nu, \beta \neq \alpha} H_{i\nu} R^{(i\alpha)}_{\nu\beta} \ttA_{i\beta} \prec \left( \frac{1}{N^2} \sum_{\nu \neq \beta:\nu, \beta \neq \alpha} \left| R^{(i\alpha)}_{\nu\beta} \right|^2 \right)^{1/2} \prec N^{-1/2}.
	\]
	Thus, since $w_\ttB=1,$
	\[ \begin{split}
	\sum_{\alpha\neq\beta} H_{i\alpha} R_{\alpha i} R_{i\beta} \ttA_{i\beta} &= z^2 s(z)^2 \mathfrak{s}(z)^2 \sum_{\alpha\neq\beta} (H_{i\alpha})^2 H_{i\beta} \ttA_{i\beta} + \caO_{\prec} (N^{-1/2}) \\
	&= z^2 s(z)^2 \mathfrak{s}(z)^2 w_{\ttA\ttB} + \caO_{\prec} (N^{-1/2}),
	\end{split} \]
	and putting it back to \eqref{eq:decomp1} and \eqref{eq:decomp_K}, together with \eqref{eq:decomp_K1}, we conclude that
	\beq \label{eq:estimate_K}
	(\ttA\caG\ttB)_{ii} = w_{\ttA\ttB}\mathfrak{s}(z)+w_{\ttA\ttB}zs(z)\mathfrak{s}(z)^2 + \caO_{\prec} (N^{-1/2}) = w_{\ttA\ttB}\sigma(z) \frac{\ef}{\sqrt{\vf}} + \caO_{\prec} (N^{-1/2}),
	\eeq
	where we used the identity $zs(z)\mathfrak{s}(z) = -\sigma(z)$.
	In the same manner, we also find that
	\beq \begin{split} \label{eq:estimate_caK}
		(\ttA\caG\ttA)_{ii} &= \frac{1}{z} \sum_{\alpha} \ttA_{i\alpha} R_{\alpha\alpha} \ttA_{i\alpha} + zs(z)\mathfrak{s}(z)^2 \sum_{\alpha\neq\beta} \ttA_{i\alpha} H_{i\alpha} H_{i\beta} \ttA_{i\beta} + \caO_{\prec} (N^{-1/2}) \\
		&= w_\ttA\mathfrak{s}(z) + w_{\ttA\ttB}^2zs(z)\mathfrak{s}(z)^2 + \caO_{\prec} (N^{-1/2}).
	\end{split} \eeq
	
	We next estimate the off-diagonal entry $(\ttA\caG\ttB)_{ij}$. We expand it as
	\beq \begin{split} \label{eq:decomp_off}
		(\ttA\caG\ttB)_{ij} &= \frac{1}{z} \sum_{\alpha, \beta} H_{i\alpha} R_{\alpha\beta} \ttA_{j\beta} = \frac{1}{z} \sum_{\alpha, \beta} H_{i\alpha} R^{(i)}_{\alpha\beta} \ttA_{j\beta} + \frac{1}{z} \sum_{\alpha, \beta} H_{i\alpha} \frac{R_{\alpha i} R_{i\beta}}{R_{ii}} \ttA_{j\beta} \\
		&= \frac{1}{z} \sum_{\alpha, \beta} H_{i\alpha} R^{(ij)}_{\alpha\beta} \ttA_{j\beta} + \frac{1}{z} \sum_{\alpha, \beta} H_{i\alpha} \frac{R^{(i)}_{\alpha j} R^{(i)}_{j\beta}}{R^{(i)}_{jj}} \ttA_{j\beta} + \frac{1}{z} \sum_{\alpha, \beta} H_{i\alpha} \frac{R^{(j)}_{\alpha i} R^{(j)}_{i\beta}}{R^{(i)}_{jj}} \ttA_{j\beta} + \caO_{\prec}(N^{-1/2})
	\end{split} \eeq
	From Lemma \ref{lemma: LDE},
	\[
	\sum_{\alpha, \beta} H_{i\alpha} R^{(ij)}_{\alpha\beta} \ttA_{j\beta} \prec N^{-1/2}.
	\]
	We also have
	\[ \begin{split}
	\sum_{\alpha, \beta} H_{i\alpha} \frac{R^{(i)}_{\alpha j} R^{(i)}_{j\beta}}{R^{(i)}_{jj}} \ttA_{j\beta} &\prec \left( \frac{1}{N} \sum_{\alpha} \left| \sum_{\beta} \frac{R^{(i)}_{\alpha j} R^{(i)}_{j\beta}}{R^{(i)}_{jj}} \ttA_{j\beta} \right|^2 \right)^{1/2} \prec \left( \frac{1}{N} \sum_{\alpha} \left| \sum_{\beta} N^{-3/2} \right|^2 \right)^{1/2} \\
	&\prec N^{-1/2}
	\end{split} \]
	and a similar estimate holds for the third term in the right side of \eqref{eq:decomp_off}. Thus,
	\[
	(\ttA\caG\ttB)_{ij} \prec N^{-1/2}
	\]
	In the same manner, we also find that $(\ttA\caG\ttA)_{ij} \prec N^{-1/2}$. Together with \eqref{eq:estimate_K} and \eqref{eq:estimate_caK}, this proves Lemma \ref{lem:entrywise_ecrm}.
\end{proof}
%
%
%

\subsubsection{Isotropic local law}

We also assume that $w_\ttB=1$ and use the same notation in previous section. Then our goal is to prove the following statement:
\begin{lem}\label{lem:isotropic coupled}
	Let $(\ttA,\ttB)$ be the entrywise correlated random matrices where $w_\ttB=1$ and $\bsx$, $\bsy$ are deterministic and $\ell^2$ - normalized vectors in $\R^M.$ Then, for $z \in \R$ outside an open interval containing $[d_-, d_+]$,
	\[\langle\bsx, \ttA(\ttB^T\ttB-zI)^{-1}\ttA^T\bsy\rangle=(w_\ttA\mathfrak{s}(z)+w_{\ttA\ttB}^2zs(z)\mathfrak{s}(z)^2)\langle\bsx,\bsy\rangle+\caO_\prec(N^{-1/2}).\]
\end{lem}

\begin{proof}[Proof of Lemma \ref{lem:isotropic coupled}] \label{subsec:proof_isotropictype}
	Note that, due to polarization identity, we suffice to prove for $\langle\bsx, \ttA(\ttB^T\ttB-zI)^{-1}\ttA^T\bsx\rangle.$ Recall that we have
	\[
	(\ttA(\ttB^T\ttB-zI)^{-1}\ttA^T)_{ij}=(w_\ttA\mathfrak{s}(z)+w_{\ttA\ttB}^2zs(z)\mathfrak{s}(z)^2)\delta_{ij}+\caO_\prec(N^{-1/2})
	\]
	and
	\[
	(\ttA(\ttB^T\ttB-zI)^{-1}\ttB^T)_{ij}=(\ttB(\ttB^T\ttB-zI)^{-1}\ttA^T)_{ij}=(w_{\ttA\ttB}\mathfrak{s}(z)+w_{\ttA\ttB}zs(z)\mathfrak{s}(z)^2)\delta_{ij}+\caO_\prec(N^{-1/2}).
	\] 
	Once the entrywise local law is given, the proof of the isotropic (or anisotropic) type law follows exactly as in \cite{Bloemendal-Erdos-Knowles-Yau-Yin2014}. To be more precisely, we can write 
	\[
	\langle\bsx, \ttA(\ttB^T\ttB-zI)^{-1}\ttA^T\bsx\rangle=\sum_{i}x_i(\ttA\caG \ttA^T)_{ii}x_i+\sum_{i\neq j}x_i(\ttA\caG \ttA^T)_{ij}x_j.
	\]
	Then the entrywise local law implies
	\[\begin{split}
	&\sum_{i}x_i(\ttA\caG \ttA^T)_{ii}x_i-(w_\ttA\mathfrak{s}(z)+w_{\ttA\ttB}^2zs(z)\mathfrak{s}(z)^2)\langle\bsx,\bsx\rangle\\
	&~~~=\sum_{i}x_i^2\left[(\ttA\caG \ttA^T)_{ii}-(w_\ttA\mathfrak{s}(z)+w_{\ttA\ttB}^2zs(z)\mathfrak{s}(z)^2)\right]\prec N^{-1/2},
	\end{split}\]
	and so the main difficulty is to control the off-diagonal part
	\[\mathcal{Z}_{\ttA\ttB}:=\sum_{\alpha,\beta}\sum_{i\neq j}x_i\ttA_{i\alpha}\caG_{\alpha\beta}\ttA_{j\beta}x_j=\caO_\prec(N^{-1/2}).\]
	For instance, for the sample covariance matrix case
	\[\begin{split}
	\langle\bsx, \ttB\caG \ttB^T\bsx\rangle&=(zs(z)+1)\langle\bsx,\bsx\rangle+\caO_\prec(N^{-1/2})\\
	&=(\mathfrak{s}(z)+zs(z)\mathfrak{s}(z)^2)\langle\bsx,\bsx\rangle+\caO_\prec(N^{-1/2})
	\end{split}\]
	was proved in \cite{Bloemendal-Erdos-Knowles-Yau-Yin2014} by proving the following bound for higher moments
	\beq\label{eq:pth moment bound iso}
	\E|\mathcal{Z}_\ttB|^p\prec N^{-p/2}
	\eeq 
	for any large and even $p$, where
	\[
	\mathcal{Z}_\ttB:=\sum_{i\neq j}x_i(\ttB\caG \ttB^T)_{ij}x_j=z\sum_{i\neq j}x_iG_{ij}x_j.
	\] 
	In particular, the \eqref{eq:pth moment bound iso} have proved by using the standard \emph{maximal expansion method} in \cite{Bloemendal-Erdos-Knowles-Yau-Yin2014} and \cite{ajanki2017wignertype}, which only requires the independence between each element, the boundedness of the moment of each entries, and the entrywise local law. Thus, from the definition of the entrywise correlated random matrices $(\ttA,\ttB)$, it can be expected that 
	\[
	\E|\mathcal{Z}_{\ttA\ttB}|^p\prec N^{-p/2}
	\]
	also holds for any large and even $p$, by expanding maximally $\caG_{\alpha\beta}$ instead of $G_{ab}$ as in \eqref{eq:decomp_off}. Then, we can conclude the proof by using Markov inequality. 
	
	To prove such an argument , we only need to check what is changing. First, we express the $p$-th moment of $\mathcal{Z}_{\ttA\ttB}$ by
	\beq\label{eq:p-th moment}
	\E|\mathcal{Z}_{\ttA\ttB}|^p=\E\sum_{b_{11}\neq b_{12}}\cdots\sum_{b_{p1}\neq b_{p2}}\left(\prod_{k=1}^{p/2} x_{b_{k1}}(\ttA\caG\ttA^T)_{b_{k1}b_{k2}}x_{b_{k2}}\right)\left(\prod_{k=p/2+1}^{p} x_{b_{k1}}(\ttA\caG\ttA^T)_{b_{k1}b_{k2}}x_{b_{k2}}\right).
	\eeq
	Let $\mathbb{T}=\{b_{k1}\}\cup\{b_{k2}\}$ be the set of indices of $\bsx$ appearing in the fixed summand of the representation of the $p$-th moment of $\mathcal{Z}_{\ttA\ttB}.$ Then our goal is to decompose the off-diagonal entry of the matrix $(\ttA\caG \ttA^T)$ into the two parts by using Lemma \ref{lem:res identity}, where one consists of the finite number of \emph{the maximally expanded term} and the other consists of the terms containing a sufficiently large number of off-diagonal entries. We note that the latter case is small enough due to the entrywise local laws of off-diagonal entries, and so the leading order term contained in the formal.

	
	\subsubsection*{Step 1 : The maximal expansion for the off-diagonal entries of $(A\caG_BA).$}
	
	In our case, \emph{the maximally expanded terms} (cf. Definition 5.4 of \cite{Bloemendal-Erdos-Knowles-Yau-Yin2014}) refer to terms that have one of the following forms: $(\ttA\caG^{(\mathbb{T}\backslash a,b)}\ttA^T)_{ab}$, $(\ttA\caG^{(\mathbb{T}\backslash a,b)}\ttB^T)_{ab}$, $(\ttB\caG^{(\mathbb{T}\backslash a,b)}\ttA^T)_{ab}$ or  $(\ttB\caG^{(\mathbb{T}\backslash a,b)}\ttB^T)_{ab}=z(G^{(\mathbb{T}\backslash a,b)})_{ab},$ for some $a\neq b\in\mathbb{T}.$  
	To proceed, we use the following operation successively : 
	
	\textbf{Operation (a)} 
	
	Let $\caT\subset\{1,\ldots M\}$ be a set of indices.
	\begin{itemize}
		\item For $a\neq b$ and $c\notin\caT,$	
		\[\begin{split}
		(\ttA\caG^{(\caT)}\ttA^T)_{ab}&=(\ttA\caG^{(\caT c)}\ttA^T)_{ab}+\frac{(\ttA\caG^{(\caT)}\ttB^T)_{ac}(\ttB\caG^{(\caT)}\ttA^T)_{cb}}{z(G^{(\caT)})_{cc}}
		\end{split}
		\]
		\item For $a\neq b$ and $c\notin\caT,$
		\[
		\begin{split}
		(\ttA\caG^{(\caT)}B^T)_{ab}&=(\ttA\caG^{(\caT c)}\ttB^T)_{ab}+\frac{(\ttA\caG^{(\caT)}\ttB^T)_{ac}(\ttB\caG^{(\caT)}\ttB^T)_{cb}}{z(G^{(\caT)})_{cc}}
		\\&=(\ttA\caG^{(\caT c)}\ttB^T)_{ab}+\frac{(\ttA\caG^{(\caT)}\ttB^T)_{ac}(G^{(\caT)})_{cb}}{(G^{(\caT)})_{cc}}
		\end{split}
		\]
		\item For $a\neq b\notin\caT$ and $a,b\neq c$
		\[
		\frac1{z}(\ttB\caG^{(\caT)}\ttB^T)_{ab}=(G^{(\caT)})_{ab}=(G^{(\caT  c)})_{ab}+\frac{(G^{(\caT)})_{ac}(G^{(\caT)})_{cb}}{(G^{(\caT)})_{cc}}
		\]
		\item For $a\neq b\notin\caT$
		\[
		\frac{1}{(G^{(\caT)})_{aa}}=\frac{1}{(G^{(\caT  b)})_{aa}}-\frac{(G^{(\caT)})_{ab}(G^{(\caT)})_{ba}}{(G^{(\caT)})_{aa}(G^{(\caT b)})_{aa}(G^{(\caT)})_{bb}}
		\]
	\end{itemize}
	We then observe that the expanded terms contains at most two crossed terms $(\ttA\caG^{(\mathbb{T}\backslash a,b)}\ttB^T)_{ab}$, $(\ttB\caG^{(\mathbb{T}\backslash a,b)}\ttA^T)_{ab}$ and each expansions produce two types of terms, the first one has one more additional upper index, and the second one at least one more additional off-diagonal entry of $\ttB\caG \ttA^T$, $\ttA\caG \ttB^T$ or $\ttB\caG \ttB^T$. Moreover, we also remark that the denominators are always the diagonal entries of the resolvent $G^{(\caT)}.$
	
	It can be seen that the above expansion formulas eventually play the same role as operation (a) in \cite{Bloemendal-Erdos-Knowles-Yau-Yin2014}.  Therefore, to obtain the desired decomposition, we only need to iterate the operation (a) until it can no longer be expanded or contains sufficiently many off-diagonal entries. 
	
	\subsubsection*{Step 2 : The further expansions for the maximally expanded off-diagonal entries}	
	
	We further expand the maximally expanded term by using the following operations : 
	
	\textbf{Operations (b) (and (c))}
	\begin{itemize}
		\item For $a\neq b\in\mathbb{T}$
		\[
		(G^{(\mathbb{T}\backslash a,b)})_{ab}=z(G^{(\mathbb{T}\backslash a,b)})_{aa}(G^{(\mathbb{T}\backslash b)})_{bb}(\ttB\caG^{(\mathbb{T})}\ttB^T)_{ab}.
		\] 
		\item 	Furthermore, we use the following type expansion, which is from the above formula, to the terms $(G^{(\mathbb{T}\backslash a,b)})_{aa}$ and $(G^{(\mathbb{T}\backslash a,b)})_{bb}$ 
		\[\begin{split}
		(G^{(\mathbb{T}\backslash a,b)})_{aa}&=(G^{(\mathbb{T}\backslash a)})_{aa}+\frac{(G^{(\mathbb{T}\backslash a,b)})_{ab}(G^{(\mathbb{T}\backslash a,b)})_{ba}}{(G^{(\mathbb{T}\backslash a,b)})_{bb}}
		\\&=(G^{(\mathbb{T}\backslash a)})_{aa}+z^2(G^{(\mathbb{T}\backslash a,b)})_{aa}(G^{(\mathbb{T}\backslash a)})_{aa}(G^{(\mathbb{T}\backslash b)})_{bb}(\ttB\caG^{(\mathbb{T})}\ttB^T)_{ab}^2.
		\end{split}
		\]
		Then, this expansion splits such not-maximally expanded term into two parts, one is maximally expanded and the other is a monomial expressed as the product of itself, the diagonal entry, and the maximally expanded terms. In particular, it can be seen that the number of the off-diagonal entries included in the latter monomial increases by exactly two. 
		
		\item For $a\neq b\in\mathbb{T}$
		\[
		\begin{split}
		(\ttA\caG^{(\mathbb{T}\backslash a,b)}\ttA^T)_{ab}&=(\ttA\caG^{(\mathbb{T})}\ttA^T)_{ab}+z(G^{(\mathbb{T}\backslash b)})_{bb}(\ttA\caG^{(\mathbb{T})}\ttB^T)_{ab}(\ttB\caG^{(\mathbb{T})}\ttA^T)_{bb}\\&~~~+\frac{(\ttA\caG^{(\mathbb{T}\backslash a,b)}\ttB^T)_{aa}(\ttB\caG^{(\mathbb{T}\backslash a,b)}\ttA^T)_{ab}}{z(G^{(\mathbb{T}\backslash a,b)})_{aa}}
		\\&=(\ttA\caG^{(\mathbb{T})}\ttA^T)_{ab}+z(G^{(\mathbb{T}\backslash b)})_{bb}(\ttA\caG^{(\mathbb{T})}\ttB^T)_{ab}(\ttB\caG^{(\mathbb{T})}\ttA^T)_{bb}
		\\&~~~+z(G^{(\mathbb{T}\backslash a,b)})_{aa}(\ttA\caG^{(\mathbb{T})}\ttB^T)_{aa}\left[(\ttB\caG^{(\mathbb{T})}\ttA^T)_{ab}+z(G^{(\mathbb{T}\backslash b)})_{bb}(\ttB\caG^{(\mathbb{T})}\ttB^T)_{ab}\right]
		\\&~~~+z^2(G^{(\mathbb{T}\backslash a,b)})_{aa}(G^{(\mathbb{T}\backslash b)})_{bb}(\ttA\caG^{(\mathbb{T})}\ttB^T)_{ab}(\ttB\caG^{(\mathbb{T})}\ttB^T)_{ab}\times
		\\&~~~~~~\left[(\ttB\caG^{(\mathbb{T})}\ttA^T)_{ab}+z(G^{(\mathbb{T}\backslash b)})_{bb}(\ttB\caG^{(\mathbb{T})}\ttB^T)_{ab}\right].
		\end{split}
		\] 
		since 
		\[\begin{split}
		(\ttA\caG^{(\mathbb{T}\backslash a,b)}\ttB^T)_{aa}&=-z(G^{(\mathbb{T}\backslash a,b)})_{aa}(\ttA\caG^{(\mathbb{T}\backslash b)}\ttB^T)_{aa}\\&=-z(G^{(\mathbb{T}\backslash a,b)})_{aa}\left[(\ttA\caG^{(\mathbb{T})}\ttB^T)_{aa}+z(G^{(\mathbb{T}\backslash b)})_{bb}(\ttA\caG^{(\mathbb{T})}\ttB^T)_{ab}(\ttB\caG^{(\mathbb{T})}\ttB^T)_{ab}\right]
		\end{split}
		\]
		and
		\[
		(\ttB\caG^{(\mathbb{T}\backslash a,b)}\ttA^T)_{ab}=-z(G^{(\mathbb{T}\backslash a,b)})_{aa}\left[(\ttB\caG^{(\mathbb{T})}\ttA^T)_{ab}+z(G^{(\mathbb{T}\backslash b)})_{bb}(\ttB\caG^{(\mathbb{T})}\ttB^T)_{ab}\right].
		\] 
		The expansion of the first two monomials terminated since every term were maximally expanded. After this, for any fixed positive integer $\ell$, we expand the term which contains the term $(G^{(\mathbb{T}\backslash a,b)})_{aa}$ until the last term is a monomial containing $\ell$ or more off-diagonal entries by applying the first formula recursively to the not-maximally expanded diagonal entry $(G^{(\mathbb{T}\backslash a,b)})_{aa}.$
		\item For $a\neq b\in\mathbb{T}$
		\[
		\begin{split}
		(\ttA\caG^{(\mathbb{T}\backslash a,b)}\ttB^T)_{ab}&=-z(G^{(\mathbb{T}\backslash b)})_{bb}(\ttA\caG^{(\mathbb{T})}\ttB^T)_{ab}+\frac{(\ttA\caG^{(\mathbb{T}\backslash a,b)}\ttB^T)_{aa}(G^{(\mathbb{T}\backslash a,b)})_{ab}}{(G^{(\mathbb{T}\backslash a,b)})_{aa}}
		\\&=-z(G^{(\mathbb{T}\backslash b)})_{bb}(A\caG^{(\mathbb{T})}\ttB^T)_{ab}
		\\&~~~-z^2(G^{(\mathbb{T}\backslash a,b)})_{aa}(\ttA\caG^{(\mathbb{T})}\ttB^T)_{aa}(G^{(\mathbb{T}\backslash b)})_{bb}(\ttB\caG^{(\mathbb{T})}\ttB^T)_{ab}
		\\&~~~-z^3(G^{(\mathbb{T}\backslash a,b)})_{aa}(G^{(\mathbb{T}\backslash b)})_{bb}^2(\ttA\caG^{(\mathbb{T})}\ttB^T)_{ab}(\ttB\caG^{(\mathbb{T})}\ttB^T)_{ab}(\ttB\caG^{(\mathbb{T})}\ttB^T)_{ab}
		\end{split}
		\]
		Even in this case, we also expand the second and third monomials recursively by applying the first formula to not-maximally expanded diagonal entry $(G^{(\mathbb{T}\backslash a,b)})_{aa}.$
	\end{itemize} 
	In particular, we have two observations from the above operations. 
	\begin{itemize}
		\item The expansions of the maximally expanded off-diagonal entry consist of the monomials containing only an odd number of off-diagonal entries: $(\ttB\caG^{(\mathbb{T})}\ttA^T)_{ab}$, $(\ttA\caG^{(\mathbb{T})}\ttB^T)_{ab}$ and $(\ttB\caG^{(\mathbb{T})}\ttB^T)_{ab}.$
		\item The diagonal entries $(\ttA\caG^{(\mathbb{T})}\ttB^T)_{aa}=(\ttB\caG^{(\mathbb{T})}\ttA^T)_{aa}$ for $a\in\mathbb{T}$, appear in the expanded term by implementing operation (b) and (c). These terms can be interpreted as a loop of the vertex $a$ in the structure of the graph considered in \cite{Bloemendal-Erdos-Knowles-Yau-Yin2014}, since like the maximally expended diagonal entry, these terms are comparable to $ w_{\ttA\ttB}\mathfrak{s}(z)$ by the entrywise local law. Therefore, similar to the maximally expanded diagonal entry, terms of such types have no effect on the partial expectation techniques in subsection 5.13 of \cite{Bloemendal-Erdos-Knowles-Yau-Yin2014}. This part will be explained in more detail in the next step.
	\end{itemize}  
	
	As with the previous step, from the explanations depicted in each expansion formula, we can see that the above expansions eventually play the same role as operations (b) and (c) in \cite{Bloemendal-Erdos-Knowles-Yau-Yin2014}. 
	
	\subsubsection*{Step 3 : The further expansions for the maximally expanded diagonal entries}
	
	Finally, unless we end up with an expression that includes a sufficiently large numbers of off-diagonal resolvent entries (such \emph{trivial leaves} are dealt with separately in Subsection 5.11 of \cite{Bloemendal-Erdos-Knowles-Yau-Yin2014}), we need to expand the maximally expanded diagonal elements $(\ttA\caG^{(\mathbb{T})}\ttB^T)_{aa}=(\ttB\caG^{(\mathbb{T})}\ttA^T)_{aa}$ and $(G^{(\mathbb{T}\backslash a)})_{aa}$ for $a\in\mathbb{T}$ appearing in the \emph{non-trivial leaves} (cf. Subsection 5.12 $\sim$ 14 of \cite{Bloemendal-Erdos-Knowles-Yau-Yin2014}), where we need to slightly adjust the proof to the setting. These terms corresponds to the maximally expanded diagonal $G$-edge in \cite{Bloemendal-Erdos-Knowles-Yau-Yin2014}.
	
	First, for $c\in\mathbb{T}$,
	\beq\label{eq:shur comple}
	\frac{1}{(G^{(\mathbb{T}\backslash c)}_\ttB)_{cc}}=-z-z(\ttB\caG^{(\mathbb{T})}_\ttB\ttB^T)_{cc}.
	\eeq  
	We note that $|(\caG^{(\mathbb{T})})_{\mu\mu}-\mathfrak{s}(z)|\prec N^{-1/2}$ by following the proof of the entrywise local law.
	Using \eqref{eq:shur comple} and the facts $zs(z)\mathfrak{s}(z)=-(zs(z)+1)$ and $|\mathfrak{s}(z)|\asymp1,$ we see that
	\[
	\frac{1}{(G^{(\mathbb{T}\backslash c)})_{cc}}=\frac{1}{s(z)}-z\left((\ttB\caG^{(\mathbb{T})}\ttB^T)_{cc}-\mathfrak{s}(z)\right)
	\]
	and this implies that
	\[
	(G^{(\mathbb{T}\backslash c)})_{cc}=\sum_{k=0}^{\ell-1}(s(z))^{k+1}z^k\left((\ttB\caG^{(\mathbb{T})}\ttB^T)_{cc}-\mathfrak{s}(z)\right)^k+\caO_\prec(N^{-\ell/2})
	\]
	for any integer $\ell\geq1$ since $(\ttB\caG^{(\mathbb{T})}\ttB^T)_{cc}-\mathfrak{s}(z)$ is $\caO_\prec(N^{-1/2})$, by using Lemma \ref{lemma: LDE}.
	
	Similarly, 
	for $a\in\mathbb{T},$ we see that
	\beq\label{eq:diag cross}\begin{split}
		\frac{1}{(\ttA\caG^{(\mathbb{T})}\ttB^T)_{aa}}&=\frac{1}{w_{\ttA\ttB}\mathfrak{s}(z)}-\frac{(\ttA\caG^{(\mathbb{T})}_\ttB\ttB^T)_{aa}-w_{\ttA\ttB}\mathfrak{s}(z)}{w_{\ttA\ttB}\mathfrak{s}(z)(\ttA\caG^{(\mathbb{T})}_\ttB\ttB^T)_{aa}}
	\end{split}
	\eeq
	and so
	\[
	(\ttA\caG^{(\mathbb{T})}\ttB^T)_{aa}=w_{\ttA\ttB}\mathfrak{s}(z)-(\ttA\caG^{(\mathbb{T})}\ttB^T)_{aa}\frac{\frac{w_{\ttA\ttB}\mathfrak{s}(z)-(\ttA\caG^{(\mathbb{T})}\ttB^T)_{aa}}{w_{\ttA\ttB}\mathfrak{s}(z)}}{1-\frac{w_{\ttA\ttB}\mathfrak{s}(z)-(\ttA\caG^{(\mathbb{T})}\ttB^T)_{aa}}{w_{\ttA\ttB}\mathfrak{s}(z)}}.
	\]
	

	By using the estimate \[\begin{split}
	(\ttA\caG^{(\mathbb{T})}\ttB^T)_{aa}-w_{\ttA\ttB}\mathfrak{s}(z)&=\sum_{\mu\neq\nu}\ttA_{a\mu}(\caG^{(\mathbb{T})})_{\mu\nu}\ttB_{a\nu}+\sum_{\mu}\ttA_{a\mu}\ttB_{a\mu}\left((\caG^{(\mathbb{T})})_{\mu\mu}-\mathfrak{s}(z)\right)\\&~~~+\mathfrak{s}(z)\left(\frac1{N}\sum_{\mu}(N\ttA_{a\mu}\ttB_{a\mu}-w_{\ttA\ttB})\right)\prec N^{-1/2},
	\end{split}
	\] we have the following series expansion for any integers $\ell\geq1,$
	\[\begin{split}
	(\ttA\caG^{(\mathbb{T})}\ttB^T)_{aa}&=w_{\ttA\ttB}\mathfrak{s}(z)-(\ttA\caG^{(\mathbb{T})}\ttB^T)_{aa}\textbf{1}(\ell\geq2)\sum_{k=1}^{\ell-1}(w_{\ttA\ttB}\mathfrak{s}(z))^{-k}\left(w_{\ttA\ttB}\mathfrak{s}(z)-(\ttA\caG^{(\mathbb{T})}\ttB^T)_{aa}\right)^{k}
	\\&+\caO_\prec(N^{-\ell/2})
	\end{split}\]
	which corresponds to the term (5.42) in \cite{Bloemendal-Erdos-Knowles-Yau-Yin2014}.
	
	This way we end up with an expression where only contains the resolvent terms of the type $(\ttA\caG^{(\mathbb{T})}\ttA^T)_{ab}$, $(\ttA\caG^{(\mathbb{T})}\ttB^T)_{ab}$, $(\ttB\caG^{(\mathbb{T})}\ttA^T)_{ab}$ or  $(\ttB\caG^{(\mathbb{T})}\ttB^T)_{ab}=(G^{(\mathbb{T})})_{ab},$ for some $a\neq b\in\mathbb{T}.$ In other words, the $\bsx$ indices
	and the indices of the resolvent entries are completely decoupled; only explicit products of entries of $(\ttA,\ttB)$ represent the connections between them.
	
	\subsubsection*{Step 4 : Sketch of the rest of the proof.}	
	
	Through previous steps, for our case $(\ttA\caG \ttA^T)$, we observed the modified version of the operations, which are done for the resolvents $G$ and $\caG$ in \cite{Bloemendal-Erdos-Knowles-Yau-Yin2014}. 
	
	After with these modifications, it can be seen that the rest procedures (Step 6 $\sim$ 8 in \cite{Bloemendal-Erdos-Knowles-Yau-Yin2014}) of the proof for the non-trivial leaves with the stopping rule, which relies on the number of off-diagonal terms (cf. Definition 5.7 of \cite{Bloemendal-Erdos-Knowles-Yau-Yin2014}), are also valid for the $\mathcal{Z}_{\ttA\ttB}$. 
	
	More precisely, by using the entrywise laws and H\"{o}lder's inequality, the same estimation also holds for the trivial leave as in Subsection 5.11. Furthermore, the most of the finitely generated non-trivial leaves have a decay $N^{-p/2}$ also by applying the same argument in the case of the trivial leaves (Subsection 5.12 in \cite{Bloemendal-Erdos-Knowles-Yau-Yin2014}), and the remaining leading order non-trivial leaves have the same decay by applying the partial expectation method (Subsection 5.13 in \cite{Bloemendal-Erdos-Knowles-Yau-Yin2014}).  

	We conclude the proof.
\end{proof}
\begin{proof}[Proof of Lemma \ref{lem:local_MP2}]
	From the above version of an isotropic law, we also arrive at the isotropic version of the entrywise law in Lemma \ref{lem:entrywise_ecrm} by taking $\ttA=Q+X$ and $\ttB=Q.$ Then, it is easy to check that 
	\begin{align*}
	&w_\ttA=2\left(1+\frac{\ef}{\sqrt{\vf}}\right),&&w_{\ttA\ttB}=1+\frac{\ef}{\sqrt{\vf}},&&w_\ttB=1.
	\end{align*}
	Precisely, applying Lemma \ref{lem:isotropic coupled} directly, we see that 
	\[\begin{split}
	&2\langle\bsu, X(Q^TQ-zI)^{-1}Q^T\bsu\rangle
	\\&=\langle\bsu, X(Q^TQ-zI)^{-1}Q^T\bsu\rangle+\langle\bsu, Q(Q^TQ-zI)^{-1}X^T\bsu\rangle
	\\&=\langle\bsu, A(B^TB-zI)^{-1}A^T\bsu\rangle
	-\langle\bsu, B(B^TB-zI)^{-1}B^T\bsu\rangle-\langle\bsu, X(B^TB-zI)^{-1}X^T\bsu\rangle
	\\&=2\mathfrak{s}(z)\left(1+\frac{\ef}{\sqrt{\vf}}\right)+zs(z)\mathfrak{s}(z)^2\left(1+\frac{\ef}{\sqrt{\vf}}\right)^2-\mathfrak{s}(z)-zs(z)\mathfrak{s}(z)^2-\mathfrak{s}(z)-zs(z)\mathfrak{s}(z)^2\frac{\ef^2}{\vf}
	\\&=2\frac{\ef}{\sqrt{\vf}}(\mathfrak{s}(z)+zs(z)\mathfrak{s}(z)^2)=2\frac{\ef}{\sqrt{\vf}}(zs(z)+1)
	\end{split}
	\]
	with $\caO_\prec(N^{-\phi})$ error terms, and it exactly matches the entrywise law since correlation $w_{XQ}=\frac{\ef}{\sqrt{\vf}}$. Thus, we conclude that the improved PCA via the entrywise transform holds for the spike $\bsU$ s.t. $\|\bsU^T\bsU- I_k\|_F$, $\|\bsU\|_\infty\prec N^{-\phi}$, where $\phi>1/4.$
	
\end{proof}

\section{Proof of CLTs}\label{app:CLT}
%
%
%

In Appendix \ref{app:CLT}, we prove the CLT for the LSS of spiked random matrices. The proof of the CLT for the LSS is based on the strategy of \cite{Bai-Yao2005} in which the LSS is first written as a contour integral of the resolvent of a spiked Wigner matrix. Then, the averaged trace of the resolvent converges to a Gaussian process, which also implies that the limiting distribution of the LSS is Gaussian. 

It is the biggest obstacle in adapting the proof in \cite{Bai-Yao2005} for spiked matrices that the martingale CLT and covariance computation are hard to be reproduced with spikes; even with the special choice of rank-$1$ spike the proof for the CLT is very tedious as in \cite{Baik-Lee2017}. In \cite{chung2019weak}, the interpolation between a general rank-$1$ spike and the special rank-$1$ spiked was introduced to compare the LSS, based on an ansatz that the mean and the variance of the LSS do not depend on the choice of the spike. In this paper, since we do not have a reference matrix to be compared with as in the rank-$1$ case, we introduce a direct interpolation between a spiked random matrices of general rank and a matrix without any spikes. With the interpolation, we find the change of the mean in the limiting Gaussian distribution and also prove that its variance is invariant.
\subsection{Proof of CLTs for spiked random matrices}
\begin{proof}[Proof of Theorem \ref{thm:CLT}]
	We adapt the proof of Theorem 5 in \cite{chung2019weak} with the following change. Instead of interpolating the spiked Wigner matrices $M$ with the original signal and with the signal with all $1$'s considered in \cite{Baik-Lee2017}, we directly interpolate $M$ and $W$ and track the change of the mean. Consider the following interpolating matrix
	\[
	M(\theta) = \theta \sqrt{\lambda} \bsU \bsU^T + W
	\]
	and the corresponding eigenvalues $\{\mu_i(\theta)\}_{i=1}^{N}$ of $M(\theta)$ for $\theta \in [0, 1]$. 
	Let $\Gamma$ be a rectangular contour in the proof of Theorem 5 in \cite{chung2019weak}. Applying Cauchy's integral formula, we have 
	\beq \label{eq:Cauchy_int}
	\sum_{i=1}^N f(\mu_i(1)) - N \int_{-2}^2 \frac{\sqrt{4-x^2}}{2\pi} f(x) \, \dd x = -\frac{N}{2\pi \ii} \oint_{\Gamma} f(z) \big( s_N(1,z) - s_{sc}(z) \big) \dd z
	\eeq
	where $s_{sc}(z) = \frac{-z + \sqrt{z^2 - 4}}{2}$ is the Stieltjes transform of the Wigner semicircle law and $s_N(\theta,z)$ is the Stieltjes transform of the empirical spectral distribution (ESD) of $M(\theta)$ for $\theta\in[0,1]$.
	Note that the normalized trace of the resolvent satisfies
	\beq
	\frac{1}{N} \Tr R(\theta,z) = \frac{1}{N} \sum_{i=1}^N \frac{1}{\mu_i(\theta) -z} = s_N(\theta,z)
	\eeq
	where $R(\theta,z)$ is the resolvent corresponding to $M(\theta)$, defined as
	\beq\label{resolvent}
	R(\theta, z) := (M(\theta) - zI)^{-1}
	\eeq
	for $z \in \C^+$ and $\theta\in[0,1]$. 
	
	The change of the mean in the CLT for $W$ and the CLT for $M$ can be computed by tracking the change of the corresponding resolvent in \eqref{resolvent}, since \eqref{eq:Cauchy_int} can be decomposed by
	\begin{align}
	\sum_{i=1}^N f(\mu_i(1)) - N \int_{-2}^2 \frac{\sqrt{4-x^2}}{2\pi} f(x) \, \dd x 
	&= -\frac{1}{2\pi \ii} \oint_{\Gamma} f(z) \big( \Tr R(1,z) - \Tr R(0,z) \big) \dd z\label{eq:Trace difference}\\
	&~~~~-\frac{1}{2\pi \ii} \oint_{\Gamma} f(z) \big( \Tr R(0,z) - Ns_{sc}(z) \big) \dd z\label{null_case}
	\end{align} 
	and the fluctuation result of \eqref{null_case} is already given in \cite{Bai-Yao2005}.
	
	Set $\Gamma^{\varepsilon}=\{z\in\bbC:\min_{w\in\Gamma}|z-w|\leq\varepsilon\}.$ Choose $\varepsilon$ so that 
	\[
	\min_{w\in\Gamma^\varepsilon,x\in[-2,2]}|x-w|>2\varepsilon.
	\]
	Following the proof of Theorem 5 in \cite{chung2019weak}, on $z\in\Gamma^{\varepsilon}_{1/2}:=\Gamma^{\varepsilon}\cap \{z\in\bbC:\,|\text{Im} z|>N^{-1/2}\}$, we first find that
	\begin{align}
	\frac{\partial}{\partial \theta} \Tr R(\theta, z) &= -\sum_{m=1}^k \sqrt{\lambda} \frac{\partial}{\partial z} \left( \bsx(m)^T R(\theta, z) \bsu (m) \right) = -k \frac{\partial}{\partial z} \left( \frac{\sqrt{\lambda} s_{sc}(z)}{1+\theta \sqrt{\lambda} s_{sc}(z)} \right) + O(N^{-\frac{1}{2}}) \nonumber\\
	&= -\frac{k \sqrt{\lambda} s_{sc}'(z)}{(1+\theta \sqrt{\lambda} s_{sc}(z))^2} + O(N^{-\frac{1}{2}})\label{eq:derivative_of_trace}
	\end{align}
	with high probability. More precisely, since the elementary resolvent expansion implies
	\beq
	\begin{split}
		R(0,z)-R(\theta,z)=\theta\sqrt{\lambda}R(\theta,z)\left(\sum_{\ell=1}^{k}\bsu(\ell)\bsu(\ell)^T\right)R(0,z),
	\end{split}\eeq
	we then find that
	\begin{align}\label{res_exp}
	\left( \bsu(m)^T R(0, z) \bsu (m) \right)=\left( \bsu(m)^T R(\theta, z) \bsu (m) \right)+\theta\sqrt{\lambda}\sum_{\ell=1}^k\left( \bsu(m)^T R(\theta, z) \bsu (\ell) \right)\left( \bsu(\ell)^T R(0, z) \bsu (m) \right)\nonumber.
	\end{align} 
	From the rigidity of the eigenvalues, we have a deterministic bound for resolvent
	\beq
	|\left( \bsu(m)^T R(\theta, z) \bsu (\ell) \right)|\leq \lVert R(\theta, z)\rVert\leq C.
	\eeq
	Since columns of spike $\{\bsu(\ell)\}_{\ell=1}^{k}$ are orthonormal, the isotropic local law for $R(0,z)$ implies that 
	\beq
	\left( \bsu(m)^T R(0, z) \bsu (\ell) \right)=s(z)\delta_{m\ell}+\caO(N^{-1/2}).
	\eeq
	uniformly on $z\in\Gamma^\varepsilon.$
	We then obtain that 
	\begin{align*}
	\left( \bsu(m)^T R(0, z) \bsu (m) \right)=\left( \bsu(m)^T R(\theta, z) \bsu (m) \right)\left[1+\theta\sqrt{\lambda}\left( \bsu(m)^T R(0, z) \bsu (m) \right)\right]+O(N^{-\frac{1}{2}})\nonumber
	\end{align*}
	and so
	\begin{align}
	\left( \bsu(m)^T R(\theta, z) \bsu (m) \right)=\frac{s_{sc}(z)}{1+\theta\sqrt{\lambda}s_{sc}(z)}+O(N^{-\frac{1}{2}}).\nonumber
	\end{align}
	This proves \eqref{eq:derivative_of_trace}.
	
	Moreover, on $\Gamma^{\varepsilon},$ we easily check that the exactly same argument holds for a finite rank perturbation of Wigner matrix (e.g. interlacing and rigidity properties). Thus, we conclude that \eqref{eq:Trace difference} is 
	\[\frac{k}{2\pi\ii}\int_\Gamma \frac{\sqrt{\lambda}s_{sc}'(z)}{1+\sqrt{\lambda}s_{sc}(z)}f(z)d z+o(1)\]
	with high probability.
	
	Finally, following the computation in the proof of Lemma 4.4 in \cite{Baik-Lee2017}, we then find that the difference between the LSS of $M$ and the LSS of $W$ is
	\beq
	k \sum_{\ell=1}^{\infty} \sqrt{\lambda^{\ell}} \tau_{\ell}(f).
	\eeq
	This proves the desired theorem.
\end{proof}
\begin{proof}[Proof of Theorem \ref{thm:CLT_rec}]
	The proof of the CLT for the spiked rectangular matrices is quite similar to the case of spiked Wigner matrix. We first consider the interpolating matrix for the additive model, defined as 
	\beq 
	Y(\theta)=\theta\sqrt{\lambda}\bsU\bsV^T+X
	\eeq
	for $\theta \in [0, 1]$. Note that $Y(0) = X$ and $Y(1) = Y$. Denote by $\mu_1(\theta) \geq \mu_2(\theta) \geq \dots \geq \mu_M(\theta)$ the eigenvalues of $Y(\theta) Y(\theta)^T$. We also define the resolvent
	\beq
	G(\theta, z) = (Y(\theta)Y(\theta)^T -zI)^{-1}, \qquad \caG(\theta, z) = (Y(\theta)^T Y(\theta) -zI)^{-1}
	\eeq
	for $z \in \C$.
	
	We choose ($N$-independent) constants $a_- <d_-$, $a_+ > d_+$, and $v_0 \in (0, 1)$ so that the function $f$ is analytic on the rectangular contour $\Gamma$ whose vertices are $(a_- \pm \ii v_0)$ and $(a_+ \pm \ii v_0)$. With overwhelming probability, all eigenvalues of $Y(\theta)Y(\theta)^T$ are contained in $\Gamma$. Applying Cauchy's integral formula, we find that
	\beq
	\sum_{i=1}^M f(\mu_i(1)) - \sum_{i=1}^M f(\mu_i(0)) = -\left(\frac{1}{2\pi \ii} \oint_{\Gamma}f(z) \left( \Tr G(1,z) - \Tr G(0,z) \right) \dd z \right)
	\eeq
	To estimate the difference $\Tr G(1,z) - \Tr G(0,z)$, we consider its derivative $\frac{\partial}{\partial \theta}\Tr G(\theta,z)$. Note that
	\beq
	\frac{\partial G_{ab}(\theta)}{\partial Y_{ij}(\theta)}=-G_{ai}(\theta)(Y(\theta)^T G(\theta))_{jb}-(G(\theta)Y(\theta))_{aj}G_{ib}(\theta), \qquad 
	\frac{\dd Y_{ij}(\theta)}{\dd\theta}=\sqrt{\lambda}\bsu_i \bsv_j^T.
	\eeq
	Thus, by chain rule
	\beq \begin{split} \label{eq:der_Tr_mean}
		\frac{\partial}{\partial \theta}\Tr G(\theta,z)&=\sum_{a=1}^{M}\sum_{i=1}^{M}\sum_{j=1}^{N}\frac{\partial Y_{ij}(\theta)}{\partial\theta}\frac{\partial G_{aa}(\theta)}{\partial Y_{ij}(\theta)} \\
		&=-\sum_{a=1}^{M}\sum_{i=1}^{M}\sum_{j=1}^{N}\sqrt{\lambda}\bsu_i \bsv_k^T[G_{ai}(\theta)(Y(\theta)^T G(\theta))_{ja}+(G(\theta)Y(\theta))_{aj}G_{ia}(\theta)] \\
		&=-2\sum_{a=1}^{M}\sum_{i=1}^{M}\sum_{j=1}^{N}\sum_{b=1}^{M}\sqrt{\lambda}\bsu_i \bsv_j^T[Y_{b j}(\theta)G_{b a}(\theta)G_{ai}(\theta)]
	\end{split} \eeq
	From the fact 
	\[
	\left(\frac{\partial}{\partial z}G(\theta)\right)_{b i}=(G(\theta)^2)_{b i}=\sum_a G_{b a}(\theta)G_{ai}(\theta),
	\]
	we then find that
	\beq \begin{split} \label{eq:Tr_derivative}
		\frac{\partial}{\partial \theta}\Tr G(\theta,z) =-2\sqrt{\lambda}\frac{\partial}{\partial z}\sum_{i=1}^{M}\sum_{j=1}^{N}\bsu_i \bsv_j^T(G(\theta)Y(\theta))_{ij}=-2\sqrt{\lambda}\frac{\partial}{\partial z}\sum_{\ell=1}^k\langle \bsu(\ell),G(\theta)Y(\theta)\bsv(\ell)\rangle.
	\end{split} \eeq
	
	It remains to estimate $\frac{\partial}{\partial z}\langle \bsu(\ell),G(\theta)Y(\theta)\bsv(\ell)\rangle$ for $1\leq \ell\leq k$. We suffices to estimate the desired term for fixed $\ell.$ From now, we omit $\ell$-dependency. Note that
	\[
	\langle \bsu,G(\theta)Y(\theta)\bsv\rangle=\theta\sqrt{\lambda}\langle \bsu,G(\theta)\bsu\rangle+\langle \bsu,G(\theta)X\bsv\rangle.
	\]
	We consider the resolvent expansion 
	\beq\label{eq:resolvent_mean}\begin{split}
		G(0,z)-G(\theta,z)&=G(\theta,z)\,(H(\theta)-H(0))\,G(0,z) \\
		&=G(\theta,z)\,(\theta^2\lambda \bsu\bsu^T+\theta\sqrt{\lambda}X\bsv\bsu^T+\theta\sqrt{\lambda}\bsu\bsv^TX^T)\,G(0,z).
	\end{split}\eeq
	Taking inner products with $\bsu$ and $\bsv$, we obtain
	\beq\begin{split}\label{eq:ex1}
		\langle \bsu,G(0)\bsu\rangle&=\langle \bsu,G(\theta)\bsu\rangle+\theta^2\lambda\langle \bsu,G(\theta)\bsu\rangle\langle \bsu,G(0)\bsu\rangle\\&~~~+\theta\sqrt{\lambda}\langle \bsu,G(\theta)X\bsv\rangle\langle \bsu,G(0)\bsu\rangle+\theta\sqrt{\lambda}\langle\bsu,G(0)X\bsv\rangle\langle \bsu,G(\theta)\bsu\rangle
	\end{split}\eeq
	and
	\beq\begin{split}\label{eq:ex2}
		\langle \bsu,G(0)X\bsv\rangle&=\langle \bsu,G(\theta)X\bsv\rangle+\theta^2\lambda\langle \bsu,G(\theta)X\bsv\rangle\langle \bsu,G(0)X\bsv\rangle\\&~~~+\theta\sqrt{\lambda}\langle \bsu,G(\theta)X\bsv\rangle\langle \bsu,G(0)X\bsv\rangle+\theta\sqrt{\lambda}\langle \bsv,X^T G(0)X\bsv\rangle\langle \bsu,G(\theta)\bsu\rangle,
	\end{split}\eeq
	where we omitted $z$-dependence for brevity. We then use the following result to control the terms in \eqref{eq:ex1} and \eqref{eq:ex2}. Recall the definition of $s(z)$ and $\mathfrak{s}(z)$ in Lemmas \ref{lem:local_MP} and \ref{lem:entrywise_MPlaw}. Moreover, we consider the same linearization $H_X(z)$ of the matrix $X$ and its inverse $R_{X}(z)=H_{X}(z)^{-1}$ as in \eqref{eq:linearization} and \eqref{eq:resolvent_linearization}.
	
	\begin{lem}[Isotropic local law] \label{lem:local_law}
		For an $N$-independent constant $\varepsilon > 0$, let $\Gamma^{\varepsilon}$ be the $\varepsilon$-neighborhood of $\Gamma$, i.e.,
		\[
		\Gamma^{\varepsilon} = \{ z \in \mathbb{C} : \min_{w \in \Gamma} |z-w| \leq \varepsilon \}.
		\]
		Choose $\varepsilon$ small so that the distance between $\Gamma^{\varepsilon}$ and $[d_-, d_+]$ is larger than $2\varepsilon$, i.e., 
		\beq
		\min_{w \in \Gamma^{\varepsilon}, x \in [d_-, d_+]} |x-w| > 2\varepsilon.
		\eeq
		Then, for any unit vectors $\bsx,\bsy\in\mathbb{C}^{M+N}$ independent of $X$,
		\beq \begin{split} \label{eq:iso_spike}
			\left|\left\langle \bsx,(R_X(z)-\Pi(z)) \bsy\right\rangle\right|\prec N^{-1/2}, 
		\end{split} \eeq
		uniformly on $z\in \Gamma^\varepsilon$, where
		\beq 
		\Pi(z) =
		\begin{pmatrix}
			s(z)\cdot I_M & 0 \\
			0 & z\mathfrak{s}(z)\cdot I_N
		\end{pmatrix}.
		\eeq
	\end{lem}
	\begin{proof}
		See Theorems 3.6, 3.7, Corollary 3.9, and Remark 3.10 in \cite{Knowles-Yin2016}. Note that $\im \mathfrak{s}(z), \im s(z) = \Theta(\eta)$ on the vertical part of $\Gamma_{\varepsilon}$, i.e., the neighborhood of the line segment joining $(a_+ + \ii v_0)$ and $(a_+ - \ii v_0)$ (respectively $(a_- + \ii v_0)$ and $(a_- - \ii v_0)$).
	\end{proof}
	
	Set
	\[\begin{split}
	A:=\langle \bsu,G(0, z)\bsu\rangle,\qquad
	B:=\langle \bsu,G(0, z)X\bsv\rangle,\qquad
	C:=\langle \bsv,X^T G(0, z)X\bsv\rangle.
	\end{split}\]
	Recall that 
	\beq 
	R_X(z) =
	\begin{pmatrix}
		G(0,z) & G(0,z)X \\
		X^TG(0,z) & z\caG(0,z)
	\end{pmatrix}.
	\eeq
	Then, as consequences of Lemma \ref{lem:local_law} with appropriate choices of the deterministic vectors,
	\beq
	A= s(z) +\caO_\prec(N^{-1/2}), \qquad C =\langle \bsv, z\caG(0, z)\bsv\rangle+1+\caO(N^{-1/2})=\rat(zs(z)+1)+\caO_\prec(N^{-1/2}),
	\eeq
	and
	\[
	B=\caO_\prec(N^{-1/2}).
	\]
	
	We thus have from \eqref{eq:ex1} and \eqref{eq:ex2} that
	\beq \begin{split} \label{eq:iso_perturb}
		\langle \bsu,G(\theta)X\bsv\rangle &= -\frac{\theta\rat \sqrt{\lambda}s(z)(zs(z)+1)}{\theta^2\lambda zs(z)+\theta^2\lambda+1}+\caO_\prec(N^{-1/2}) \\
		\langle \bsu,G(\theta)\bsu\rangle &=\frac{s(z)}{\theta^2\lambda zs(z)+\theta^2\lambda+1}+\caO_\prec(N^{-1/2})
	\end{split} \eeq
	and hence
	\beq
	\langle \bsu,G(\theta)Y(\theta)\bsv\rangle =\theta\sqrt{\lambda}\langle \bsu,G(\theta)\bsu\rangle+\langle \bsu,G(\theta)X\bsv\rangle =\frac{\theta\sqrt{\lambda} zs(z)+\theta\sqrt{\lambda}}{\theta^2\lambda zs(z)+\theta^2\lambda+1}+\caO_\prec(N^{-1/2}).
	\eeq
	Note that this estimate is uniform on $\theta$.
	Differentiating it with respect to $z$ and plugging it back to \eqref{eq:Tr_derivative}, we get
	\[
	\frac{\partial}{\partial \theta}\Tr G(\theta,z)=-k\frac{2\theta\lambda\frac{\dd}{\dd z}(zs(z)+1)}{(\theta^2\lambda zs(z) +\theta^2\lambda+1)^2}+\caO_\prec(N^{-1/2})
	\] 
	and, integrating over $\theta$, we obtain
	\beq \label{eq:difference_Tr}
	\Tr G(1,z)-\Tr G(0,z)=\int_{0}^{1}\frac{\partial}{\partial \theta}\Tr G(\theta,z) \dd\theta=-k\frac{\frac{\dd}{\dd z}\lambda(zs(z)+1)}{\lambda zs(z)+\lambda+1}+\caO_\prec(N^{-1/2}).
	\eeq
	
	We now invoke the following relation between the Stieltjes transforms for Marchenko--Pastur law and the Wigner semicircle law. Let 
	\[
	s_{sc}(z) = \frac{-z+\sqrt{z^2 -4}}{2}
	\]
	be the Stieltjes transform of the Wigner semicircle law and 
	\[
	\varphi(z)=\frac1{\sqrt{\rat}}(z-(1+\rat)).
	\]
	Then
	\beq\label{eq:relation_MP_SC}
	\sqrt{\rat} (zs(z)+1)=s_{sc}(\varphi(z)).
	\eeq
	We thus have
	\beq \begin{split}
		\frac1{2\pi \ii}\oint_\Gamma f(z)\frac{\lambda\frac{\dd}{\dd z}(zs(z)+1)}{\lambda zs(z)+\lambda+1} \dd z &=\frac1{2\pi \ii}\oint_\Gamma \wt{f}(\varphi(z))\frac{\lambda s_{sc}'(\varphi(z))\varphi'(z)}{\lambda s_{sc}(\varphi(z))+\sqrt{\rat}} \dd z \\
		&=\frac1{2\pi \ii}\oint_{\wt\Gamma} \wt{f}(\varphi)\frac{\lambda s_{sc}'(\varphi)}{\lambda s_{sc}(\varphi)+\sqrt{\rat}} \dd\varphi
	\end{split} \eeq
	where we let $f(\sqrt{\rat}z+1+\rat)=\wt{f}(z)$ and $\wt\Gamma=\varphi(\Gamma)$. (Note that $\wt \Gamma$ contains the interval $[-2,2]$.)
	
	So far, we have proved that
	\beq \label{eq:mean_mean}
	\sum_{i=1}^M f(\mu_i(1)) - \sum_{i=1}^M f(\mu_i(0)) = \frac{k}{2\pi \ii}\oint_{\wt\Gamma} \wt{f}(\varphi)\frac{\lambda s_{sc}'(\varphi)}{\lambda s_{sc}(\varphi)+\sqrt{\rat}} \dd\varphi+\caO_\prec(N^{-1/2}). 
	\eeq
	Since the difference in \eqref{eq:mean_mean} is the sum of a deterministic term and a random term stochastically dominated by $N^{-1/2}$, we can see that the CLT holds for the LSS with the non-null model $Y(1)$. Moreover, the variance is the same as that of the null model, which is
	\beq
	V_Y(f) =2\sum_{\ell=1}^\infty \ell\tau_\ell(\wt{f})^2+(w_4-3)\tau_1(\wt{f})^2.
	\eeq
	(See, e.g., \cite{Baik-Lee2018}.)

	The change of the mean is the first term in the right side of \eqref{eq:mean_mean}, which can be computed by following the proof of Lemma 4.4 in \cite{Baik-Lee2017}. We obtain
	\beq
	m_Y(f) = \frac{\wt{f}(2) + \wt{f}(-2)}{4}  -\frac{1}{2} \tau_0(\wt{f})+(w_4-3)\tau_2(\wt{f})+k\sum_{\ell=1}^{\infty}\left(\frac{\lambda}{\sqrt{\rat}}\right)^{\ell}\tau_\ell(\wt{f}).
	\eeq
	This proves the first part of Theorem \ref{thm:CLT} for the additive model.

	For the multiplicative model, we will follow the same strategy as in the additive model. Let
	\beq 
	Y(\theta)= X+ \theta\gamma \bsU\bsU^T X
	\eeq
	for $\theta \in [0, 1]$. Note that $Y(0) = X$ and $Y(1) = Y$. We denote by $\mu_1(\theta) \geq \mu_2(\theta) \geq \dots \geq \mu_M(\theta)$ the eigenvalues of $Y(\theta) Y(\theta)^T$, and also let
	\beq
	G(\theta, z) = (Y(\theta)Y(\theta)^T -zI)^{-1}, \qquad \caG(\theta, z) = (Y(\theta)^T Y(\theta) -zI)^{-1}
	\eeq
	for $z \in \C$. We have the relations
	\beq \begin{split}
		\frac{\partial G_{ab}(\theta)}{\partial Y_{ij}(\theta)}=-G_{ai}(\theta)(Y(\theta)^T G(\theta))_{jb}-(G(\theta)Y(\theta))_{aj}G_{ib}(\theta), \qquad 
		\frac{\partial Y_{ij}(\theta)}{\partial\theta}=\gamma\sum_{c=1}^{M}\bsu_i\bsu_b^T X_{b j}.
	\end{split} \eeq
	
	Following \eqref{eq:der_Tr_mean}-\eqref{eq:Tr_derivative}, we get
	\beq \begin{split}
		\frac{\partial}{\partial \theta}\Tr G(\theta,z) &= -\gamma\sum_{a=1}^{M}\sum_{i=1}^{M}\sum_{j=1}^{N}\sum_{b=1}^{M}\bsu_i\bsu_b^T X_{b j}[G_{ai}(\theta)(Y(\theta)^T G(\theta))_{ja}+(G(\theta)Y(\theta))_{aj}G_{ia}(\theta)] \\
		&=-2\gamma\sum_{a=1}^{M}\sum_{i=1}^{M}\sum_{j=1}^{N}\sum_{b=1}^{M}\bsu_i\bsu_b^T X_{b j}[(Y(\theta)^TG(\theta))_{ja}G_{ai}(\theta)] \\
		&=-2\gamma\frac{\partial}{\partial z}\sum_{i=1}^{M}\sum_{j=1}^{N}\sum_{b=1}^{M}\bsu_i\bsu_b^T X_{b j}(G(\theta)Y(\theta))_{ij} \\
		&=-2\gamma\frac{\partial}{\partial z}
		\sum_{\ell=1}^k\langle \bsu(\ell),G(\theta)Y(\theta)X^T \bsu(\ell)\rangle=-2\gamma\frac{\partial}{\partial z}\sum_{\ell=1}^k\langle \bsu(\ell),G(\theta)Y(\theta)Y(0)^T \bsu(\ell)\rangle.
	\end{split} \eeq
	Moreover, since 
	\beq
	Y(0) = X = (I + \theta\gamma \bsU\bsU^T)^{-1} Y(\theta) = \left( I-\frac{\theta\gamma}{1+\theta\gamma}\bsU\bsU^T \right) Y(\theta),
	\eeq 
	we have
	\beq \begin{split}
		\langle \bsu(\ell),G(\theta)Y(\theta)Y(0)^T \bsu(\ell)\rangle &=\langle \bsu(\ell),G(\theta)Y(\theta)Y(\theta)^T (I + \theta\gamma \bsU\bsU^T)^{-1} \bsu(\ell)\rangle\\ 
		&=\langle \bsu(\ell),(I+zG(\theta))(I + \theta\gamma \bsU\bsU^T)^{-1} \bsu(\ell)\rangle \\
		&=\frac{1}{1+\theta\gamma}+\frac{z}{1+\theta\gamma}\langle \bsu(\ell),G(\theta)\bsu(\ell)\rangle.\label{eq:term_cov}
	\end{split} \eeq
	To estimate the term $\langle \bsu(\ell),G(\theta)\bsu(\ell)\rangle$, we use the following Anisotropic local law in \cite{Knowles-Yin2016}.
	
	\begin{lem}[Anisotropic local law] \label{lem:anisotropic}
		Let $\Gamma^{\varepsilon}$ be the $\varepsilon$-neighborhood of $\Gamma$ as in Lemma \ref{lem:local_law}. Then, for any unit vectors  $\bsx,\, \bsy \in \C^M$ independent of $X$, the following estimate holds uniformly on $z \in \Gamma^{\varepsilon}$ :
		\beq \label{eq:aniso_spike}
		\left|\left\langle \bsx, \left( G(\theta, z)+ \left( zI + z\mathfrak{s}(z) (I + \theta \gamma \bsU \bsU^T)^2 \right)^{-1} \right) \bsy\right\rangle \right| \prec N^{-\frac{1}{2}}.
		\eeq
	\end{lem}
	\begin{proof}
		The proof of Lemma \ref{lem:anisotropic} is the same as that of Lemma \ref{lem:local_law}.
	\end{proof}
	Now, as in the additive case, we drop the $\ell$-dependency. From Lemma \ref{lem:anisotropic}, we find that
	\beq\begin{split}
		\langle \bsu,G(\theta)\bsu\rangle &= -\left\langle \bsu, \left( zI + z\mathfrak{s}(z) (I + \theta \gamma \bsU \bsU^T)^2 \right)^{-1} \bsu \right\rangle  +\caO(N^{-1/2}) 
		\\&= -\frac{1}{(1+\theta\gamma)^2 z(1+\mathfrak{s}(z) )}+\caO(N^{-1/2}),
	\end{split}\eeq
	and plugging it into \eqref{eq:term_cov}, we obtain
	\beq\begin{split}
		\langle \bsu,G(\theta)Y(\theta)Y(0)^T \bsu\rangle=\frac{1}{1+\theta\gamma}-\frac{1}{(1+\theta\gamma)(1+ (1+\theta\gamma)^2 \mathfrak{s}(z) )}+\caO(N^{-1/2}).
	\end{split}
	\eeq
	We thus get
	\beq
	\frac{\partial}{\partial \theta}\Tr G(\theta,z)=-2k\gamma\frac{(1+\theta\gamma) \mathfrak{s}'(z)}{(1+ (1+\theta\gamma)^2 \mathfrak{s}(z) )^2}+\caO(N^{-1/2}),
	\eeq
	and integrating it yields 
	\beq \label{eq:difference_Tr_2}
	\begin{split}
		\Tr G(1,z)-\Tr G(0,z)&=-k\frac{\lambda\mathfrak{s}'(z)}{(1+\mathfrak{s}(z) )(1+(1+\lambda)\mathfrak{s}(z))}+\caO(N^{-1/2})
		\\&=-\frac{\lambda k \frac{\dd}{\dd z}(zs(z)+1)}{\lambda zs(z)+\lambda+1}+\caO(N^{-1/2}).
	\end{split}
	\eeq
	
	Since \eqref{eq:difference_Tr_2} coincides with \eqref{eq:difference_Tr}, the rest of the proof is exactly the same as in the additive case. This finishes the proof of the first part of Theorem \ref{thm:CLT}.
\end{proof}
\subsection{Proof of CLTs for entrywise transformed matrices}
\begin{proof}[Proof of Theorem \ref{thm:trans_CLT}]
	We adapt the proof of Theorem 7 in \cite{chung2019weak} with the following changes. Let $S$ be the variance matrix of the transformed matrix $\wt M.$ We then find that
	\[S_{ij}=\E[\wt M_{ij}^2]-(\E[\wt M_{ij}])^2=\frac1{N}+\lambda(\gh-\fh) (\bsu_i\bsu_j^T)^2+\caO(N^{1-8\phi})\]
	and
	\[S_{ii}=\E[\wt M_{ii}^2]-(\E[\wt M_{ii}])^2=\frac{w_2}{N}+\lambda(\ghd-\fhd) (\bsu_i\bsu_i^T)^2+\caO(N^{1-8\phi}).\]
	Normalizing and centering each entry of the matrix $\wt M$, we arrive at another Wigner matrix $\wt W$ where
	\begin{align*}
	&\wt W_{ij}=\frac{1}{\sqrt{NS_{ij}}}(\wt M_{ij}-\E\wt M_{ij}),&&\wt W_{ii}=\sqrt{\frac{w_2}{NS_{ii}}}(\wt M_{ii}-\E\wt M_{ii}).\end{align*}
	Interpolating $\wt W$ and $\wt M-\E[\wt M]$ by $\wt W(\theta)=(1-\theta)\wt W+\theta (\wt M-\E[\wt M])$, $\wt W(\theta)$ is a general Wigner-type matrix with the corresponding quadratic vector equation
	\[-\frac{1}{m_{i}(\theta,z)}=z+\sum_{j=1}^{N}\E[\wt W_{ij}(\theta)^2]\cdot m_j(\theta,z)\]
	where $m_i(\theta,z)\delta_{ij}$ is the limiting distribution of the $(i,j)$-element of the resolvent 
	\[R^{\wt W}(\theta,z)=(\wt W(\theta)-zI)^{-1}\]
	for $0\leq\theta\leq1.$
	Recall the $s_{sc}(z)$ is the Stieltjes transform of the Wigner semicircle law. We also directly check that $m_{i}(\theta,z)=s_{sc}(z)+C_1(\bsu_i\bsu_i^T)+C_2N^{-1}=s_{sc}(z)+\caO(N^{-2\phi}).$
	Moreover, the anisotropic local law for the general Wigner-type matrix in \cite{ajanki2017wignertype} implies that uniformly on $z\in\Gamma^{\varepsilon}_{1/2}$
	\[(\bsu(m)^TR^{\wt W}(\theta,z)\bsu(\ell))=s_{sc}(z)\delta_{m\ell}+\caO(N^{-1/2}).\]
	Following the proof of Lemmas B.2 and B.3 in \cite{chung2019weak}, we check that 
	\begin{itemize}
		\item Uniformly on $z\in\Gamma^{\varepsilon}_{1/2},$
		\beq\label{eq:Trace_diff_wigner}
		\Tr R^{\wt W}(1,z)-\Tr R^{\wt W}(0,z)=k\lambda(\gh-\fh)s_{sc}'(z)s_{sc}(z)+\caO(N^1N^{-4\phi})
		\eeq
		\item Uniformly on $z\in\Gamma^{\varepsilon}\backslash\Gamma^{\varepsilon}_{1/2},$
		\beq\label{eq:Rigidity_bound_wigner}
		|\Tr R^{\wt W}(1,z)-\Tr R^{\wt W}(0,z)|=\caO(N^{1}N^{-2/3}).
		\eeq
	\end{itemize} 
	Compared with the bound shown in \cite{chung2019weak}, we give the following remark:
	\begin{itemize}
		\item The error bound in \eqref{eq:Trace_diff_wigner} is better. This sharper bound can be obtained by using the fact $\sum_au_a(\ell)^2=1$ instead of $\left|\sum_au_a(\ell)^2\right|\leq N\|\bsu(\ell)\|_\infty^2.$
	\end{itemize} 
	
	Our next step is to consider $\wt M=\wt W(1)+\E[\wt M].$ Since 
	\[
	\wt M=\wt W(1)+\sqrt{\lambda \fh}\bsU\bsU^T+\diag(d_1,\cdots,d_N)+E
	\]
	where $d_i=\E[\wt M_{ii}]-\sqrt{\lambda\fh}(\bsU\bsU^T)_{ii},$ we then find that
	\[ \begin{split}
	&\Tr(\wt M-zI)^{-1}-\Tr R^{\wt W}(0,z)\\
	&=k\lambda(\gh-\fh)s_{sc}'(z)s_{sc}(z)-\frac{k\sqrt{\lambda\fh}s_{sc}'(z)}{1+\sqrt{\lambda\fh}s_{sc}(z)}-k\sqrt{\lambda}(\sqrt{\fhd}-\sqrt{\fh})s_{sc}'(z)+O(N^{-1/2})
	\end{split} \]
	uniformly on $z\in\Gamma^{\varepsilon}_{1/2}.$
	Thus, we obtain the desired CLT by applying Cauchy's integral formula as in the proof of Theorem \ref{thm:CLT}.
\end{proof}
\begin{proof}[Proof of Theorem \ref{thm:trans_CLT_rec}]
	Since the proof of the transformed CLT for the spiked Wigner matrix follows the proof in \cite{chung2019weak}, we only describe the process briefly. On the other hand, there is no technical reference for the spiked rectangular matrices. As we mentioned before, our consideration is only the additive case. 
	
	We consider the optimal entrywise transformation defined by a function
	\beq
	h(w) := -\frac{g'(w)}{g(w)}.
	\eeq
	If $\lambda = 0$, it is immediate to see that for all $i,j$
	\[
	\E[h(\sqrt{N} Y_{ij})] = \int_{-\infty}^{\infty} h(w) g(w) \dd w = -\int_{-\infty}^{\infty} g'(w) \dd w = 0.
	\]
	Further, with $\lambda = 0$, as shown in Proposition 4.2 of \cite{Perry2018},
	\beq
	\fh := \E[h(\sqrt{N} Y_{ij})^2] = \int_{-\infty}^{\infty} h(w)^2 g(w) \dd w = \int_{-\infty}^{\infty} \frac{g'(w)^2}{g(w)} \dd w \geq 1,
	\eeq
	where the equality holds if and only if $\sqrt{N} X_{ij}$ is a standard Gaussian (hence $h(w) = w$).
	
	We define a transformed matrix $\tY$ as follows: the terms of $\tY$ are defined by
	\beq
	\tY_{ij} = \frac{1}{\sqrt{\fh N}} h(\sqrt{N} Y_{ij}).
	\eeq
	Note that the entries of $\tY$ are independent up to symmetry. Since $g$ is smooth, $h$ is also smooth and all moments of $\sqrt{N} \tY_{ij}$ are $O(1)$. Thus, applying a high-order Markov inequality, it is immediate to find that $\tY_{ij} = \caO(N^{-\frac{1}{2}})$.

	\subsubsection{Decomposition of the transformed matrix}
	
	We first estimate the mean and the variance of entry by using the comparison method with the pre-transformed entries. For all $i,j$, we find that
	\beq \begin{split}\label{eq:Taylor_approx_tY}
		\E[\tY_{ij}] &= \frac{1}{\sqrt{\fh N}} \int_{-\infty}^{\infty} h(w) g\left(w - \sqrt{N\lambda}\bsu_i \bsv_j^T\right) \dd w \\
		&= -\frac{1}{\sqrt{\fh N}} \int_{-\infty}^{\infty} \frac{g'(w)}{g(w)} \left[ g\left(w - \sqrt{N\lambda}\bsu_i \bsv_j^T\right) - g(w) \right] \dd w.
	\end{split} \eeq
	In the Taylor expansion
	\beq \begin{split}
		&g\left(w - \sqrt{N\lambda}\bsu_i \bsv_j^T\right) - g(w) \\
		&= \sum_{\ell=1}^4 \frac{g^{(\ell)}(w)}{\ell!} \left( -\sqrt{N\lambda}\bsu_i \bsv_j^T \right)^{\ell}+ \frac{g^{(5)}\left(w - \theta \sqrt{N\lambda}\bsu_i \bsv_j^T\right)}{5!} \left( - \sqrt{N\lambda}\bsu_i \bsv_j^T\right)^5 
	\end{split} \eeq
	for some $\theta \in (0, 1)$. Note that the second term and the fourth term in the summation are even functions. Since $g'/g$ is an odd function, we find that
	\beq \begin{split} \label{eq:tY_mean}
		\E[\tY_{ij}] &= \frac1{\sqrt{\fh}}\sqrt{\lambda}\bsu_i \bsv_j^T  \int_{-\infty}^{\infty} \frac{g'(w)^2}{g(w)} \dd w + C_3 N \left(\sqrt{\lambda}\bsu_i \bsv_j^T\right)^3 + O(N^2 (\bsu_i \bsv_j^T)^5) \\
		&= \sqrt{\lambda\fh}\bsu_i \bsv_j^T + C_3 N \left(\sqrt{\lambda}\bsu_i \bsv_j^T\right)^3 + O(N^2 (\bsu_i \bsv_j^T)^5)
	\end{split} \eeq
	for some ($N$-independent) constant $C_3$.
	Similarly, since $\left(\frac{g'}{g}\right)^2$ is even,
	\beq \begin{split} \label{eq:tY_second}
		\E[\tY_{ij}^2] &= \frac{1}{\fh N} \int_{-\infty}^{\infty} \left( \frac{g'(w)}{g(w)} \right)^2 g\left(w - \sqrt{N\lambda}\bsu_i \bsv_j^T\right) \dd w \\
		&= \frac{1}{N} + \frac{1}{\fh N} \int_{-\infty}^{\infty} \left( \frac{g'(w)}{g(w)} \right)^2 \left( g\left(w - \sqrt{N\lambda}\bsu_i \bsv_j^T\right) - g(w) \right) \dd w \\
		&= \frac{1}{N} + \frac1{2\fh}\left(\sqrt{\lambda}\bsu_i \bsv_j^T\right)^2 \int_{-\infty}^{\infty} \frac{g'(w)^2 g''(w)}{g(w)^2} \dd w + O(N (\bsu_i \bsv_j^T)^4) \\
		&= \frac{1}{N} + \lambda\gh (\bsu_i \bsv_j^T)^2  + O(N (\bsu_i \bsv_j^T)^4).
	\end{split} \eeq
	where\[
	\gh = \frac{1}{2\fh} \int_{-\infty}^{\infty} \frac{g'(w)^2 g''(w)}{g(w)^2} \dd w.
	\]
	
	The evaluation of the mean and the variance shows that the transformed matrix $\tY$ is not a spiked rectangular matrix when $\lambda > 0$, since the variances of the entries are not identical. Our strategy is to approximate $\tY$ as a spiked generalized rectangular Gram matrix for which the variances of the each entries is $1/N$ in high-dimensional regime. 
	Let $S$ be the variance matrix of $\tY$ defined as
	\beq
	S_{ij} = \E[\tY_{ij}^2] - (\E[\tY_{ij}])^2.
	\eeq
	From \eqref{eq:tY_mean} and \eqref{eq:tY_second},
	\beq
	S_{ij} = \frac{1}{N} + (\gh-\fh)\left(\sqrt{\lambda}\bsu_i \bsv_j^T\right)^2+ O(N \| \bsU \|_{\infty}^{4}\| \bsV \|_{\infty}^{4}), 
	\eeq
	which shows that $\tY$ is indeed approximately a spiked generalized Gram matrix.
	\subsubsection{CLT for a random Gram matrix}
	
	We use the local law for general rectangular Gram matrices in \cite{alt2017gram}. Consider an another $M\times N$ rectangular matrix $A = (A_{ij})$ defined by
	\beq
	A_{ij} = \frac{1}{\sqrt{N S_{ij}}} (\tY_{ij} - \E[\tY_{ij}]).
	\eeq
	Note that $\E[A_{ij}] = 0$, $\E[A_{ij}^2] = \frac{1}{N}.$
	Then the matrix $A$ is a usual rectangular matrix. We set 
	\beq
	G^A(z) = (AA^T-zI)^{-1} \quad (z \in \bbC^+).
	\eeq
	Next, we introduce an interpolation for $A$. For $0 \leq \theta \leq 1$, we define a matrix $A(\theta)$ by
	\beq \begin{split} \label{eq:def_W_theta}
		A_{ij}(\theta) &= (1-\theta) A_{ij} + \theta (\tY_{ij} - \E[\tY_{ij}]) = \left( 1-\theta + \theta \sqrt{N S_{ij}} \right) A_{ij} \\
		&= \left( 1+ \frac{\theta N\lambda(\gh-\fh)(\bsu_i \bsv_j^T)^2}{2}+O(N^2 (\bsu_i\bsv_j^T)^4)\right)A_{ij}
	\end{split} \eeq
	
	Note that $A(0) = A$ and $A(1) = \tY - \E[\tY]$. For $0 \leq \theta \leq 1$, $A(\theta)$ is a random Gram matrix considered in \cite{alt2017gram} satisfying the conditions (A)--(D) therein. Moreover, if we let
	\beq \label{eq:def_G^A}
	G^A(\theta, z) = (A(\theta)A(\theta)^T-zI)^{-1} \quad (z \in \bbC^+)
	\eeq
	and $S_{ij}(\theta)=\E[A_{ij}(\theta)^2],$ then Theorem 1.7 of \cite{alt2017gram} asserts that the limiting distribution of $G^A_{ij}(z)$ is $s_i(z) \delta_{ij}$, where $s_i(\theta, z)$ is the unique solution to the system of quadratic vector equations
	\beq\label{eq:QVE1}\begin{split}
		-\frac{1}{s_i(\theta, z)} &= z + \sum_{j=1}^N S_{ij}(\theta)\; z\mathfrak{s}_j (\theta, z)
	\end{split}
	\eeq 
	and
	\beq\label{eq:QVE2}\begin{split}
		-\frac{1}{\mathfrak{s}_j(\theta, z)} &= z + \sum_{i=1}^M S_{ij}(\theta)\; zs_i (\theta, z)
	\end{split}
	\eeq
	\begin{rem}\label{rem:Gram_approx_sol}
		Recall that $s(z)$ is the Stieltjes transform of the Marchenko-Pastur measure. We can then find that $s_i(\theta, z) = s(z) + C_1(\bsu_i\bsu_i^T)+C_2N^{-1}=s(z)+\caO(N^{-1/2})$ and $\mathfrak{s}_j(\theta, z) = \mathfrak{s}(z) + C_1(\bsv_j\bsv_j^T)+C_2N^{-1}=\mathfrak{s}(z)+\caO(N^{-1/2})$; see also Lemma 3.9 of \cite{alt2017gram}.		
	\end{rem}

	For the resolvent $G^A(\theta, z)$, we will use the following lemma for the random Gram matrix:
	
	\begin{lem}[Anisotropic local law for random Gram matrix] \label{lem:anisotropic_Gram}
		Let $\Gamma^{\varepsilon}$ be the $\varepsilon$-neighborhood of $\Gamma$ as in Lemma \ref{lem:local_law}. Then, for any deterministic $\bsx = (x_1, \dots, x_M), \bsy = (y_1, \dots, y_M) \in \bbC^M$ with $\| \bsx \| = \| \bsy \| = 1$, the following estimate holds uniformly on $z \in \Gamma^{\varepsilon} \cap \{ z \in \bbC^+ : \im z > N^{-\frac{1}{2}} \}$:
		\beq \label{eq:aniso_spike1}
		\left|\sum_{i=1}^M\sum_{j=1}^M \overline{x_i} G^A_{ij}(\theta, z) y_j - \sum_{i=1}^M s_i(\theta, z) \overline{x_i} y_i \right| = \caO(N^{-\frac{1}{2}}).
		\eeq
		and, for any deterministic $\bsx = (x_1, \dots, x_N), \bsy = (y_1, \dots, y_N) \in \bbC^N$ with $\| \bsx \| = \| \bsy \| = 1$,
		\beq \label{eq:aniso_spike2}
		\left|\sum_{i=1}^N\sum_{j=1}^N \overline{x_i} \mathcal{G}^A_{ij}(\theta, z) y_j - \sum_{i=1}^N \mathfrak{s}_i(\theta, z) \overline{x_i} y_i \right| = \caO(N^{-\frac{1}{2}}).
		\eeq
	\end{lem}
	\begin{proof}[Proof of Lemma \ref{lem:anisotropic_Gram}]
		Let $\Psi(z)=\sqrt{\frac{1}{M\im z}}$ be the control parameter for the random gram matrix model. We then note that the bound for the entrywise local law is $N^{-1/2}$ since $\Psi(z)\prec N^{-1/2}$ on $\Gamma^{\varepsilon} \cap \{ z \in \bbC^+ : \im z > N^{-\frac{1}{2}} \}.$ With the entrywise local law in \cite{alt2017gram}, the proof of the anisotropic law exactly follows the \emph{maximal expansion} argument used in \cite{ajanki2017wignertype,Bloemendal-Erdos-Knowles-Yau-Yin2014} and Lemma \ref{lem:isotropic coupled}. We consider the following decomposition of \eqref{eq:aniso_spike1}:
		\beq\label{eq:aniso1_decomp}
		\sum_{i\neq j}^M \overline{x_i} G^A_{ij}(\theta, z) x_j + \sum_{i=1}^M (G^A_{ii}-s_i(\theta, z)) \overline{x_i} x_i.
		\eeq
		From now, we drop $A$, $\theta$ and $z$-dependencies for brevity and use the linearization matrix $H_{A(\theta)}(z)\equiv H$ and its inverse $R.$ Then, in usual, we suffices to prove that 
		\[
		\mathcal{Z}\equiv\sum_{a\neq b} \overline{x_a} R_{ab} x_b\prec N^{-1/2}.
		\] 
		To prove the above high probability bound, we will bound the $2p$-moments
		$\E[|\mathcal{Z}|^{2p}]\prec N^{-p/2}$ by deriving the \emph{maximally expanded} form via the resolvent identity in Lemma \ref{lem:res identity}. 
		
		Now, we will check the representation of the maximally expanded diagonal resolvent elements. e.g. $R_{bb}^{(\bbB\backslash b)}$, $b\in \bbB.$ Together with Remark \ref{rem:Gram_approx_sol}, we then conclude that the standard argument in \cite{ajanki2017wignertype} is valid for our model. By applying Shur's complement lemma and \eqref{eq:QVE1}, for $b\in \bbB,$ 
		\beq\label{eq:maximally_expanded_diag1}\begin{split}
			\frac1{R_{bb}^{(\bbB\backslash b)}}&=-z-\sum^{(\bbB)}_{\alpha,\beta}H_{b\alpha}R_{\alpha\beta}^{(\bbB)}H_{\beta b}
			\\&=\frac{1}{s_b(\theta, z)}+\sum_{\beta} S_{b\beta}\; (z\mathfrak{s}_\beta (\theta, z))-\sum^{(\bbB)}_{\alpha,\beta}H_{b\alpha}R_{\alpha\beta}^{(\bbB)}H_{\beta b}
			\\&=\frac{1}{s_b(\theta, z)}-\sum^{(\bbB)}_{\beta}(H_{b\beta}R_{\beta\beta}^{(\bbB)}H_{\beta b}-S_{b\beta}\; (z\mathfrak{s}_\beta (\theta, z)))-\sum^{(\bbB)}_{\alpha\neq\beta}H_{b\alpha}R_{\alpha\beta}^{(\bbB)}H_{\beta b}.
		\end{split}\eeq
		Then \eqref{eq:maximally_expanded_diag1} and the analogue representation of $R_{\beta\beta}^{(\bbB\backslash \beta)}$ for $\beta\in \bbB$ replace the (6.2) in \cite{ajanki2017wignertype}.
		
		With linearization $H$ and its inverse $R$, one useful by-product of the above argument is
		\beq\label{eq:off-diagonal}
		\langle\bsx,G^A(\theta)A(\theta) \bsy\rangle=\sum_{a}\sum_{\alpha}\overline{x_a} R_{a\alpha} y_\alpha\prec N^{-1/2}.
		\eeq
	\end{proof}
	Note that our model satisfies the closeness condition \textbf{(A3)} of Assumption 2.2 in \cite{alt2017singularities} (See also Remark 2.4 therein). On $\Gamma \backslash \Gamma^{\varepsilon}_{1/2}$, we use the following results on the rigidity of eigenvalues.
	\begin{lem}[Rigidity of eigenvalues for the random Gram matrix] \label{lem:rigidity}
		Denote by $\mu^A_1 (\theta) \geq \mu^A_2 (\theta) \geq \dots \geq \mu^A_M (\theta)$ the eigenvalues of $A(\theta)A(\theta)^T$. Let $\gamma_i$ be the classical location of the eigenvalues with respect to the Marchenko-Pastur measure defined by
		\beq
		\int_{\gamma_i}^{\infty} \rho_{MP,d_0}(\dd x) = \frac{1}{M} \left( i - \frac{1}{2} \right)
		\eeq
		for $i=1, 2, \dots, M$. Then,
		\beq
		|\mu^A_i (\theta) - \gamma_i| =\caO(M^{-2/3}).
		\eeq
	\end{lem}
	
	\begin{proof}
		Note that the rigidity of the eigenvalues with an error of at most $\caO(M^{-2/3})$ holds for random gram matrices at the classical location of the eigenvalues with respect to the probability measure $\rho$ from the Stieltjes transform $\texttt{s}_\rho(z):=\frac{1}{M}\sum_{i}s_i(z)$, see Lemma 4 in \cite{ding2022tracy}. Moreover, since $|s_i(\theta,z)-s(z)|,|\mathfrak{s}_j(\theta,z)-\mathfrak{s}(z)|=O(M^{-2\phi})$ for all $i$ and $j,$ we also have the desired rigidity near the classical location of Marchenko-Pastur law $\rho_{MP,\rat}$.
	\end{proof}
	\begin{rem}
		In fact, rigorous proofs of the rigidity and anisotropic law are not given in \cite{alt2017gram,alt2017singularities}. However, as in the proof of anisotropic local law for general Wigner-type matrix in \cite{ajanki2017wignertype}, the above lemmas may be proved by using the local laws in \cite{alt2017gram} and standard methods in \cite{ajanki2017wignertype} (Remark 2.10 in \cite{alt2017gram} and Remark 2.7 in \cite{alt2017singularities}.)
	\end{rem}
	On $\Gamma^{\varepsilon}_{1/2}$, as a simple corollary to Lemma \ref{lem:anisotropic_Gram}, we obtain
	\beq \label{eq:iso_general}
	\left| \langle \bsx, G^A(\theta, z) \bsy \rangle - s(z) \langle \bsx, \bsy \rangle \right| = \caO(N^{-\frac{1}{2}}),
	\eeq
	and 
	\beq \label{eq:iso_general2}
	\left| \langle \bsx, \mathcal{G}^A(\theta, z) \bsy \rangle - \mathfrak{s}(z) \langle \bsx, \bsy \rangle \right| = \caO(N^{-\frac{1}{2}}).
	\eeq
	
	We have the following lemma for the difference between $\Tr G^A(0, z)$ and $\Tr G^A(1, z)$ on $\Gamma^{\varepsilon}_{1/2}$.
	
	\begin{lem} \label{lem:Tr_G^A}
		Let $G^A(\theta, z)$ be defined as in Equations \eqref{eq:def_W_theta} and \eqref{eq:def_G^A}. Then, the following holds uniformly for $z \in \Gamma^{\varepsilon}_{1/2}$:
		\beq \label{eq:Y_inter_claim}
		\Tr G^A(1, z) - \Tr G^A(0, z) = -\lambda(\gh - \fh)k \frac{\partial}{\partial z}(zs(z)+ 1) + \caO(N \| \bsU \|_{\infty}^{2}\|\bsV \|_{\infty}^{2}).
		\eeq
	\end{lem}
	
	We will prove Lemma \ref{lem:Tr_G^A} later.

	From Lemma \ref{lem:rigidity}, we find that
	\beq \begin{split} \label{eq:Y_edge}
		|\Tr G^A(1, z) - \Tr G^A(0, z)| &= \left| \sum_{i=1}^N \left( \frac{1}{\mu^A_i(1) - z} - \frac{1}{\mu^A_i(0) - z} \right) \right| = \left| \sum_{i=1}^N \frac{\mu^A_i(0) - \mu^A_i(1)}{(\mu^A_i(1) - z)(\mu^A_i(0) - z)} \right| \\
		&\leq \left| \sum_{i=1}^N \frac{|\mu^A_i(0) - \gamma_i| + |\gamma_i - \mu^A_i(1)|}{(\mu^A_i(1) - z)(\mu^A_i(0) - z)} \right| = \caO(N^{1/3})
	\end{split} \eeq
	uniformly for $z\in\Gamma.$
	Thus, from \eqref{eq:Y_inter_claim} and \eqref{eq:Y_edge},
	\beq \begin{split} \label{eq:tY_chain_1}
		&\frac{1}{2\pi \ii} \oint_{\Gamma} f(z) \Tr G^A(1, z) \dd z - \frac{1}{2\pi \ii} \oint_{\Gamma} f(z) \Tr G^A(0, z) \dd z \\
		&= \frac{1}{2\pi \ii} \oint_{\Gamma^{\varepsilon}_{1/2}} f(z) \left( \Tr G^A(1, z) - \Tr G^A(0, z) \right) \dd z + \frac{1}{2\pi \ii} \oint_{\Gamma \backslash \Gamma^{\varepsilon}_{1/2}} f(z) \left( \Tr G^A(1, z) - \Tr G^A(0, z) \right) \dd z \\
		&=- \frac{\lambda(\gh - \fh)k}{2\pi \ii} \oint_{\Gamma^{\varepsilon}_{1/2}} f(z) \frac{\partial}{\partial z}(zs(z)+ 1) \dd z + \caO(N \| \bsU \|_{\infty}^{2}\|\bsV \|_{\infty}^{2}) + \caO(N^{-1/6}) \\
		&= -\frac{\lambda(\gh - \fh)k}{2\pi \ii} \oint_{\Gamma} f(z) \frac{\partial}{\partial z}(zs(z)+ 1) \dd z  + o(1).
	\end{split} \eeq
	Furthermore, using the relation \eqref{eq:relation_MP_SC}, we have
	\[\begin{split}
	\frac1{2\pi \ii} \oint_{\Gamma} f(z) \frac{\partial}{\partial z}(zs(z)+ 1) \dd z&=\frac{1}{2\pi \ii} \oint_{\Gamma} f(z) \frac{1}{\sqrt{\rat}}s_{sc}'(\varphi(z))\varphi'(z) \dd z
	\\&=\frac{1}{2\sqrt{\rat}\pi \ii} \oint_{\wt\Gamma} \wt f(\varphi)s_{sc}'(\varphi)\dd\varphi
	=\frac{1}{\sqrt{\rat}}\tau_1(\wt{f}).
	\end{split}\]

	\subsubsection{CLT for a random Gram matrix with a spike and small perturbation}
	
	Recall that $A(1) = \tY - \E[\tY]$. 
	Our next step in the approximation is to consider $\tY = A(1) + \E[\tY]$. 
	%
	Since $\E[\tY]$ is not a matrix of rank $k$, we instead consider 	
	\beq
	B(\theta) = A(1) + \theta \sqrt{\lambda\fh}\bsU\bsV^T, \qquad G^B(\theta, z) = (B(\theta)B(\theta)^T-zI)^{-1}
	\eeq
	%
	To prove this part of CLT, we will adapt the strategy for the proof of Theorem \ref{thm:CLT_rec} 
	with Lemmas \ref{lem:anisotropic_Gram} and \ref{lem:rigidity}. We then find that, uniformly for $z \in \Gamma^{\varepsilon}_{1/2},$
	\beq
	\Tr G^B(1,z)-\Tr G^B(0,z)=-k\frac{\frac{\dd}{\dd z}\lambda\fh(zs(z)+1)}{\lambda\fh zs(z)+\lambda\fh+1}+\caO_\prec(N^{-\phi}),
	\eeq
	since $\|\bsU^T\bsU-I_k\|_F,\|\bsV^T\bsV-I_k\|_F\prec N^{-\phi}.$
	Using the rigidity (Lemma \ref{lem:rigidity}) and the eigenvalue interlacing property, we have 
	\[
	\Tr G^{B'}(1,z)-\Tr G^{B'}(0,z)=\caO(1)\qquad\text{on $\Gamma\backslash \Gamma_{1/2}$}
	\]
	and so
	\beq \begin{split} \label{eq:tY_chain_2}
		&\frac{1}{2\pi \ii} \oint_{\Gamma} f(z) \Tr G^B(1, z) \dd z - \frac{1}{2\pi \ii} \oint_{\Gamma} f(z) \Tr G^B(0, z) \dd z \\
		&= -\frac{k}{2\pi \ii} \oint_{\Gamma} f(z) \frac{\lambda\fh\frac{d}{dz}(zs+1)}{\lambda\fh (zs+1)+1} \dd z + o(1).
	\end{split} \eeq
	The remaining part is to control an effect of small perturbation $(\bbE[\tY]- \sqrt{\lambda\fh}\bsU\bsV^T)_{ij}=CN(\bsu_i\bsv_j^T)^3+\caO(N^2(\bsu_i\bsv_j^T)^5)$. First, we let 
	\beq
	B' = B(1) +  CN(\bsu_i\bsv_j^T)^3, \qquad G^{B'}( z) = (B'(B')^T-zI)^{-1}
	\eeq
	For $1\leq\ell_1,\ell_2,\ell_3\leq k,$ we consider vectors $\bsu^3$ and $\bsv^3$ such that 
	\[
	(\bsu^3(\ell_1,\ell_2,\ell_3))_i=u^3_i(\ell_1,\ell_2,\ell_3):=\sqrt{N}u_i(\ell_1)u_i(\ell_2)u_i(\ell_3)
	\]
	and 
	\[
	(\bsv^3(\ell_1,\ell_2,\ell_3))_j=v^3_j(\ell_1,\ell_2,\ell_3):=\sqrt{N}v_j(\ell_1)v_j(\ell_2)v_j(\ell_3).
	\]
	We then observe that $B'$ contains $k^3$ additional small spikes: 
	\[
	B'=B(1)+C\sum_{\ell_1,\ell_2,\ell_3}\bsu^3(\ell_1,\ell_2,\ell_3)\bsv^3(\ell_1,\ell_2,\ell_3)^T
	\]
	where $\|\bsu^3(\ell_1,\ell_2,\ell_3)\|_\infty,\|\bsv^3(\ell_1,\ell_2,\ell_3)\|_\infty\prec N^{1/2-3\phi}.$
	
	In the above point of view, we are able to consider $B'$ as another spiked Gram matrix model with two types of spikes $\bsu(\ell)\bsv(\ell)^T$ and $\bsu^3(\ell_1,\ell_2,\ell_3)(\bsv^3(\ell_1,\ell_2,\ell_3))^T.$ As before, for $0\leq\theta\leq1,$ let 
	\[
	B'(\theta) = A(1) + \theta \sqrt{\lambda\fh}\sum_\ell\bsu(\ell)\bsv(\ell)^T+\theta C\sum_{\ell_1,\ell_2,\ell_3}\bsu^3(\ell_1,\ell_2,\ell_3)(\bsv^3(\ell_1,\ell_2,\ell_3))^T\]
	and
	\[G^{B'}(\theta, z) = (B'(\theta)B'(\theta)^T-zI)^{-1}.
	\]
	Following the proof of Theorem \ref{thm:CLT_rec}, we have 
	\[\begin{split}
	\frac{\partial}{\partial\theta}\Tr G^{B'}(\theta, z)&=-2\sqrt{\lambda\fh}\frac{\partial}{\partial z}\sum_\ell\langle\bsu(\ell),G^{B'}(\theta, z)B'(\theta)\bsv(\ell)\rangle
	\\&~~~-2C\frac{\partial}{\partial z}\sum_{\ell_1,\ell_2,\ell_3}\langle\bsu^3(\ell_1,\ell_2,\ell_3),G^{B'}(\theta, z)B'(\theta)\bsv^3(\ell_1,\ell_2,\ell_3)\rangle,
	\end{split}\]
	and it can be observed that the first term of the right-hand side is the leading order term, since $\|\bsu^3(\ell_1,\ell_2,\ell_3)\|_\infty,\|\bsv^3(\ell_1,\ell_2,\ell_3)\|_\infty\prec N^{1/2-3\phi}<N^{-\phi}$.
	Moreover, from the definition of $B'(\theta),$ the leading order term of $\langle\bsu(\ell),G^{B'}(\theta, z)B'(\theta)\bsv(\ell)\rangle$ 
	is $\langle\bsu(\ell),G^{B'}(\theta, z)B'(0)\bsv(\ell)\rangle+\theta\sqrt{\lambda\fh}\langle\bsu(\ell),G^{B'}(\theta, z)\bsu(\ell)\rangle,$
	since $\langle\bsv(\ell_1),\bsv(\ell_2)\rangle=\delta_{\ell_1\ell_2}+\caO(N^{-\phi})$, $\langle\bsv(\ell),\bsv^3(\ell_1,\ell_2,\ell_3)\rangle=\caO(N^{1/2-2\phi})$ and $\langle\bsv^3(\ell_1,\ell_2,\ell_3),\bsv^3(\ell_4,\ell_5,\ell_6)\rangle=\caO(N^{1-4\phi}).$ 
	Carrying out the remaining procedures presented in the proof of Theorem \ref{thm:CLT_rec} and collecting the leading order terms, we eventually obtain
	
	\beq\label{eq:Trace_approx_Gram1}
	\Tr G^{B'}(1,z)-\Tr G^{B'}(0,z)=-k\frac{\frac{\dd}{\dd z}\lambda\fh(zs(z)+1)}{\lambda\fh zs(z)+\lambda\fh+1}+\caO_\prec(N^{1/2-2\phi})
	\eeq
	uniformly for $z \in \Gamma^{\varepsilon}_{1/2}.$ Here, we last apply Lemma \ref{lem:anisotropic_Gram} and \eqref{eq:off-diagonal} for $B'(0)=A(1)$. Further, on $\Gamma\backslash \Gamma_{1/2},$ from the rigidity and the interlacing property of the eigenvalues, 
	\beq\label{eq:Trace_approx_Gram2}
	\Tr G^{B'}(1,z)-\Tr G^{B'}(0,z)=\caO(1).
	\eeq
	Thus, we conclude that
	\[\begin{split}
	&\frac{1}{2\pi \ii} \oint_{\Gamma} f(z) \Tr G^{B'}(z) \dd z - \frac{1}{2\pi \ii} \oint_{\Gamma} f(z) \Tr G^{B'}(0, z) \dd z
	\\&~~= -\frac{k}{2\pi \ii} \oint_{\Gamma} f(z) \frac{\lambda\fh\frac{d}{dz}(zs+1)}{\lambda\fh (zs+1)+1} \dd z + o(1).
	\end{split}\]
	Furthermore, we set $E_{ij}=(\tY-B')_{ij}=\caO(N^2(\bsu_i\bsv_j^T)^5).$ Then 
	\beq\label{eq:Error_bound_Gram}
	\|E\|\leq\|E\|_F=O(N^2\|\bsU\|_\infty^4\|\bsV\|_\infty^4)=o(N^{-1})
	\eeq for some $\phi>3/8.$ This implies that
	\[
	\frac{1}{2\pi \ii} \oint_{\Gamma} f(z) \Tr G^{\tY}(z) \dd z - \frac{1}{2\pi \ii} \oint_{\Gamma} f(z) \Tr G^{B'}(1,z) \dd z=o(1).
	\]
	\begin{rem}
		Under the assumption that $\phi>3/8$, we suffices to consider $\E[\tY_{ij}]$ up to $\caO(N^2(\bsu_i\bsv_j^T)^5)$ error. However, \eqref{eq:Trace_approx_Gram1} and \eqref{eq:Trace_approx_Gram2} are valid for any finite approximation of $\E[\tY]$ as presented in \eqref{eq:tY_mean}, even for any $\phi>1/4.$ This means that the condition $\phi>3/8$ can be improved by considering a higher order expansion of $\E[\tY].$ For example, if we consider 
		\[
		\E[\tY]=\sqrt{\lambda\fh}\bsu_i\bsv_j^T+C_1N(\bsu_i\bsv_j^T)^3+C_2N^2(\bsu_i\bsv_j^T)^5+\caO(N^3(\bsu_i\bsv_j^T)^7),
		\]
		then it can be checked that the contributions of the second and third terms are negligible, and the error $E_{ij}=\caO(N^3(\bsu_i\bsv_j^T)^7)$ is also negligible if $\phi>1/3$, since 
		\[
		\|E\|\leq\|E\|_F=\caO(N^3\|\bsU\|_\infty^6\|\bsV\|_\infty^6)\prec N^{3-12\phi}=o(N^{-1}).
		\]
	\end{rem}

	\subsubsection{Conclusion for the proof of pre-transformed CLT}
	
	We are now ready to prove pre-transformed CLT. Denote by $\wt\mu_1 \geq \wt\mu_2 \geq \dots \geq \wt\mu_N$ the eigenvalues of $\wt Y\wt Y^T$. Recall that we denoted by $\mu^A_1 (0) \geq \mu^A_2 (0) \geq \dots \geq \mu^A_N (0)$ the eigenvalues of $A(0)A(0)^T$. From Cauchy's integral formula, we have
	\beq \begin{split} \label{eq:tY_chain_final}
		&\sum_{i=1}^M f(\wt\mu_i) - M \int_{d_-}^{d_+}  f(x) \, \rho_{MP,d_0}(\dd x) \\
		&=\left( \sum_{i=1}^M f(\mu^A_i (0)) - \int_{d_-}^{d_+}  f(x) \, \rho_{MP,d_0}(\dd x) \right) + \left( \sum_{i=1}^M f(\wt\mu_i) - \sum_{i=1}^M f(\mu^A_i (0))\right) \\
		&=\left( \sum_{i=1}^M f(\mu^A_i (0)) - M\int_{d_-}^{d_+}  f(x) \, \rho_{MP,d_0}(\dd x) \right) - \left( \frac{1}{2\pi \ii} \oint_{\Gamma} f(z) \Tr G^{\tY}(z) \dd z - \frac{1}{2\pi \ii} \oint_{\Gamma} f(z) \Tr G^A(0, z) \dd z \right).
	\end{split} \eeq
	Since $AA^\ast$ is a usual sample covariance matrix, the first term in the right-hand side converges to a Gaussian random variable. Further, as computed in \eqref{eq:tY_second},
	\[\E[\wt Y_{ij}^4]=:\frac{\wt{w_4}}{N^2} + \frac{1}{(N\fh)^2} \int_{-\infty}^{\infty} \left(\frac{g'(w)}{g(w)}\right)^4\left( g\left(w - \sqrt{N\lambda}\bsu_i \bsv_j^T\right) - g(w) \right) \dd w,\]
	where the first term is the leading term of $\E[\tY_{ij}^4]$ and hence the leading term of $\E[A_{ij}^4]$ as well. This means that the difference between $\wt w_4$ and $\E[A_{ij}^4]$ is negligible in the sense that it has no contribution in the limiting behavior of the resolvent, which can be checked from standard Green function comparison theorems. (Refer to  \cite{ding2022tracy}.)	
	
	Thus, the mean and the variance of the limiting Gaussian distribution are given by
	\beq
	m_A(f) = \frac{\wt{f}(2) + \wt{f}(-2)}{4}  -\frac{1}{2} \tau_0(\wt{f})+(\wt{w_4}-3)\tau_2(\wt{f})
	\eeq
	and
	\beq
	V_A(f) =2\sum_{\ell=1}^\infty \ell\tau_\ell(\wt{f})^2+(\wt{w_4}-3)\tau_1(\wt{f})^2,
	\eeq
	respectively.
	
	For the second term in the right-hand side of \eqref{eq:tY_chain_final}, by \eqref{eq:tY_chain_2}, we obtain that
	\beq \begin{split}
		&\frac{1}{2\pi \ii} \oint_{\Gamma} f(z) \Tr G^{\tY}(z) \dd z - \frac{1}{2\pi \ii} \oint_{\Gamma} f(z) \Tr G^A(0, z) \dd z \\
		&=
		-\frac{k}{2\pi \ii} \oint_{\Gamma} f(z) \frac{\lambda\fh\frac{d}{dz}(sz+1)}{\lambda\fh (sz+1)+1} \dd z + o(1)
	\end{split} \eeq
	with high probability.
	From \eqref{eq:tY_chain_final}, we thus find that the CLT for the LSS holds, i.e.,
	\beq \label{eq:CLT_entry}
	\left( \sum_{i=1}^M f(\wt\mu_i) - M \int_{d_-}^{d_+}  f(x) \, \rho_{MP,d_0}(\dd x) \right) \to \caN(m_{\tY}(f), V_{\tY}(f)),
	\eeq
	and the variance $V_{\tY}(f) = V_A(f)$ since the second term in \eqref{eq:tY_chain_final} converges to a deterministic value as $N \to \infty$, which corresponds to the change of the mean. In particular,
	\beq \begin{split} \label{eq:mean_change}
		m_{\tY}(f) - m_A(f) = \frac{(\gh - \fh)\lambda k}{2\pi \ii} \oint_{\Gamma} f(z) \frac{\partial}{\partial z}(zs(z)+ 1) \dd z +\frac{k}{2\pi \ii} \oint_{\Gamma} f(z) \frac{\lambda\fh\frac{d}{dz}(sz+1)}{\lambda\fh (sz+1)+1} \dd z.
	\end{split} \eeq
	Following the computation in the proof of Lemma 4.4 in \cite{Baik-Lee2017} with the relation \eqref{eq:relation_MP_SC}, we find that the right-hand side of \eqref{eq:mean_change} is given by
	\beq \begin{split}
		&\frac{k}{2\pi \ii} \oint_{\Gamma} f(z) (zs(z)+1)' \left[ \lambda (\gh - \fh) + \frac{\lambda\fh}{\lambda\fh(zs(z)+1)+1}   \right] \dd z \\
		&=\frac{\lambda k}{\sqrt{\rat}}(\gh - \fh) \tau_1 (\wt{f})  +k\sum_{\ell=1}^{\infty}\left(\frac{\lambda\fh}{\sqrt{\rat}}\right)^{\ell}\tau_\ell(\wt{f}).
	\end{split} \eeq
	(See also Remark 1.7 of \cite{Baik-Lee2017}.) Thus, 
	\beq
	m_{\wt{Y}}(f) = \frac{\wt{f}(2) + \wt{f}(-2)}{4}  -\frac{1}{2} \tau_0(\wt{f})+\frac{\lambda k}{\sqrt{\rat}}(\gh - \fh) \tau_1 (\wt{f})+(\wt{w_4}-3)\tau_2(\wt{f})+k\sum_{\ell=1}^{\infty}\left(\frac{\lambda\fh}{\sqrt{\rat}}\right)^{\ell}\tau_\ell(\wt{f})
	\eeq
	and
	\beq
	V_{\wt{Y}}(f) =2\sum_{\ell=1}^\infty \ell\tau_\ell(\wt{f})^2+(\wt{w_4}-3)\tau_1(\wt{f})^2.
	\eeq
\end{proof}
\subsection{Proof of Lemma \ref{lem:Tr_G^A}}\label{apx:recursive_lemma}
\subsubsection*{Notational remarks}

In the rest of the section, we use $C$ order to denote a constant that is independent of $N$. Even if the constant is different from one place to another, we may use the same notation $C$ as long as it does not depend on $N$ for the convenience of the presentation. Now, we recall the linearization $H_{A(\theta)}(z)$ and its inverse $R_{A}(\theta, z)=H_{A(\theta)}(z)^{-1}.$ For simplicity, we drop the subscript $A$ and index $z$ of the linearization entries.

\begin{proof}[Proof of Lemma \ref{lem:Tr_G^A}]
	To prove the lemma, we consider
	\beq \begin{split} \label{eq:derivative_G^A}
		\frac{\partial}{\partial \theta} \Tr G^A(\theta, z) &=\sum_{b}\sum_{a}\sum_{\alpha}\frac{\partial A_{a\alpha}(\theta)}{\partial\theta}\frac{\partial G_{bb}(\theta)}{\partial A_{a\alpha}(\theta)}
		\\&=\sum_{b}\sum_{a}\sum_{\alpha}\frac{\partial H_{a\alpha}(\theta)}{\partial\theta}\frac{\partial R_{bb}(\theta)}{\partial H_{a\alpha}(\theta)}
		\\&=-\sum_{b}\sum_{a}\sum_{\alpha}\frac{\partial H_{a\alpha}(\theta)}{\partial\theta}[R_{ba}(\theta)R_{\alpha b}(\theta)+R_{b\alpha}(\theta)R_{ab}(\theta)]
		\\&=-2\sum_{a}\sum_{\alpha}\frac{\partial H_{a\alpha}(\theta)}{\partial\theta}(R(\theta)^2)_{a\alpha}
		\\&=-2\sum_{a}\sum_{\alpha}\frac{\partial H_{a\alpha}(\theta)}{\partial\theta} \frac{\partial}{\partial z} R_{a\alpha}(\theta),
	\end{split} \eeq
	where we again used that $\frac{\partial}{\partial z} G^A(\theta, z) = G^A(\theta, z)^2$. We expand the right-hand side by using the definition of $A(\theta)$,
	\beq
	H_{a\alpha}(\theta)=A_{a\alpha}(\theta) = \left( 1-\theta + \theta \sqrt{N S_{a\alpha}} \right) A_{a\alpha}(0)=\left( 1-\theta + \theta \sqrt{N S_{a\alpha}} \right) H_{a\alpha}(0),
	\eeq
	and so
	\beq \begin{split} \label{eq:A_inter_1}
		\sum_{a}\sum_{\alpha}\frac{\partial H_{a\alpha}(\theta)}{\partial\theta}R_{a\alpha}(\theta) &= \sum_{a}\sum_{\alpha} \left( -1 + \sqrt{N S_{a\alpha}} \right) H_{a\alpha}(0) R_{a\alpha}(\theta)
		\\&=  \sum_{a}\sum_{\alpha} \frac{-1+\sqrt{N S_{a\alpha}}}{1-\theta + \theta \sqrt{N S_{a\alpha}}}H_{a\alpha}(\theta) R_{a\alpha}(\theta)
		\\&=  \frac{N\lambda(\gh-\fh) }{2} \sum_{a}\sum_{\alpha} (\bsu_a\bsv_\alpha^T)^2H_{a\alpha}(\theta) R_{a\alpha}(\theta)+ \caO( N \| \bsU \|_{\infty}^{2}\| \bsV \|_{\infty}^{2}).
	\end{split} \eeq
	From now, we further drop the $\theta$-dependency for the brevity.
	
	Then 
	\[
	\frac{\partial}{\partial \theta} \Tr G^A(\theta, z)=-N\lambda(\gh-\fh) \frac{\partial}{\partial z} \sum_{a}\sum_\alpha (\bsu_a\bsv_\alpha^T)^2 H_{a\alpha} R_{a\alpha}+ \caO( N \| \bsU \|_{\infty}^{2}\| \bsV \|_{\infty}^{2}).
	\]
	Here, we used the properties that $H_{a\alpha}=A_{ab}(\theta) = \caO(N^{-\frac{1}{2}})$, $R_{ab}=G^A_{ba}(\theta)=\caO(N^{-\frac{1}{2}})$ for $b \neq a$, $R_{aa}=G^A_{aa}(\theta) = \caO(1)$, and $\sum_a u_a(\ell_1)u_{a}(\ell_2) =\delta_{\ell_1\ell_2}= \sum_\alpha v_\alpha(\ell_1)v_\alpha(\ell_2)$, which imply
	\beq\begin{split}
		\left| N^2 \sum_{a}\sum_\alpha (\bsu_a\bsv_\alpha^T)^4 H_{a\alpha} R_{a\alpha} \right| \leq N^2 \| \bsU \|_{\infty}^{2}  \| \bsV \|_{\infty}^{2}\sum_{a}\sum_{\alpha}(\bsu_a\bsv_\alpha^T)^2 |H_{a\alpha} R_{a\alpha}| =\caO(N \| \bsU \|_{\infty}^{2}  \| \bsV\|_{\infty}^{2}).
	\end{split}\eeq
	Together with Remark \ref{rem:Gram_approx_sol}, from the elementary equality for $R$ and $H$, we have
	\beq \begin{split}
		\sum_a\sum_{\alpha}u_a(\ell)^2 H_{a\alpha} R_{a\alpha} &= \sum_{a}u_a(\ell)^2 \left(\sum_{\alpha} H_{a\alpha} R_{a\alpha}\right)\\&= \sum_{a} u_a(\ell)^2 (1+zR_{aa})\\&=1 + zs(z) + \caO(N^{-\frac{1}{2}}),
	\end{split} \eeq
	and
	\beq \begin{split}
		\sum_{a}\sum_{\alpha}v_\alpha(\ell)^2  H_{a\alpha} R_{a\alpha} &= \sum_{\alpha}v_\alpha(\ell)^2  \left(\sum_{a}H_{a\alpha} R_{a\alpha}\right) \\
		&=\sum_{\alpha} v_\alpha(\ell)^2 (1+R_{\alpha\alpha})\\
		&= \rat(1 + zs(z)) + \caO(N^{-\frac{1}{2}}).
	\end{split} \eeq
	Plugging them into \eqref{eq:A_inter_1}, we get
	\beq \begin{split} \label{eq:A_inter_2}
		&\frac{4}{\lambda(\gh-\fh)}\times\eqref{eq:A_inter_1}\\&= N\sum_\ell\sum_a\sum_{\alpha}\Big\{\frac1{N}u_a(\ell)^2 H_{a\alpha} R_{a\alpha}+\frac1{M}v_\alpha(\ell)^2 H_{a\alpha} R_{a\alpha}\\
		&~~~+\left(u_a(\ell)^2-\frac1{M}\right)v_{\alpha}(\ell)^2 H_{a\alpha} R_{a\alpha}+u_a(\ell)^2\left(v_\alpha(\ell)^2-\frac1{N}\right) H_{a\alpha} R_{a\alpha}\Big\}\\
		&~~~+2N\sum_{\ell_1\neq\ell_2}\sum_a\sum_{\alpha}u_a(\ell_1)u_a(\ell_2)v_\alpha(\ell_1)v_\alpha(\ell_2) H_{a\alpha} R_{a\alpha}+ \caO(N \| \bsU \|_{\infty}^{2}\| \bsV \|_{\infty}^{2})\\&
		=N\sum_\ell\sum_a\sum_{\alpha}\Big\{\left(u_a(\ell)^2-\frac1{M}\right)v_{\alpha}(\ell)^2 H_{a\alpha} R_{a\alpha}+u_a(\ell)^2\left(v_\alpha(\ell)^2-\frac1{N}\right) H_{a\alpha} R_{a\alpha}\Big\}\\
		&~~~+2N\sum_{\ell_1\neq\ell_2}\sum_a\sum_{\alpha}u_a(\ell_1)u_a(\ell_2)v_\alpha(\ell_1)v_\alpha(\ell_2) H_{a\alpha} R_{a\alpha}
		\\&~~~+2k(zs(z)+1)+ \caO(N \| \bsU \|_{\infty}^{2}\| \bsV \|_{\infty}^{2}).
	\end{split} \eeq
	
	It remains to estimate the first three terms in \eqref{eq:A_inter_2}. Set
	\beq \label{eq:X_def}
	\ttX_1 \equiv \ttX_1(\theta, z, \ell) := \sum_a\sum_{\alpha}\left(u_a(\ell)^2-\frac1{M}\right)v_{\alpha}(\ell)^2 H_{a\alpha} R_{a\alpha},
	\eeq
	\beq \label{eq:X_2_def}
	\ttX_2 \equiv \ttX_2(\theta, z, \ell) := \sum_a\sum_{\alpha}u_a(\ell)^2\left(v_\alpha(\ell)^2-\frac1{N}\right) H_{a\alpha} R_{a\alpha}
	\eeq
	and
	\beq \label{eq:X_3_def}
	\ttX_3 \equiv \ttX_3(\theta, z, \ell_1, \ell_2) := \sum_a\sum_{\alpha}u_a(\ell_1)u_a(\ell_2)v_\alpha(\ell_1)v_\alpha(\ell_2) H_{a\alpha} R_{a\alpha}\qquad(\ell_1\neq\ell_2).
	\eeq
	We notice that $|\ttX_1|,|\ttX_2|,|\ttX_3| = \caO(N^{-1})$ on $\Gamma_{1/2}$ by a naive power counting as in \eqref{eq:A_inter_1} after applying H\"older inequality once.
	To obtain a better bound, we use a method based on a recursive moment estimate, introduced in \cite{LeeSchnelli2018}. We need the following lemma:
	
	\begin{lem} \label{lem:recursive}
		Let $\ttX_1,\;\ttX_2$ and $\ttX_3$ be as in \eqref{eq:X_def}, \eqref{eq:X_2_def} and \eqref{eq:X_3_def}. Define an event $\Omega_{\varepsilon}$ by
		\[
		\Omega_{\varepsilon}=\bigcap_{a,b,\alpha,\beta}  \{ |H_{a\alpha}|,|R_{a\alpha}| \leq N^{-\frac{1}{2}+\varepsilon} \}\cap\{ |R_{ab} -  s(z)\delta_{ab}| \leq N^{-\frac{1}{2}+\varepsilon}\}\cap\{ |R_{\alpha\beta} -  z\mathfrak{s}(z)\delta_{\alpha\beta}| \leq
		N^{-\frac{1}{2}+\varepsilon}\}
		\]
		Then, for any fixed (large) $D$ and (small) $\varepsilon$, which may depend on $D$, 
		\beq \begin{split} \label{eq:recursive}
			\E[|\ttX|^{2D} | \Omega_{\varepsilon}] &\leq C N^{-\frac{1}{2}+\varepsilon} \| \bsu \|_{\infty}^{2} \E[|\ttX|^{2D-1} | \Omega_{\varepsilon}] + C N^{-1+4\varepsilon} \| \bsu \|_{\infty}^{4} \E[|\ttX|^{2D-2} | \Omega_{\varepsilon}] \\
			&\quad + C N^{-2+10\varepsilon} \| \bsu \|_{\infty}^{6} \E[|\ttX|^{2D-3} | \Omega_{\varepsilon}] + C N^{-3+14\varepsilon} \| \bsu \|_{\infty}^{8} \E[|\ttX|^{2D-4} | \Omega_{\varepsilon}],
		\end{split} \eeq
		where $\ttX$ is $\ttX_1$, $\ttX_2$ and $\ttX_3$.
	\end{lem}
	Since the rank of the signal $k$ is fixed, we suffices to prove the above lemma for fixed $\ell,$ $\ell_1$ and $\ell_2.$
	We will prove Lemma \ref{lem:recursive} for $\ttX_1$ at the end of this section (the calculation for the $\ttX_2$ and $\ttX_3$ is almost the same). With Lemma \ref{lem:recursive}, we are ready to obtain an improved bound for $\ttX$. First, note that the contribution from the exceptional event $\Omega_{\varepsilon}^c$ is negligible i.e., $\p(\Omega_{\varepsilon}^c) < N^{-D^2}$, which can be checked by applying a high-order Markov inequality with the moment condition on $\tY$ (See Assumption \ref{assump:entry1}). We decompose $\E[|\ttX|^{2D}]$ by
	\beq \begin{split} \label{eq:X_decompose}
		\E[|\ttX|^{2D}] = \E[|\ttX|^{2D} \cdot \mathbf{1}(\Omega_{\varepsilon})] + \E[|\ttX|^{2D} \cdot \mathbf{1}(\Omega_{\varepsilon}^c)] = \E[|\ttX|^{2D} | \Omega_{\varepsilon}] \cdot \p(\Omega_{\varepsilon}) + \E[|\ttX|^{2D} \cdot \mathbf{1}(\Omega_{\varepsilon}^c)].
	\end{split} \eeq
	Then the second term in the right-hand side of \eqref{eq:X_decompose},
	\beq \label{eq:Omega^c_1}
	\E[|\ttX|^{2D} \cdot \mathbf{1}(\Omega_{\varepsilon}^c)] \leq \left( \E[|\ttX|^{4D} ] \right)^{\frac{1}{2}} \left( \p(\Omega_{\varepsilon}^c) \right)^{\frac{1}{2}} \leq N^{-\frac{D^2}{2}} \left( \E[|\ttX|^{4D} ] \right)^{\frac{1}{2}} 
	\eeq
	and by using a trivial bound for the resolvent $|R_{ab}(z)| \leq \| G^A(z) \| \leq \frac{1}{\im z}$
	\beq \label{eq:Omega^c_2}
	\begin{split}
		\E[|\ttX|^{4D}] \leq \E\left( \sum_a\sum_{\alpha}| H_{a\alpha} R_{a\alpha}| \right)^{4D}\leq \frac{(M^2N)^{4D}}{(\im z)^{4D}}\max_{a,b,\alpha}\E|H_{a\alpha}H_{b\alpha}|^{4D}\leq CN^{14D}.
	\end{split}
	\eeq
	
	To bound the right-hand side of \eqref{eq:recursive}, we use Young's inequality: For any $a, b > 0$ and $p, q > 0$ with $\frac{1}{p} + \frac{1}{q} = 1,$
	\[
	ab \leq \frac{a^p}{p} + \frac{b^q}{q}.
	\]
	We then find that the first term has the following upper bound
	\beq \begin{split}
		N^{-\frac{1}{2}+\varepsilon} \| \bsu \|_{\infty}^{2} |\ttX|^{2D-1} &= N^{\frac{(2D-1)\varepsilon}{2D}} N^{-\frac{1}{2}+\varepsilon} \| \bsu \|_{\infty}^{2} \cdot N^{-\frac{(2D-1)\varepsilon}{2D}} |\ttX|^{2D-1} \\
		&\leq \frac{1}{2D} N^{(2D-1)\varepsilon} (N^{-\frac{1}{2}+\varepsilon} \| \bsu \|_{\infty}^{2})^{2D} + \frac{2D-1}{2D} N^{-\varepsilon} |\ttX|^{2D}.
	\end{split} \eeq
	Applying Young's inequality for other terms in \eqref{eq:recursive}, we get
	\beq \begin{split}
		\E[|\ttX|^{2D} | \Omega_{\varepsilon}] &\leq C N^{(2D-1)\varepsilon} (N^{-\frac{1}{2}+\varepsilon} \| \bsu \|_{\infty}^{2})^{2D} + C N^{(D-1)\varepsilon} (N^{-1+4\varepsilon} \| \bsu \|_{\infty}^{4})^D \\
		& \quad + C N^{(\frac{2D}{3}-1)\varepsilon} (N^{-2+10\varepsilon} \| \bsu \|_{\infty}^{6})^{\frac{2D}{3}} + C N^{(\frac{D}{2}-1)\varepsilon} (N^{-3+14\varepsilon} \| \bsu \|_{\infty}^{8})^{\frac{D}{2}} \\&\quad+ C N^{-\varepsilon} \E[|\ttX|^{2D} | \Omega_{\varepsilon}].
	\end{split} \eeq
	Absorbing the last term in the right-hand side to the left-hand side and plugging the estimates \eqref{eq:Omega^c_1} and \eqref{eq:Omega^c_2} into \eqref{eq:X_decompose}, we now get
	\beq \begin{split}
		\E[|\ttX|^{2D}] &\leq C N^{(2D-1)\varepsilon} (N^{-\frac{1}{2}+\varepsilon} \| \bsu \|_{\infty}^{2})^{2D} + C N^{(D-1)\varepsilon} (N^{-1+4\varepsilon} \| \bsu \|_{\infty}^{4})^D \\
		& \quad + C N^{(\frac{2D}{3}-1)\varepsilon} (N^{-2+10\varepsilon} \| \bsu \|_{\infty}^{6})^{\frac{2D}{3}} + C N^{(\frac{D}{2}-1)\varepsilon} (N^{-3+14\varepsilon} \| \bsu \|_{\infty}^{8})^{\frac{D}{2}} +  CN^{-\frac{D^2}{2} + 7D}.
	\end{split} \eeq
	From the $(2D)$-th order Markov inequality, for any fixed $\varepsilon' > 0$ independent of $D$, 
	\beq
	\p \big(|\ttX| \geq N^{\varepsilon'} N^{-\frac1{2}} \| \bsu \|_{\infty}^{2} \big) \leq N^{-2D\varepsilon'} \frac{\E[|\ttX|^{2D}]}{( N^{-\frac1{2}}\| \bsu \|_{\infty}^{2})^{2D}} \leq N^{-2D\varepsilon'} N^{8D\varepsilon}.
	\eeq
	By choosing $\varepsilon = 1/D$, for sufficiently large $D$, we find that
	\beq\label{eq:bound for X}
	|\ttX| = \caO(N^{-\frac1{2}} \| \bsu \|_{\infty}^{2}).
	\eeq
	
	We now return to \eqref{eq:derivative_G^A} and use \eqref{eq:A_inter_2} with the bound \eqref{eq:bound for X}, 
	\beq
	\sum_{j=1}^M\sum_{k=1}^N \frac{\partial A_{jk}(\theta)}{\partial \theta}(G^A(\theta)A(\theta))_{jk} = \frac{(\gh - \fh)\lambda k}{2}(1+zs(z)) + \caO(N \| \bsu \|_{\infty}^{2}\|\bsv \|_{\infty}^{2}).
	\eeq
	To handle the derivative of the right-hand side, we use Cauchy's integral formula with a rectangular contour, contained in $\Gamma^{\varepsilon}_{1/2}$, whose perimeter is larger than $\varepsilon$. Then, we get from \eqref{eq:derivative_G^A} that
	\beq
	\frac{\partial}{\partial \theta} \Tr G^A(\theta, z) = -\lambda(\gh - \fh)\cdot k\frac{\partial}{\partial z}(1+zs(z)) + \caO(N \| \bsu \|_{\infty}^{2}\|\bsv \|_{\infty}^{2}).
	\eeq
	After integrating over $\theta$ from $0$ to $1$, we conclude that \eqref{eq:Y_inter_claim} holds for a fixed $z \in \Gamma^{\varepsilon}_{1/2}$.
\end{proof}

At last, we prove Lemma \ref{lem:recursive}.

\begin{proof}[Proof of Lemma \ref{lem:recursive}]
	As we mentioned above, we consider $\ttX=\ttX_1$ and drop the $\ell$-dependency. i.e.
	\[
	\E[|\ttX|^{2D}] = \E\left[ \sum_a\sum_{\alpha}\left(u_a^2-\frac1{M}\right)v_{\alpha}^2 H_{a\alpha}R_{a\alpha} \ttX^{D-1} \overline{\ttX}^{D} \right]
	\]
	%
	We use the following inequality that generalizes Stein's lemma (see Proposition 5.2 of \cite{BaikLeeWu}):
	Let $\Phi$ be a $C^2$ function. Fix a (small) $\varepsilon > 0$, which may depend on $D$. Recall that $\Omega_{\varepsilon}$ is the complement of the exceptional event on which $|H_{a\alpha}|$ or $|R_{a\alpha}|$ is exceptionally large for some $a, \alpha$, defined by
	$\Omega_{\varepsilon}$ by 
	\[
	\bigcap_{a,b,\alpha,\beta}  \{ |H_{a\alpha}|,|R_{a\alpha}| \leq N^{-\frac{1}{2}+\varepsilon} \}\cap\{ |R_{ab} -  s(z)\delta_{ab}| \leq N^{-\frac{1}{2}+\varepsilon}\}\cap\{ |R_{\alpha\beta} -  z\mathfrak{s}(z)\delta_{\alpha\beta}| \leq
	N^{-\frac{1}{2}+\varepsilon}\}
	\]
	Then,
	\beq \label{eq:Stein}
	\E[H_{a\alpha} \Phi(H_{a\alpha}) | \Omega_{\varepsilon}] = (\E[H_{a\alpha}^2|\Omega_{\varepsilon}]-\E[H_{a\alpha}|\Omega_{\varepsilon}]^2) \E[\Phi'(H_{a\alpha}) | \Omega_{\varepsilon}] + \varepsilon_1,
	\eeq
	where the error term $\varepsilon_1$ admits the bound
	\beq \label{eq:Stein_error}
	|\varepsilon_1| \leq C_1 \E \Big[ |H_{a\alpha}|^3 \sup_{|t| \leq 1} \Phi''(t H_{a\alpha}) \Big| \Omega_{\varepsilon} \Big]
	\eeq
	for some constant $C_1$. Note that by applying a decomposition \eqref{eq:X_decompose} to $\E[H_{a\alpha}| \Omega_{\varepsilon}]$ and $\E[H_{a\alpha}^2| \Omega_{\varepsilon}]$, we see that 
	\beq\label{eq:mean_estimate}
	\E[H_{a\alpha}| \Omega_{\varepsilon}]-\E(\E[H_{a\alpha}| \Omega_{\varepsilon}])=\E[H_{a\alpha}| \Omega_{\varepsilon}]=\caO(N^{-D_0})
	\eeq
	and
	\beq\label{eq:variance_estimate}
	\E[H_{a\alpha}^2| \Omega_{\varepsilon}]=\E(\E[H_{a\alpha}^2| \Omega_{\varepsilon}])+\caO(N^{-D_0})=\frac1{N}+\caO(\| \bsu \|_{\infty}^{2}\| \bsv \|_{\infty}^{2})+\caO(N^{-D_0})
	\eeq
	for $D_0=\frac{D^2+1}{2}>1.$ The estimate \eqref{eq:Stein} follows from the proof of Proposition 5.2 of \cite{BaikLeeWu} with $p=1$, where we use the inequality (5.38) therein only up to second to the last line.
	
	In the estimate \eqref{eq:Stein}, we choose
	\beq
	\Phi(H_{a\alpha}) = R_{a\alpha} \ttX^{D-1} \overline{\ttX}^{D}
	\eeq
	so that
	\beq \begin{split} \label{eq:X_expansion}
		\E[|\ttX|^{2D} | \Omega_{\varepsilon}] = \sum_{a}\sum_{\alpha}  \left(u_a^2-\frac1{M}\right)v_\alpha^2 \E \left[ H_{a\alpha} \Phi(H_{a\alpha}) | \Omega_{\varepsilon} \right].
	\end{split} \eeq

	Applying \eqref{eq:mean_estimate} and \eqref{eq:variance_estimate} to the equation \eqref{eq:Stein},
	\beq \begin{split} \label{eq:Phi_expansion}
		&\E \left[ H_{a\alpha} \Phi(H_{a\alpha}) | \Omega_{\varepsilon} \right] = \E \left[ H_{a\alpha}^2 \right] \E[\Phi'(H_{a\alpha}) | \Omega_{\varepsilon}] + \varepsilon_1 \\
		&= \E[H_{a\alpha}^2] \left( -\E \left[R_{aa} R_{\alpha\alpha} \ttX^{D-1} \overline{\ttX}^{D} | \Omega_{\varepsilon} \right] - \E \left[R^2_{a\alpha} \ttX^{D-1} \overline{\ttX}^{D} | \Omega_{\varepsilon} \right] \right. \\
		&\qquad \left. +(D-1) \E \left[ R_{a\alpha} \frac{\partial \ttX}{\partial H_{a\alpha}} \ttX^{D-2} \overline{\ttX}^{D} \big| \Omega_{\varepsilon} \right] + D \E \left[ R_{a\alpha} \frac{\partial \overline{\ttX}}{\partial H_{a\alpha}} \ttX^{D-1} \overline{\ttX}^{D-1} \big| \Omega_{\varepsilon} \right] \right) + \varepsilon_1,
	\end{split} \eeq
	for sufficiently large $D$.
	We plug it into \eqref{eq:X_expansion} and estimate each term. Then the term originated from the first term in \eqref{eq:Phi_expansion} can be separated by
	\beq \begin{split}
		&\sum_a\sum_{\alpha}  \left(u_a^2-\frac1{M}\right)v_\alpha^2\E[H_{a\alpha}^2] \E \left[R_{aa} R_{\alpha\alpha} \ttX^{D-1} \overline{\ttX}^{D} | \Omega_{\varepsilon} \right] \\
		&= \sum_a\sum_{\alpha}  \left(u_a^2-\frac1{M}\right)v_\alpha^2\E[H_{a\alpha}^2] \E \left[(R_{aa}-s) R_{\alpha\alpha} \ttX^{D-1} \overline{\ttX}^{D} | \Omega_{\varepsilon} \right]\\
		&~~~ + s \sum_a\sum_{\alpha}  \left(u_a^2-\frac1{M}\right)v_\alpha^2\E[H_{a\alpha}^2] \E \left[R_{\alpha\alpha} \ttX^{D-1} \overline{\ttX}^{D} | \Omega_{\varepsilon} \right].
	\end{split} \eeq
	
	The first term satisfies that
	\beq \begin{split}
		&\left|\sum_a\sum_{\alpha}  \left(u_a^2-\frac1{M}\right)v_\alpha^2\E[H_{a\alpha}^2] \E \left[(R_{aa}-s) R_{\alpha\alpha} \ttX^{D-1} \overline{\ttX}^{D} | \Omega_{\varepsilon} \right] \right| \\
		&\leq C M \| \bsu \|_{\infty}^{2} N^{-1} N^{-\frac{1}{2}+\varepsilon} \E[|\ttX|^{2D-1} | \Omega_{\varepsilon}]\sum_\alpha v_\alpha^2 = C N^{-\frac{1}{2}+\varepsilon} \| \bsu \|_{\infty}^{2} \E[|\ttX|^{2D-1} | \Omega_{\varepsilon}]
	\end{split} \eeq
	for some constant $C$ since $\sum_\alpha v_\alpha^2=1.$ Using \eqref{eq:variance_estimate} and $\sum_{a}  \left(u_a^2-\frac1{M}\right)=0$, we also have
	\beq
	\begin{split}
		&\left|s \sum_a\sum_{\alpha}  \left(u_a^2-\frac1{M}\right)v_\alpha^2\E[H_{a\alpha}^2| \Omega_{\varepsilon}] \E \left[R_{\alpha\alpha} \ttX^{D-1} \overline{\ttX}^{D} | \Omega_{\varepsilon} \right]\right|
		\\&~~~\leq C\| \bsu \|_{\infty}^{2}\| \bsv \|_{\infty}^{2}|s| \sum_a\sum_{\alpha}  \left(u_a^2+\frac1{M}\right)v_\alpha^2 \E \left[|R_{\alpha\alpha} \ttX^{D-1} \overline{\ttX}^{D}| | \Omega_{\varepsilon} \right]
		\\&~~~\leq C\| \bsu \|_{\infty}^{2}\| \bsv \|_{\infty}^{2}\E[|\ttX|^{2D-1} | \Omega_{\varepsilon}]
	\end{split}
	\eeq
	for some constant $C$ and large $D>1$.
	For the second term in \eqref{eq:Phi_expansion}, we also have
	\beq \label{eq:X_ex_2}
	\begin{split}
		&\left|\sum_a\sum_{\alpha}  \left(u_a^2-\frac1{M}\right)v_\alpha^2\E[H_{a\alpha}^2| \Omega_{\varepsilon}]\E \left[R^2_{a\alpha} \ttX^{D-1} \overline{\ttX}^{D} | \Omega_{\varepsilon} \right] \right| 
		\\&~~~\leq	CN^{-1}\| \bsu \|_{\infty}^{2}\left|\sum_a\sum_{\alpha}v_\alpha^2\E \left[R^2_{a\alpha} \ttX^{D-1} \overline{\ttX}^{D} | \Omega_{\varepsilon} \right] \right|
		\\&~~~\leq C N^{-1+2\varepsilon} \| \bsu \|_{\infty}^{2} \E[|\ttX|^{2D-1} | \Omega_{\varepsilon}]\sum_\alpha v_\alpha^2
		\\&~~~\leq C N^{-1+2\varepsilon} \| \bsu \|_{\infty}^{2} \E[|\ttX|^{2D-1} | \Omega_{\varepsilon}].
	\end{split}
	\eeq
	To estimate the third term and the fourth term in \eqref{eq:Phi_expansion}, we notice that on $\Omega_{\varepsilon}$
	\beq\begin{split}
		\left| \frac{\partial \ttX}{\partial H_{a\alpha}} \right| &=\left| -\sum_b\sum_{\beta}  \left(u_b^2-\frac1{M}\right)v_\beta^2H_{b\beta}[R_{ab}R_{\alpha\beta}+R_{b\alpha}R_{a\beta}]+\left(u_a^2-\frac1{M}\right)v_\alpha^2R_{a\alpha}\right|\\
		&\leq CN^{-\frac1{2}+3\varepsilon}\| \bsu \|_{\infty}^{2}\sum_\alpha v_\alpha^2+CN^{-\frac1{2}+\varepsilon}\| \bsu \|_{\infty}^{2}\| \bsv \|_{\infty}^{2}\leq CN^{-\frac1{2}+3\varepsilon}\| \bsu \|_{\infty}^{2}.
	\end{split}
	\eeq
	for some constant $C$. Similarly, we can observe that
	\beq\begin{split}
		\left| \frac{\partial^2 \ttX}{\partial H_{a\alpha}^2} \right| \leq CN^{-\frac1{2}+3\varepsilon}\| \bsu \|_{\infty}^{2}.
	\end{split}
	\eeq
	Thus, we also obtain that
	\beq \label{eq:X_ex_3}\begin{split}
		&\left|\sum_a\sum_{\alpha}  \left(u_a^2-\frac1{M}\right)v_\alpha^2\E[H_{a\alpha}^2| \Omega_{\varepsilon}]\E \left[ R_{a\alpha} \frac{\partial \ttX}{\partial H_{a\alpha}} \ttX^{D-2} \overline{\ttX}^{D} \big| \Omega_{\varepsilon} \right] \right|
		\\&~~~\leq C N^{-1+4\varepsilon} \| \bsu \|_{\infty}^{4} \E[|\ttX|^{2D-2} | \Omega_{\varepsilon}]
	\end{split}\eeq
	and
	\beq \label{eq:X_ex_4}\begin{split}
		&\left|\sum_a\sum_{\alpha}  \left(u_a^2-\frac1{M}\right)v_\alpha^2\E[H_{a\alpha}^2| \Omega_{\varepsilon}]\E \left[ R_{a\alpha} \frac{\partial \overline{\ttX}}{\partial H_{a\alpha}} \ttX^{D-1} \overline{\ttX}^{D-1} \big| \Omega_{\varepsilon} \right]\right| 
		\\&~~~\leq C N^{-1+4\varepsilon} \| \bsu \|_{\infty}^{4} \E[|\ttX|^{2D-2} | \Omega_{\varepsilon}].
	\end{split}\eeq
	Hence, from \eqref{eq:Phi_expansion}, \eqref{eq:X_ex_2}, \eqref{eq:X_ex_3}, and \eqref{eq:X_ex_4}, 
	\beq \begin{split} \label{eq:X_main}
		&\left| \sum_a\sum_{\alpha}  \left(u_a^2-\frac1{M}\right)v_\alpha^2 \E[H_{a\alpha}^2|\Omega_{\varepsilon}] \E[\Phi'(H_{a\alpha}) | \Omega_{\varepsilon}] \right| 
		\\&~~~\leq C N^{-\frac{1}{2}+\varepsilon} \| \bsu \|_{\infty}^{2} \E[|\ttX|^{2D-1} | \Omega_{\varepsilon}] + C N^{-1+4\varepsilon} \| \bsu \|_{\infty}^{4} \E[|\ttX|^{2D-2} | \Omega_{\varepsilon}] + \varepsilon_1.
	\end{split} \eeq
	It remains to estimate $|\varepsilon_1|$ in \eqref{eq:Stein_error}. Proceeding as before,
	\beq \begin{split} \label{eq:epsilon_1_1}
		&\sum_a\sum_{\alpha}  \left(u_a^2-\frac1{M}\right)v_\alpha^2\E \Big[ |H_{a\alpha}|^3\Phi''(H_{a\alpha}) \Big| \Omega_{\varepsilon} \Big] \\
		&\leq C N^{-1+4\varepsilon} \| \bsu \|_{\infty}^{2} \E[|\ttX|^{2D-1} | \Omega_{\varepsilon}] + C N^{-2+7\varepsilon} \| \bsu \|_{\infty}^{4} \E[|\ttX|^{2D-2} | \Omega_{\varepsilon}] 
		\\&~~~+ C N^{-2+10\varepsilon} \| \bsu \|_{\infty}^{6} \E[|\ttX|^{2D-3} | \Omega_{\varepsilon}].
	\end{split} \eeq
	
	Our last goal is to find the bound for the error term $\varepsilon_1.$ To handle $\Phi''(tH_{a\alpha})$, we want to compare $\Phi''(H_{a\alpha})$ and $\Phi''(tH_{a\alpha})$ for some $|t|<1$. Let $G^{A, t}$ be the resolvent of $A$ where $A_{a\alpha}$ is replaced by $tA_{a\alpha}$, and let $\ttX^t$ be defined as $\ttX$ in \eqref{eq:X_def} with the same replacement for $A_{a\alpha}$ and also $G^A$ is replaced by $G^{A, t}$. Correspondingly, we also consider the replacement $R^{t}$ of the linearization $R$ by substituting $tH_{a\alpha}$ into $H_{a\alpha}$ (also for $H_{\alpha a}$). Then,
	\beq \label{eq:G^A,t-G^A}
	R_{AB}^t-R_{AB}=(R^t(H-H^t)R)_{AB}=(1-t)R^t_{Aa}H_{a\alpha}R_{\alpha B}+(1-t)R^t_{A\alpha}H_{\alpha a}R_{aB}.
	\eeq
	and
	\beq
	\begin{split}
		\ttX^t - \ttX &= \sum_{b}\sum_{\beta}  \left(u_b^2-\frac1{M}\right)v_\beta^2 (H_{b\beta}^tR_{b\beta}^t- H_{b\beta}R_{b\beta})
		\\&=\sum_{b}\sum_{\beta}  \left(u_b^2-\frac1{M}\right)v_\beta^2 H_{b\beta}(R_{b\beta}^t- R_{b\beta})+(t-1)\left(u_a^2-\frac1{M}\right)v_\alpha^2 (H_{a\alpha}R_{a\alpha}^t)
		\\&=(1-t)\sum_{b}\sum_{\beta}  \left(u_b^2-\frac1{M}\right)v_\beta^2 H_{b\beta}R^t_{ba}H_{a\alpha}R_{\alpha\beta}
		\\&~~~+(1-t)\sum_{b}\sum_{\beta}  \left(u_b^2-\frac1{M}\right)v_\beta^2 H_{b\beta}R^t_{b\alpha}H_{\alpha a}R_{a\beta}+(t-1)\left(u_a^2-\frac1{M}\right)v_\alpha^2 H_{a\alpha}R_{a\alpha}^t.
	\end{split}
	\eeq
	
	Thus, on $\Omega_{\varepsilon}$,
	\beq \label{eq:X_t-X}
	|\ttX^t - \ttX| \leq C N^{-1+4\varepsilon} \| \bsu \|_{\infty}^{2}.
	\eeq
	Using the estimates \eqref{eq:G^A,t-G^A} and \eqref{eq:X_t-X}, on $\Omega_{\varepsilon}$, we obtain that
	\beq \label{eq:epsilon_1_2}
	|\Phi''(H_{a\alpha}) - \Phi''(tH_{a\alpha})| \leq C |\Phi''(H_{a\alpha})| + N^{-\frac{5}{2}+11\varepsilon} \| \bsu \|_{\infty}^{6} |\ttX|^{2D-4}
	\eeq
	uniformly on $t \in (-1, 1)$.
	
	Combining \eqref{eq:X_expansion} and \eqref{eq:X_main} with \eqref{eq:epsilon_1_1}, \eqref{eq:epsilon_1_2}, and \eqref{eq:Stein_error}, we finally get
	\beq \begin{split}
		\E[|\ttX|^{2D} | \Omega_{\varepsilon}] &\leq C N^{-\frac{1}{2}+\varepsilon} \| \bsu \|_{\infty}^{2} \E[|\ttX|^{2D-1} | \Omega_{\varepsilon}] + C N^{-1+4\varepsilon} \| \bsu \|_{\infty}^{4} \E[|\ttX|^{2D-2} | \Omega_{\varepsilon}] \\
		&\quad + C N^{-2+10\varepsilon} \| \bsu \|_{\infty}^{6} \E[|\ttX|^{2D-3} | \Omega_{\varepsilon}] + C N^{-3+14\varepsilon} \| \bsu \|_{\infty}^{8} \E[|\ttX|^{2D-4} | \Omega_{\varepsilon}].
	\end{split} \eeq
	This proves the desired lemma for $\ttX=\ttX_1$.
	
	For the cases $\ttX=\ttX_2$ or $\ttX=\ttX_3$, the proofs are almost the same with the following changes:
	
	\begin{itemize}
		\item For $\ttX=\ttX_2$, we change the role of $\bsU$ and $\bsV.$ In other words, we will use $\sum_a u_a(\ell)^2=1$ and $\sum_\alpha\left(v_\alpha(\ell)^2-\frac1{N}\right)=0.$ Then the recursive bound for $\E[|\ttX_2|^{2D}|\Omega_\varepsilon]$ obtained by putting $\bsv$ instead of $\bsu$ in the upper bound in \eqref{eq:recursive}.
		\item In the same way, we use $\sum_{a}u_a(\ell_1)u_a(\ell_2)=\delta_{\ell_1\ell_2}$ and $\sum_\alpha|v_\alpha(\ell_1)||v_\alpha(\ell_2)|\leq1$ instead of $\sum_a\left(u_a^2-\frac1{M}\right)=0$ and $\sum_a u_a(\ell)^2=1,$ respectively. We then obtain the exactly same recursive bound in \eqref{eq:recursive} for $\ttX_3.$
	\end{itemize}
\end{proof}

\subsection{Computation of the test statistic} \label{sec:compute}
In this section, we prove the second part of Theorem \ref{thm:CLT_rec} and also provide the details on the computation of the test statistic in Theorem \ref{thm:test}. By performing the same calculations as we will do in this section, we can obtain optimal functions for the other models, so we omit the details. (Refer to \cite{chung2019weak,jung2020weak,jung2021}.) Recall that
\beq
m_Y(f)|_{\bsH_1} - m_Y(f)|_{\bsH_0} =\sum_{s=1}^{k} \sum_{\ell=1}^{\infty}\left(\frac{\SNR_s}{\sqrt{\rat}}\right)^{\ell}\tau_\ell(\wt{f})
\eeq
and
\beq \begin{split}
	V_{Y}(f) = 2\sum_{\ell=2}^\infty \ell\tau_\ell(\wt{f})^2+(w_4-1)\tau_1(\wt{f})^2.
\end{split} \eeq
Assuming $w_2 > 0$ and $w_4 > 1$, from Cauchy's inequality and the identity $\log(1-\lambda) = -\sum_{\ell=1}^{\infty} \lambda^{\ell}/\ell$,
\beq\begin{split} \label{eq:Cauchy_ineq}
	\left|\frac{m_Y(f)|_{\bsH_1} - m_Y(f)|_{\bsH_0}}{\sqrt{V_Y(f)}}\right|^2 
	&\leq \sum_{p,q=1}^k\frac{\SNR_p\SNR_q}{\rat}\left( \frac{1}{w_4-1} - \frac{1}{2} \right)  - \frac{1}{2}\displaystyle \log\left(1-\frac{\SNR_p\SNR_q}{\rat}\right)
	\\& = \left|\frac{m(\Omega)-m(0)}{\sqrt{V_0}}\right|^2,
\end{split}\eeq
which proves the first part of the theorem. The equality in \eqref{eq:Cauchy_ineq} holds if and only if
\beq \label{eq:Cauchy_equal}
\frac{\sqrt{\rat}(w_4-1)\tau_1(\wt{f})}{\sum_s\SNR_s} = \frac{2\ell (\sqrt{\rat})^\ell\tau_\ell(\wt{f})}{\sum_{s} \SNR_s^{\ell}} \qquad (\ell = 2, 3, 4, \dots).
\eeq

We now find all functions $f$ that satisfy \eqref{eq:Cauchy_equal}. Letting $2C$ be the common value in \eqref{eq:Cauchy_equal},
\beq \label{eq:tau_ell}
\tau_1(\wt{f}) = \frac{2C }{\sqrt{\rat}(w_4-1)}\sum_s\SNR_s, \quad \tau_{\ell}(\wt{f}) = \frac{C}{\ell(\sqrt{\rat})^\ell}\sum_{s} \SNR_s^{\ell} \qquad (\ell = 2, 3, 4, \dots).
\eeq
We can expand $\wt{f}$ in terms of the Chebyshev polynomials as
\beq
\wt f(x) = \sum_{\ell=0}^{\infty} C_{\ell} T_{\ell} \left( \frac{x}{2} \right).
\eeq
The orthogonality relation of the Chebyshev polynomials implies that for $\ell \geq 1$
\beq
\tau_{\ell}(\wt f) = \frac{C_{\ell} }{\pi} \int_{-2}^2 T_{\ell} \left( \frac{x}{2} \right) T_{\ell} \left( \frac{x}{2} \right) \frac{\dd x}{\sqrt{4-x^2}} = \frac{C_{\ell} }{\pi} \int_{-1}^1 T_{\ell} \left( y \right) T_{\ell} \left( y \right) \frac{\dd y}{\sqrt{1-y^2}} = \frac{C_{\ell}}{2}.
\eeq
Thus, \eqref{eq:tau_ell} holds if and only if
\beq \begin{split} \label{eq:pre_optimized}
	\wt{f}(x) &= c_0 + 2C \sum_s\left( \frac{2\SNR_s}{\sqrt{\rat}(w_4-1)} T_1 \left( \frac{x}{2} \right) + \sum_{\ell=2}^{\infty} \frac{1}{\ell}\left(\frac{\SNR_s}{\sqrt{\rat}}\right)^\ell T_{\ell} \left( \frac{x}{2} \right) \right) \\
	&=  c_0 + 2C \sum_s\left( \frac{\SNR_s}{\sqrt{\rat}}\left(\frac{2}{w_4-1}-1\right) T_1 \left( \frac{x}{2} \right) + \sum_{\ell=1}^{\infty} \frac{1}{\ell}\left(\frac{\SNR_s}{\sqrt{\rat}}\right)^\ell T_{\ell} \left( \frac{x}{2} \right) \right)
\end{split} \eeq
for some constant $c_0$. We notice that the following identity holds for the Chebyshev polynomials:
\beq
\sum_{\ell=1}^{\infty} \frac{t^{\ell}}{\ell} T_{\ell} \left( x \right) = \log \left( \frac{1}{\sqrt{1-2tx+t^2}} \right).
\eeq
(See, e.g., (18.12.9) of \cite{Handbook}.) Since $T_1(x) = x$, we find that \eqref{eq:pre_optimized} is equivalent to
\beq \begin{split}\label{eq:optimized_phi}
	\wt{f}(x) &= c_0 + C \sum_s\left[\frac{\SNR_s}{\sqrt{\rat}}\left(\frac{2}{w_4-1}-1\right)x  - \log\left(\frac{\rat-\SNR_s\sqrt{\rat}x+\SNR_s^2}{\rat}\right)\right],
\end{split}\eeq
or
\beq\begin{split}
	f(x) &=c_0 + C\sum_s\left[\frac{\SNR_s}{\rat}\left(\frac{2}{w_4-1}-1\right)x  -\frac{ \SNR_s(1+\rat)}{\rat}\left(\frac{2}{w_4-1}-1\right)\right]
	\\&~~~- C\sum_s\log\left[\frac{\SNR_s}{\rat}\left(\left(1+\frac{\rat}{\SNR_s}\right)(1+\SNR_s)-x\right)\right].
\end{split}\eeq
This concludes the proof of Theorem \ref{thm:CLT} with an optimal function
\beq\label{eq:phi}
\phi_{\Omega}(x)=\wt \phi_\Omega(\varphi(x))
\eeq 
where
\beq
\wt \phi_\Omega(x)=c_0 +  \sum_s\left[\frac{\SNR_s}{\sqrt{\rat}}\left(\frac{2}{w_4-1}-1\right)x  - \log\left(\frac{\rat-\SNR_s\sqrt{\rat}x+\SNR_s^2}{\rat}\right)\right].
\eeq
Choosing
\[
c_0 = \sum_s\left[\frac{(1+\rat)}{\rat} \left(\frac{2}{w_4-1}-1\right)\SNR_s + \log (\SNR_s/\rat)\right],
\]
we get \eqref{eq:optimal_f1}. Further, we can see that 
\beq\label{eq:optimal_general}
\phi_{\Omega}(x)=\sum_s\phi_{\SNR_s}(x).
\eeq

From this, we directly obtain that $L_\Omega=\sum_sL_{\SNR_s}$, 
\beq
m_Y(\phi_\SNR)|_{\bsH_0} = -\frac1{2} \sum_{s}\log\left(1-\frac{\SNR_s^2}{\rat}\right)  +\frac{1}{2\rat}(w_4-3)\sum_s\SNR_s^2,
\eeq
\beq
m_Y(\phi_\SNR)|_{\bsH_1} = m_Y(\phi_\SNR)|_{\bsH_0}  +\sum_{p,q}\left[-\log\left(1-\frac{\SNR_p\SNR_q}{\rat}\right) +\frac{\SNR_p\SNR_q}{\rat}\left(\frac2{w_4-1}-1\right)\right]
\eeq
and
\beq
V_{Y}(\phi_\SNR)|_{\bsH_1} = V_Y(\phi_\SNR)|_{\bsH_0} = 2\sum_{p,q}\left[-\log\left(1-\frac{\SNR_p\SNR_q}{\rat}\right) +\frac{\SNR_p\SNR_q}{\rat}\left(\frac2{w_4-1}-1\right)\right].
\eeq

\end{document}